\documentclass[11pt]{article}

\usepackage[utf8]{inputenc}
\usepackage[T1]{fontenc}
\usepackage[margin=1in]{geometry}
\usepackage{amsmath,amsfonts,amssymb,amsthm}
\usepackage{bm}
\usepackage{microtype}
\usepackage{graphicx}
\usepackage{float}
\usepackage[table]{xcolor}
\usepackage[para,online,flushleft]{threeparttable}
\usepackage{booktabs}
\usepackage{enumitem}
\usepackage{algorithm}
\usepackage{algorithmic}
\usepackage{placeins}
\usepackage{bbding}
\usepackage{makecell}
\usepackage{hhline}
\usepackage{natbib}
\usepackage{hyperref}

\hypersetup{
  hidelinks,
  pdfauthor={Chengchang Liu and Luo Luo},
  pdftitle={Randomized Quasi-Gauss--Newton Methods for Solving General Nonlinear Equations},
  pdfsubject={Quasi-Newton methods for nonlinear equations},
  pdfkeywords={quasi-Newton methods, nonlinear equations, nonlinear least squares, randomized algorithms, local convergence}
}

\theoremstyle{plain}
\newtheorem{theorem}{Theorem}
\newtheorem{lemma}[theorem]{Lemma}
\newtheorem{proposition}[theorem]{Proposition}
\newtheorem{corollary}[theorem]{Corollary}
\newtheorem{assumption}[theorem]{Assumption}
\theoremstyle{definition}
\newtheorem{definition}[theorem]{Definition}
\theoremstyle{remark}
\newtheorem{remark}[theorem]{Remark}
\theoremstyle{plain}

\newenvironment{keywords}
  {\par\medskip\noindent\small\textbf{Keywords:} }
  {\par\medskip}

\def\1{\bf{1}}

\newcommand{\Norm}[1]{\left\| #1 \right\|}
\newcommand{\norm}[1]{\left\| #1 \right\|_2}

\def\vzero{{\bf{0}}}

\def\va{{\bf{a}}}

\def\ve{{\bf{e}}}

\def\vu{{\bf{u}}}
\def\vv{{\bf{v}}}
\def\vw{{\bf{w}}}
\def\vx{{\bf{x}}}
\def\vy{{\bf{y}}}

\def\fE{{\mathcal{E}}}

\def\fN{{\mathcal{N}}}

\def\BE{{\mathbb{E}}}

\def\BR{{\mathbb{R}}}

\def\mA {{\bf A}}

\def\mF {{\bf F}}
\def\mG {{\bf G}}
\def\mH {{\bf H}}
\def\mI {{\bf I}}
\def\mJ {{\bf J}}

\def\mQ {{\bf Q}}
\def\mR {{\bf R}}

\newcommand{\R}{\mathbb{R}}

\makeatletter
\def\Ddots{\mathinner{\mkern1mu\raise\p@
\vbox{\kern7\p@\hbox{.}}\mkern2mu
\raise4\p@\hbox{.}\mkern2mu\raise7\p@\hbox{.}\mkern1mu}}
\makeatother

\makeatletter
\newcommand*{\rom}[1]{\expandafter\@slowromancap\romannumeral #1@}
\makeatother

\def\vx {{{\bf x}}}
\def\vw {{{\bf w}}}
\def\vy {{{\bf y}}}
\def\va {{{\bf a}}}
\def\BR{{\mathbb{R}}}

\def\A{{\bf A}}
\def\B{{\bf B}}

\def\F{{\bf F}}

\def\G{{\bf G}}
\def\H{{\bf H}}
\def\I{{\bf I}}
\def \J{{\bf J}}

\def\Q{{\bf Q}}

\def\R{{\bf R}}

\def\S{{\bf S}}

\def\u{{\bf u}}

\def\v{{\bf v}}

\def\w{{\bf w}}

\def\x{{\bf x}}

\def\y{{\bf y}}

\def\0{{\bf 0}}
\def\1{{\bf 1}}

\def\OM{{\mathcal O}}

\def\RB{{\mathbb R}}
\def\EB{{\mathbb E}}

\def \EBP #1{\EB\left[#1\right]}
\def \tr #1{\text{\rm tr}\left(#1\right)}

\def \srk{{\text{\rm SR1}}}

\def \bfgs{{\text{\rm BFGS}}}
\def \fbfgs{{\text{\rm F-BFGS}}}

\def \L{{\bf L}}

\begin{document}

\title{Randomized Quasi-Gauss--Newton Methods for Solving General Nonlinear Equations}
\author{Chengchang Liu \and Luo Luo}
\date{}
\maketitle
\begin{abstract}
This paper considers the local convergence for solving general nonlinear equations. 
We establish randomized quasi-Gauss--Newton methods based on the approximation of the Gram matrix, addressing both the underdetermined and the overdetermined settings.
For the underdetermined case, we show that our methods achieve local superlinear convergence to the optimal solution.
For the overdetermined case, we show that our methods achieve local condition-number-free convergence to the stationary point of the nonlinear least-square formulation of the problem.
In contrast, existing quasi-Newton methods mostly focus on square systems and additionally require the initial Jacobian estimate to be sufficiently close to the exact one.
\end{abstract}

\begin{keywords}
quasi-Newton methods, nonlinear equations, nonlinear least squares, randomized algorithms, explicit superlinear convergence
\end{keywords}

\section{Introduction}
We study the problem of solving the nonlinear equation
\begin{align}\label{eq:nonlinear_obj}
    \F(\x) = \0,
\end{align}
where $\F(\cdot)\triangleq [F_1(\cdot);\cdots;F_n(\cdot)]^{\top}$ is a general mapping such that $\F:\RB^{d}\to\RB^n$ with differentiable $F_i:\RB^d\to\RB$ for all $i\in[n]$. 
We also consider its nonlinear least square formulation
\begin{align}
    \label{eq:nonlinear_least_obj}
    \min_{\x\in\RB^d} \phi(\x)\triangleq\frac{1}{2}\|\F(\x)\|_2^2.
\end{align}
Numerous applications in machine learning~\citep{scieur2018nonlinear,alizamir2020comparative,song2021accelerating}, game theory~\citep{frehse1984nonlinear,nourian2013e}, economics~\citep{albu2006non}, and control systems~\citep{more1989collection,berthier2021fast} can be formulated as above problems, which have been extensively studied for decades and be considered one of the most important in scientific computing~\citep{nesterov2007modified}.

The Gauss--Newton method is very popular for solving nonlinear equations system, and it iterates according to
\begin{align}  
\label{update:gn}
    \x_{+} = \x-\J(\x)^{\dag}\F(\x),
\end{align}
where we use $\J(\x)\in\BR^{n\times d}$ to denote the Jacobian of $\F(\cdot)$ at $\vx$ and the notation $(\,\cdot\,)^{\dag}$ presents the pseudoinverse.
We are interested in the case that the Jacobian in the local region of the solution is full rank, which implies 
\begin{align}\label{update:gnJ}
\J(\x)^{\dag}=\begin{cases}
    \J(\x)^{\top}(\J(\x)\J(\x)^{\top})^{-1}, &\text{if }n<d,\\
	 (\J(\x)^{\top}\J(\x))^{-1}\J(\x)^{\top}, &\text{if }n\geq d.    
\end{cases}
\end{align}
It is well-known that the Gauss--Newton method enjoys faster local convergence rates than simply applying gradient descent to solve problem (\ref{eq:nonlinear_least_obj})~\citep{gratton2007approximate,ortega2000iterative,nesterov2018lectures}. 
However, accessing the exact Gram matrix $\J(\x)^{\top}\J(\x)$ 
or~$\J(\x)\J(\x)^{\top}$ and 
computing its inverse require the computational cost of 
$\OM(\min\{n,d\}nd)$, which is impractical for large-scale problems.
Therefore, we desire to reduce the cost of each iteration while maintaining a good convergence behavior.

Quasi-Newton methods are popular to solve nonlinear equations~\citep{yuan2011recent,yuan2022sketched,yuan2009subspace,broyden1967quasi,martinez1991quasi,broyden1967quasi,davidon1991variable,broyden1970convergence,broyden1970convergence2,davidon1991variable,fletcher1963rapidly}. 
They use low-rank modification to approximate the Jacobian or its inverse, with relatively low computational costs, and often enjoy fast local convergence rates.
However, classical quasi-Newton methods such as Broyden's methods \citep{broyden2000discovery}, 
SR1 method~\citep{broyden1967quasi,davidon1991variable}, BFGS method~\citep{broyden1970convergence,broyden1970convergence2}, and DFP method~\citep{davidon1991variable,fletcher1963rapidly} only focus on the square systems. 
The non-asymptotic local convergence of these methods has been established in recent years \citep{rodomanov2021greedy,rodomanov2021new,rodomanov2021rates,ye2022towards,liu2022quasi,lin2022explicit,liu2023symmetric,liu2023block}, while exiting results cannot be applied to more general settings.

There are also several works that consider general nonlinear equations. 
Specifically, \citet{martinez1991quasi,vater2024convergence} proposed quasi-Newton methods for the underdetermined systems (i.e., the case of $n<d$) by generalizing Broyden's updates to approximate the pseudoinverse of the Jacobian, but they only achieve the asymptotic local convergence.
In addition, their convergence guarantees require the assumption that the initial estimator of the Jacobian (or its inverse) to be sufficiently close to the exact one.
Although this assumption is widely adopted in the analysis~\citep{asl2024aj,martinez1991quasi,vater2024convergence,liu2023block}, 
achieving a sufficiently accurate initial estimator may be so expensive.
For the overdetermined case (i.e., the case of $n\geq d$), the nonlinear equation~\eqref{eq:nonlinear_obj} may has no solution so that we typically consider its nonlinear least square formulation~\eqref{eq:nonlinear_least_obj}.
For example, \citet{yabe1991factorized,eriksson1999quasi,zhou2010global} adopt Broyden family updates to approximate the Hessian for the objective of problem \eqref{eq:nonlinear_least_obj}.
However, these methods still require accessing the exact Gram matrix at the early stage of the iterations, leading to expensive per-iteration cost like standard Gauss--Newton methods.

We summarize the main limitations of the aforementioned quasi-Newton methods for nonlinear systems as follows.
\begin{itemize}[leftmargin=*,label={}]
\item \emph{Limitation 1.} Existing non-asymptotic superlinear convergence guarantees for quasi-Newton methods are limited to squared systems (i.e., the case of $n=d$), whereas how to establish the superlinear convergence for the general case is unclear.
\item \emph{Limitation 2.} Existing quasi-Newton methods for general nonlinear systems require either a sufficiently accurate initial Jacobian estimator or accessing the exact Gram matrix, resulting quite expensive computational cost.
\end{itemize}

\begin{table}[t]
  \caption{We summarize the local convergence results and initial conditions of quasi-Newton methods for solving nonlinear equations, where the dash indicates that the algorithm has no convergence guarantee under the corresponding setting.}
  \label{tbl:compare}\vskip-0.1cm
    \resizebox{\linewidth}{!}{
    \centering
    \begin{tabular}{cccccc} \hline\
    Method  & $n=d$ & $n<d$ & $n>d$ & Non-Asymptotic & Initial Condition \\ \hline\hline 
    \makecell{Broyden \\ 
         \footnotesize \citep{lin2021explicit} \\
         \footnotesize \citep{ye2021greedy} \\ 
         \footnotesize \citep{liu2023block}} & 
    superlinear & -- & -- & \Checkmark & $\G_0\approx\J(\x_0)$ \\\hline 
    \makecell{Generalized Broyden \\
         \footnotesize \citep{martinez1991quasi} \\
         \footnotesize \citep{vater2024convergence}
         } & 
    superlinear & superlinear & -- &  \XSolidBrush & $\G_0\approx\J(\x_0)$ \\\hline 
    \makecell{GN-BFGS \\
         \footnotesize \citep{gu2002descent} \\
         \footnotesize \citep{li1999globally}
         } &
    superlinear & -- & -- & \XSolidBrush & $\mG_0=c\mI_m$ \\ \hline
    \makecell{RQGN \\
         \footnotesize Algorithm \ref{alg:RQGN}} &
         ~superlinear~ & ~superlinear~ & ~linear~ & \Checkmark & $\mG_0=c\mI_m$ \\ \hline
    \end{tabular}} \vskip-0.1cm
\end{table}

In this paper, we develop randomized quasi-Gauss–Newton (RQGN) methods to address both the above two limitations.
In contrast to existing methods that directly approximate the (pseudo) inverse of the Jacobian~\citep{lin2021explicit,ye2021greedy,liu2023block,martinez1991quasi,vater2024convergence},
we approximate the Gram matrix in the Gauss--Newton iteration by Broyden family updates with randomized direction,
which leads to the non-asymptotic local convergence guarantees for both overdetermined and underdetermined cases.
We compare our results with existing quasi-Newton methods in Table~\ref{tbl:compare} and summarize our main theoretical contributions as follows.
\begin{itemize}[topsep=0.08cm,itemsep=0.08cm]
\item Our proposed RQGN methods work for both underdetermined and overdetermined nonlinear equations, and they requires neither the exact Gram matrix nor a sufficiently accurate initial Jacobian approximation so its Gram estimator can be initialized by a scaled identity matrix.
\item For the underdetermined case (i.e., $n<d$), we show that our RQGN enjoy explicit non-asymptotic local superlinear convergence rates without assuming uniqueness of the solution.
\item For the overdetermined case (i.e., $n\geq d$), we establish explicit local condition-number-free linear convergence rates for finding stationary points of the nonlinear least square formulation \eqref{eq:nonlinear_least_obj}, and our results hold even if the system \eqref{eq:nonlinear_obj} has no solution.
\end{itemize}

\paragraph{Paper Organization}
In Section~\ref{sec:pre}, we formalize the notations and assumptions throughout this paper and introduce the background of quasi-Newton updates. 
In Section~\ref{sec:rqn}, we propose our randomized quasi-Gauss--Newton (RQGN) method for solving the general nonlinear equations.
In Sections~\ref{sec:underde} and~\ref{sec:overfit}, we establish the non-asymptotic local convergence guarantees of our RQGN for both underdetermined and overdetermined cases, respectively.
In Section~\ref{sec:exp}, we conduct numerical experiments to show the superiority of the proposed methods.
We conclude our work in Section \ref{sec:conclu}.


\section{Notation and Preliminaries} 
\label{sec:pre}

We first introduce the notations used in this paper. We use $\Norm{\,\cdot\,}$ to present the spectral norm and the Euclidean norm of the matrix and the vector, respectively. 
Given a matrix $\S$, we write $\S^{\dag}$ for its Moore--Penrose inverse and $\sigma_{i}(\S)$ for its $i$th largest singular value; when $\S$ is square, 
we use~$\tr{\S}$ to denote its trace.
For two positive definite matrices $\G$ and $\H$ with the same size, we write
$\langle\G,\H\rangle\triangleq\text{tr}(\G\H)$ for their inner product.
Additionally, we denote the Jacobian of the mapping $\F:\BR^d\to\BR^n$ as $\J:\BR^d\to\BR^{n\times d}$ and define the Gram matrix~as
\begin{align}\label{eq:H-Gram}
    \H(\x)\triangleq
\begin{cases}
\J(\x)\J(\x)^{\top}, & \text{if}~~n<d, \\    
\J(\x)^{\top}\J(\x), & \text{if}~~n\geq d.
\end{cases}
\end{align}
For the ease of presentation, we let $m\triangleq\min\{n,d\}$ so that we can write $\H(\x)\in\RB^{m\times m}$.
Accordingly, we denote the standard basis for the Euclidean space $\BR^m$ by $\{\ve_1,\dots,\ve_m\}$ and let $\I_{m}$ be the identity matrix with the size of $m\times m$.

We then formalized the assumptions of our problems (\ref{eq:nonlinear_obj}) and (\ref{eq:nonlinear_least_obj}). 
We assume the nonlinear mapping $\F:\BR^d\to\BR^n$ 
has the bounded and Lipschitz continuous Jacobian $\mJ:\BR^d\to\BR^{n\times d}$.
\begin{assumption}
\label{ass:Jlip}
    We assume that there exist constants $L>0$ and $L_2>0$ such that
    \begin{align}
    \label{eq:lip}
    \|\J(\x)\|\leq L\qquad\text{and}\qquad\| \J(\x)-\J(\y)\|\leq L_2\|\x-\y\| 
    \end{align}
    for all $\vx,\vy\in\BR^d$. 
\end{assumption}
We also assume that the Jacobian Jacobian $\mJ:\BR^d\to\BR^{n\times d}$ is always nonsingular for some non-empty set.
\begin{assumption}
\label{ass:nonsingular}
We assume that there exist some non-empty set $\Omega\subseteq\RB^d$ 
and a constant $\mu>0$ such that $\sigma_m(\J(\x))\geq \mu$ holds for all $\x\in\Omega$, where $m=\min\{n,d\}$.
\end{assumption}

The above assumptions indicate the following properties for the nonlinear mapping $\F(\cdot)$ and the Gram matrix $\H(\cdot)$.
\begin{proposition}
\label{prop:gram_lip}
    Under Assumption~\ref{ass:Jlip}, the nonlinear mapping $\F(\cdot)$ is $L$-Lipschitz continuous and the Gram matrix $\H(\cdot)$ is $2LL_2$-Lipschitz continuous with $\H(\x)\preceq L^2\I_m$ for all $\x\in\RB^{d}$.
\end{proposition}
\begin{proposition}
\label{prop:gram_lower}
    Under Assumption~\ref{ass:nonsingular}, it holds  $\H(\x)\succeq \mu^2\I_m$ for all $\x\in\Omega$.
\end{proposition}
The Gram matrix $\mH(\cdot)$ also holds the following partial relation.
\begin{proposition}
\label{prop:Gram_concor}
    Under Assumptions \ref{ass:Jlip} and \ref{ass:nonsingular}, it holds 
    \begin{align}
    \label{eq:Gram_bound_0}
        \frac{\H(\x)}{1+M\|\y-\x\|}\preceq\H(\y) \preceq (1+M\|\y-\x\|)\H(\x).
    \end{align}
    for all $\x,\y\in\Omega$, where  $M\triangleq2LL_2/\mu^2$.    
\end{proposition}

We now introduce some popular quasi-Newton updates to approximate the target matrix~$\H\in\RB^{m\times m}$ by 
the current estimator $\G\in\RB^{m\times m}$ and the update direction $\u\in\RB^{m}$ in the following definition.

\begin{definition}[{\citet{lin2022explicit,rodomanov2021greedy}}]
\label{def:qnupdate}
Given two positive-definite matrices $\G,\H\in\RB^{m\times m}$ and vector $\u\in\RB^m$, we define the following quasi-Newton updates.
\begin{itemize}
\item The BFGS update is defined by
\begin{align*}
\bfgs(\G,\H,\u)
&\triangleq \G-\frac{\G\u\u^{\top}\G}{\u^{\top}\G\u}
 +\frac{\H\u\u^{\top}\H}{\u^{\top}\H\u}.
\end{align*}
\item The fast BFGS update is defined by
\begin{align*}
\fbfgs(\G,\H,\u)
&\triangleq\bfgs\big(\G,\H,\L^{\top}\u\big),
\end{align*}
where the matrix $\L\in\RB^{m\times m}$ satisfies $\G^{-1}=\L^{\top}\L$.
\item The SR1 update is defined by
\begin{align*}
\srk(\G,\H,\u)
&\triangleq\G-\frac{(\G-\H)\u\u^{\top}(\G-\H)}
{\u^{\top}(\G-\H)\u}.
\end{align*}
\end{itemize}
\end{definition}
\begin{remark}
The fast BFGS update requires to maintain $\L_t$ such that $\G_t^{-1}=\L_t^{\top}\L_t$, which can be implemented by a closed form update with the complexity of $\OM(m^2)$ flops per iteration.
Please refer to Algorithm~\ref{alg:fbfgs} in Appendix~\ref{app:fbfgs}.
\end{remark}

We describe the convergence of above quasi-Newton updates by introducing the quantities
\begin{align*}
    \sigma_{\H}(\G) \triangleq \tr{\H^{-1}(\G-\H)}\qquad\text{and}\qquad\tau_{\H}(\G)\triangleq\tr{\G-\H}
\end{align*}
to describe the difference between positive-definite matrices $\G\in\RB^{m\times m}$ and $\H\in\RB^{m\times m}$. 
By taking random vector $\u\in\RB^{m}$ according to Gaussian distribution, it holds the following in-expectation convergence guarantees \citep{lin2022explicit}.

\begin{lemma}[{\citet[Theorems~6, 13, and~14]{lin2022explicit}}]
\label{lm:QN_one_iter}
    Let $\H,\G\in\RB^{m\times m}$ be positive definite with $\H\preceq\G$. We suppose $\mu^2\I_m\preceq\H\preceq L^2\I_m$ for some constants $L,\mu>0$ such that~$L\geq\mu$. By taking $\vu\sim\fN(\vzero,\mI_m)$, the quasi-Newton updates has the following properties.
\begin{itemize}
\item For $\G_{+}=\bfgs(\G,\H,\u)$, it holds  $\EBP{\sigma_{\H}(\G_{+})}\leq(1-\mu^2/(mL^2))\sigma_{\H}(\G)$;
\item For $\G_{+}=\fbfgs(\G,\H,\u)$, it holds  $\EBP{\sigma_{\H}(\G_{+})}\leq(1-1/m)\sigma_{\H}(\G)$;
\item For $\G_{+}=\srk(\G,\H,\u)$, it holds  $\EBP{\tau_{\H}(\G_{+})}\leq(1-1/m)\tau_{\H}(\G)$.
\end{itemize}
\end{lemma}

\begin{lemma}[{\citet[Lemma~5]{lin2022explicit}}]
\label{lm:QN_order}
    Let $\H,\G\in\RB^{m\times m}$ be positive definite with $\H\preceq\G\preceq\eta\H$ for some $\eta\geq1$. If the matrix $\G_{+}$ is obtained by either BFGS update, fast BFGS update, or SR1 update by following Definition~\ref{def:qnupdate}, it holds $\H\preceq\G_{+}\preceq\eta\H$.
\end{lemma}


\begin{algorithm}[t]
\caption{Randomized Quasi-Gauss--Newton Methods (RQGN)}\label{alg:RQGN}
\begin{algorithmic}[1]
\STATE \textbf{Input:} $\x_0\in\Omega$, $\G_0\succeq\H_0\in\RB^{m\times m}$, $M\geq0$ \label{line:input} \\[0.15cm]
\STATE \textbf{for} $t=0,1,\ldots$ \\[0.15cm]
 		\STATE \quad
 		    $\x_{t+1}=\begin{cases}
               \x_t-\J(\x_t)^{\top}\G_t^{-1}\F(\x_t), &\text{if }n<d,\\
	 	    \x_t-\G_t^{-1}\J(\x_t)^{\top}\F(\x_t), &\text{if }n\geq d
	 	    \end{cases}$ \label{eq:line-update-x} \\[0.15cm]
	 	\STATE \quad $r_t=\|\x_{t+1}-\x_{t}\|$ \\[0.15cm]
 		\STATE \quad $\widetilde{\G}_{t}=(1+Mr_t)\G_t$ \\[0.15cm]
        \STATE \quad $\u_t\sim\fN(\0,\I_m)$ \\[0.15cm]
 		\STATE \quad [Option I]~~ $\G_{t+1}=\bfgs(\widetilde{\G}_t,\H(\x_{t+1}),\u_t)$ \label{eq:line-update-BFGS} \\[0.15cm]
 		\STATE \quad [Option II]~ $\G_{t+1}=\fbfgs(\widetilde{\G}_t,\H(\x_{t+1}),\u_t)$ \label{eq:line-update-F-BFGS} \\[0.15cm]
 		\STATE \quad [Option III] $\G_{t+1}=\srk(\widetilde{\G}_t,\H(\x_{t+1}),\u_t)$ \label{eq:line-update-SR1} \\[0.15cm]
\STATE \textbf{end for}
\end{algorithmic}
\end{algorithm}

\section{Randomized Quasi-Gauss--Newton Methods}
\label{sec:rqn}
In this section, we introduce randomized quasi-Gauss--Newton (RQGN) methods for solving general nonlinear equations.
The key idea of our algorithm design is constructing the square matrix~$\G$ to approximate the Gram matrix $\mJ(\vx)\mJ(\vx)^\top$ or $\mJ(\vx)^\top\mJ(\vx)$ in Gauss--Newton formulations~\eqref{update:gn} and \eqref{update:gnJ}, which leads to the iteration scheme
\begin{align}
\label{eq:qNupdate}
    \x_{+}=\begin{cases}
      \x-\J(\x)^{\top}\G^{-1}\F(\x), &\text{if }n<d,\\
      \x-\G^{-1}\J(\x)^{\top}\F(\x), &\text{if }n\geq d.
    \end{cases}
\end{align}
Furthermore, we apply quasi-Newton updates introduced in Definition \ref{def:qnupdate} to construct the square matrix~$\mG^+\in\BR^{m\times m}$ to approximate the Gram matrix at next point $\vx_+\in\BR^m$ by given the current estimator $\G\in\BR^{m\times m}$, i.e.,
\begin{align}
\label{eq:update_G}
\G_{+}
=\begin{cases}
\bfgs(\widetilde{\G},\H(\x_{+}),\u), & \text{Option I},\\
\fbfgs(\widetilde{\G},\H(\x_{+}),\u), & \text{Option II},\\
\srk(\widetilde{\G},\H(\x_{+}),\u), & \text{Option III},
\end{cases}
\end{align}
where $\widetilde{\G}\triangleq (1+M\|\x_{+}-\x\|)\G$, $\u\sim\fN(\0,\I_m)$, $M>0$, and $\mH(\cdot)$ follows equation \eqref{eq:H-Gram}.
Combining the iteration scheme~\eqref{eq:qNupdate} and the construction~\eqref{eq:update_G}, we formally present our randomized quasi-Gauss--Newton methods in Algorithm~\ref{alg:RQGN}.

Note that the initialization of RQGN (Line \ref{line:input} of Algorithm \ref{alg:RQGN}) only requires that the initial Gram estimator satisfies $\G_0\succeq\H_0$. Therefore, Proposition \ref{prop:gram_lip} indicates that we can simply set $\mG_0=c\mI_m$ for some $c\geq L^2$ and maintain the square matrix $\mG_t$ (or its inverse) with the complexity of $\OM(m^2)$ flops per iteration, since the quasi-Newton updates introduced in Definition \ref{def:qnupdate} only involves rank-one or rank-two modification.
Furthermore, we emphasize that the updates of $\vx_{t+1}$ and $\mG_{t+1}$ (Lines \ref{eq:line-update-x} and \ref{eq:line-update-BFGS}--\ref{eq:line-update-SR1}) can be efficiently implemented by accessing the Jacobian-vector product so that it is unnecessary to explicitly compute the Jacobian $\mJ(\cdot)$.

Additionally, our RQGN methods with appropriate scaled parameter $M$ can guarantee the partial relation 
\begin{align}
\label{eq:scale_G00}
    {\tilde \G}_t \succeq \H(\x_{t+1})
\end{align}
holds for all $t\geq 0$ (see Appendix \ref{app:proof_lemma_gram_approximate}), which is crucial to our later convergence analysis .
In contrast, simply applying quasi-Newton updates on matrix $\mG_t$ without scaling cannot achieve~\eqref{eq:scale_G00} for general nonlinear equations.

\section{The Convergence Analysis}\label{sec:local_convergence}

We analyze the local convergence behavior of the randomized quasi-Gauss--Newton method (Algorithm~\ref{alg:RQGN}) for both the underdetermined case and the overdetermined cases in Section~\ref{sec:underde} and~\ref{sec:overfit}, respectively.

For the ease of presentation, we introduce the notations
\begin{align*}
r_t\triangleq \|\x_{t+1}-\x_t\| 
\qquad\text{and}\qquad \H_t\triangleq\H(\x_t).    
\end{align*}
In addition, the initial condition $\x_0\in\Omega$ of Algorithm~\ref{alg:RQGN} and Assumption \ref{ass:nonsingular} indicate $\mH_0\succ \vzero$, so that we can define $\eta_0\in\BR$ as the smallest scalar satisfying $\G_0\preceq\eta_0\H_0$, i.e.,
\begin{align}\label{eq:dfn-eta0}
    \eta_0\triangleq \min_{\{\eta: 
 \H_0\preceq\G_0\preceq \eta\H_0\}}\eta.
\end{align}
Furthermore, we let $\mathcal F_t\triangleq\sigma(\u_0,\ldots,\u_{t-1})$ be the 
sigma filed generated by the random variable sequence $\{\vu_0,\dots,\vu_{t-1}\}$ and define $\mathbb E_t[\,\cdot\,]\triangleq\mathbb E[\,\cdot\mid\mathcal F_t]$.

For our later analysis, we first incorporate Lemmas~\ref{lm:QN_one_iter} and~\ref{lm:QN_order} and Proposition~\ref{prop:Gram_concor} to show how well the square matrix~$\G_{t}$ approximates the Gram matrix $\H_t$ in the following lemma, which extends Lemma 25 of \citet{lin2022explicit}.

\begin{lemma}
    \label{lm:Gram_approximate}
    Under Assumptions \ref{ass:Jlip} and \ref{ass:nonsingular}, we suppose points $\{\vx_t\}$ generated by Algorithm~\ref{alg:RQGN} holds $\x_t\in\Omega$ for all $t$. Then there exists non-negative randomized sequence $\{\delta_t\}$ such that 
    \begin{align}
    \label{eq:delta-recur}
        \H_t\preceq\G_t\preceq(1+\delta_t)\H_t
        \qquad\text{and}\qquad
        \mathbb E_t[\delta_{t+1}]\leq (1-\gamma)(1+Mr_t)^2(\delta_t+2\tilde{c}mMr_t)
    \end{align}
    for all $t$, where $\delta_t$, $\gamma$, and $\tilde c$ are determined as follows.
    \begin{itemize}
    \item For Algorithm~\ref{alg:RQGN} with the BFGS update (Option I), the result \eqref{eq:delta-recur} holds by taking
    \begin{align*}
    \delta_t=\sigma_{\H_t}(\G_t),\quad
    \gamma=\frac{\mu^2}{mL^2},\quad
    \tilde c=1,
    \quad\text{and}\quad\delta_0\leq m(\eta_0-1).
    \end{align*}
    \item For Algorithm~\ref{alg:RQGN} with the fast BFGS update (Option II), the result \eqref{eq:delta-recur} holds by taking
    \begin{align*}
    \delta_t=\sigma_{\H_t}(\G_t),\quad
    \gamma=\frac{1}{m},\quad
    \tilde c=1,
    \quad\text{and}\quad\delta_0\leq m(\eta_0-1).
    \end{align*}
    \item For Algorithm~\ref{alg:RQGN} with the SR1 update (Option III), the result \eqref{eq:delta-recur} holds by taking
    \begin{align*}
    \delta_t=\frac{mL^2\tau_{\H_t}(\G_t)}
    {\mu^2\tr{\H_t}},\quad
    \gamma=\frac{1}{m},\quad
    \tilde c=\frac{L^2}{\mu^2},
    \quad\text{and}\quad
    \delta_0\leq \frac{mL^2(\eta_0-1)}{\mu^2}.
    \end{align*}
    \end{itemize}
\end{lemma}

\begin{remark}
The results of the above Lemma \ref{lm:Gram_approximate} require the condition $\x_t\in \Omega$ holds for all $t$, 
which can be guaranteed by taking appropriate $\Omega$ for both underdetermined and overdetermined settings.
Please refer to Theorems \ref{thm:linear_under} and \ref{thm:linear1n>d}.    
\end{remark}

The results of Lemma~\ref{lm:Gram_approximate} suggest  $\H_t\preceq\G_t\preceq(1+\delta_t)\H_t$ holds for some  $\delta_t\geq 0$.
Hence, we can extend the definition of $\eta_0$ to every iteration by denoting
\begin{align}
\label{eq:def_eta}
    \eta_t\triangleq \min_{\{\eta: 
 \H_t\preceq\G_t\preceq \eta\H_t\}}\eta.
\end{align}

\subsection{The Local Convergence for the Underdetermined Case}
\label{sec:underde}
In this section, we analyze the convergence of our RQGN method for solving the nonlinear equations  \eqref{eq:nonlinear_obj} in the case of $n<d$.
We suppose the singular value of the Jacobian is lower bounded on a local residual region.
\begin{assumption}
\label{ass:Omegan<d}
    We assume that Assumption~\ref{ass:nonsingular} holds with
    \begin{align*}
    \Omega
    &=\Omega_{\mathrm{under}}
    \triangleq\left\{\x:\|\F(\x)\|\leq\frac{\mu^2}{2L_2}\right\},
    \end{align*}
    i.e., there exists a constant $\mu>0$ such that $\sigma_m(\J(\x))\geq \mu$ holds for all $\x\in\Omega_{\mathrm{under}}$.
\end{assumption}
\begin{remark}
   In the global analysis of a modified Gauss--Newton method, \citet{nesterov2007modified} require the singular value of the Jacobian is lower bounded for all $\x\in\RB^d$.
   Here, our Assumption~\ref{ass:Omegan<d} only requires the lower bound holds on the region $\Omega_{\mathrm{under}}$ since we focus on the local convergence analysis.
\end{remark}


Recall that the classical analysis of quasi-Newton methods for minimizing the strongly convex function $f:\BR^d\to\BR$ \citep{rodomanov2021greedy} describes the convergence rates by the scaled gradient norm 
\begin{align*}
\lambda(\vx)\triangleq\sqrt{\nabla f(\x)^{\top}(\nabla^2f(\x))^{-1}\nabla f(\x)}     
\end{align*}
where $\nabla f(\cdot)$ and $\nabla^2 f(\cdot)$ are the gradient and the Hessian of $f(\cdot)$, respectively.
A direct extension of the above metric to nonlinear linear equations in the view of taking~$\mF(\cdot)=\nabla f(\cdot)$ and $\mJ(\cdot)=\nabla^2 f(\vx)$ is
\begin{align}
\label{eq:ordi-measure}
    \hat{\lambda}(\x) \triangleq \sqrt{\F(\x)^{\top}(\J(\x))^{-1}\F(\x)}.
\end{align}
However, for the underdetermined nonlinear equations, the Jacobian $\J(\cdot)$ may non-invertible or even non-square. 

For nonlinear equation problems,
it is worth noting that Proposition~\ref{prop:gram_lower} indicates the Gram matrix $\H(\cdot)$ is positive definite on the local region.
This motivates us to construct the measure of the weighted norm with respect to $\H(\cdot)$, i.e.,
\begin{align*}
    \lambda_{\mathrm{under}}(\x) \triangleq \sqrt{\F(\x)^{\top}(\H(\x))^{-1}\F(\x)}.
\end{align*}
Based on the above metric, we establish the following lemma shows the convergence for the general iteration \eqref{eq:qNupdate} in the underdetermined case.
\begin{lemma}
\label{lm:linear-quadra_n<d}
Under Assumptions~\ref{ass:Jlip} and \ref{ass:Omegan<d}, we additionally suppose the iteration \eqref{eq:qNupdate} in the case of $n<d$ holds
  \begin{align}
  \label{eq:G_condi_2}
      \H(\x)\preceq \G\preceq\eta\H(\x)
      \qquad\text{and}\qquad
      \lambda_{\mathrm{under}}(\x)\leq \frac{\mu^2}{4LL_2}
  \end{align}
for some $\eta\geq1$, then it holds $\x,\x_{+}\in\Omega_{\mathrm{under}}$, $\|\x_{+}-\x\|\leq\lambda_{\mathrm{under}}(\x)$, and
  \begin{align}
  \label{eq:linear-quadra-n<d}
      \lambda_{\mathrm{under}}(\x_{+})\leq \left(1-\frac{1}{\eta}\right)\lambda_{\mathrm{under}}(\x) + \frac{3LL_2}{\mu^2}(\lambda_{\mathrm{under}}(\x))^2.
  \end{align}
\end{lemma}

We now combine results of Lemma~\ref{lm:linear-quadra_n<d}, the partial relation in Proposition~\ref{prop:Gram_concor}, and the partial order-preserving property in Lemma~\ref{lm:QN_order} to establish the following explicit linear convergence rate for our methods.
\begin{theorem}
\label{thm:linear_under}
    Under Assumptions~\ref{ass:Jlip} and \ref{ass:Omegan<d}, we run Algorithm~\ref{alg:RQGN} with the initial point $\x_0$ such that
    \begin{align}
        \label{eq:linear_initial_condi_2}
           \lambda_{\mathrm{under}}(\x_0)\leq \frac{\mu^2}{20\eta_0LL_2}.
    \end{align}
    Then it holds $\x_t\in\Omega_{\mathrm{under}}$ for all $t$. Moreover, we also have
    \begin{align}
\label{eq:induc_lambda_n<d}
   \H_t \preceq \G_t\preceq \eta_0\exp{\left(2\sum_{i=0}^{t-1}Mr_i\right)}\H_t \preceq \frac{3\eta_0}{2}\H_t~\text{and}~   \lambda_{\mathrm{under}}(\x_t)\leq \left(1-\frac{1}{2\eta_0}\right)^t\lambda_{\mathrm{under}}(\x_0).
\end{align}
\end{theorem}
\begin{remark}
Unlike existing quasi-Newton methods for solving underdetermined case that require the initial Jacobian estimator to be sufficiently close to the exact Jacobian~\citep{vater2024convergence,martinez1991quasi},
the initial condition $\G_0\succeq\H_0$ in Algorithm~\ref{alg:RQGN} can be easily satisfied by taking $\G_0=L^2\I_n$, which leads to $\eta_0\leq L^2/\mu^2$.  
\end{remark}
Theorem~\ref{thm:linear_under} guarantees $\x_t\in\Omega_{\mathrm{under}}$ holds for all $t$, verifying the trajectory condition required by Lemma~\ref{lm:Gram_approximate} with $\Omega=\Omega_{\mathrm{under}}$. 
Hence, the sequence $\{\eta_t\}$ follows equation \eqref{eq:def_eta} is well-defined along the iterates.
Combining Lemma~\ref{lm:Gram_approximate} and Theorem~\ref{thm:linear_under}, we can show that that the sequence $\{\eta_t\}$ converges to $1$ in expectation and Algorithm~\ref{alg:RQGN} has the following explicit local superlinear convergence for the underdetermined case.

\begin{theorem}
\label{thm:n<dsuperlinear}
    Under Assumptions~\ref{ass:Jlip} and \ref{ass:Omegan<d}, run Algorithm~\ref{alg:RQGN} with $\x_0\in\Omega_{\rm under}$ satisfying
    \begin{align}
    \label{eq:initial_condi_n<d}
        \lambda_{\mathrm{under}}(\x_0)\leq \frac{\mu^2(1-\gamma)}{20\tilde{c}nLL_2\eta_0},
    \end{align}
where the values $\gamma$ and $\tilde{c}$ for each of Options I--III follow Lemma~\ref{lm:Gram_approximate}.
Then we have
    \begin{align}
    \label{eq:expeta_tlambda_t}
        \EBP{\eta_t-1}\leq \left(1-\gamma\right)^tq_{\text{\rm under}}\qquad\text{and}\qquad\EBP{\frac{\lambda_{\mathrm{under}}(\x_{t+1})}{\lambda_{\mathrm{under}}(\x_t)}}\leq (1-\gamma)^tq_{\text{\rm under}},
    \end{align}
where $q_{\text{\rm under}}\triangleq {\rm e}^{11/10}(\delta_0+7(1-\gamma)/(20\eta_0))$
and the value of $\delta_0$ for each of Options I--III follows Lemma~\ref{lm:Gram_approximate}.
Moreover, for all $p\in(0,1)$, there exists $t_0=\OM(\log(1+q_{\text{\rm under}}/(p\gamma^2))/\gamma)$ such that it holds
    \begin{align*}
        \lambda_{\mathrm{under}}(\x_{t_0+t})\leq \left(1-\frac{\gamma}{\gamma+1}\right)^{\frac{t(t-1)}{2}}\left(\frac{1}{2}\right)^{t}\left(1-\frac{1}{2\eta_0}\right)^{t_0}\lambda_{\mathrm{under}}(\x_0)
    \end{align*}
with probability at least $1-p$ for all $t\geq1$.
\end{theorem}

\subsection{The Local Convergence for the Overdetermined Case}
\label{sec:overfit}

This section considers the overdetermined case of $n\geq d$ by focusing on the nonlinear least square formulation \eqref{eq:nonlinear_least_obj}. 
Note that the gradient and the Hessian of the objective
$\phi(\x)=(1/2)\|\F(\x)\|^2$ can be written as
\begin{align*}
\nabla \phi(\vx)=\mJ(\vx)^\top\mF(\vx)
\qquad\text{and}\qquad
\nabla^2\phi(\x) =\H(\x) + \Q(\x),
\end{align*}
where $\mH(\vx)=\mJ(\vx)^\top\mJ(\vx)$ and $\Q(\x)\triangleq \sum_{i=1}^n F_i(\x)\nabla^2 F_i(\x)$.
We suppose that the existence of the stationary point  $\x_*\in\BR^d$ such that 
$\nabla\phi(\vx_*)=\mJ(\x_*)^{\top}\F(\x_*)=\0$
and the lower boundedness for the singular value of the Jacobian $\J(\x_*)$.
\begin{assumption}
\label{ass:nonsingular_J^*}
We assume there exists $\x_*\in\BR^d$ such that $\J(\x_*)^{\top}\F(\x_*)=\0$ and the Jacobian at $\x_*$ satisfies $\sigma_m(\J(\x_*))\geq \mu_*$ for some $\mu_*>0$, where $m=\min\{n,d\}$.
\end{assumption}
Combining Assumptions~\ref{ass:Jlip} and \ref{ass:nonsingular_J^*} suggests  Assumption~\ref{ass:nonsingular} can be satisfied by taking a proper region $\Omega$, which is introduced in the following proposition. 
\begin{proposition}
\label{prop:local_nonsingularity}
    Under Assumptions~\ref{ass:Jlip} and \ref{ass:nonsingular_J^*}, Assumption~\ref{ass:nonsingular} is satisfied with
    \begin{align*}
    \Omega
    &=\Omega_{\mathrm{over}}
    \triangleq\left\{\x:\|\x-\x_*\|\leq \frac{\mu^2_*}{4LL_2}\right\}
    \quad\text{and}\quad
    \mu=\frac{\mu_*}{\sqrt{2}}.
    \end{align*}
\end{proposition}

We then impose the following smoothness assumption with respect to the stationary point $\vx_*$.
\begin{assumption}
\label{ass:x_star_philip}
    We assume there exists a constant $\rho\geq 0$ such that 
    \begin{align}
        \label{eq:x_star_philip}
        \big\|\J(\x)^{\top}\F(\x)-(\H(\x_*)+\Q(\x_*))(\x-\x_*)\big\| \leq \rho \|\x-\x_*\|^2.
    \end{align}
    for all $\vx\in\BR^d$, where $\vx_*\in\BR^d$ is the stationary point of $\phi(\cdot)$ such that $\mJ(\vx_*)^\top\mF(\vx_*)=\vzero$.
\end{assumption}
The following proposition verifies that our Assumption \ref{ass:x_star_philip} is more general than several classical smoothness assumptions used in prior studies for nonlinear equations \citep{gould2019convergence,zhou2010global,vater2024convergence,liu2023block}.
\begin{proposition}
\label{prop:assumption-sati}
    Assumption~\ref{ass:x_star_philip} can be verified in one of the following three cases.
    \begin{enumerate}
        \item  Suppose Assumption~\ref{ass:Jlip} holds and the Hessian $\nabla^2 F_i(\cdot)$ is $L_3$-Lipschitz continuous with respect to the stationary point of $\phi(\cdot)$ for all $i\in\BR^n$
         \citep[Assumption AS.1]{gould2019convergence} \citep[Lemma 2.9--2.11]{zhou2010global}, i.e., there exists constant $L_3>0$ such~that
        \begin{align*}
        \|\nabla^2F_i(\x)-\nabla^2 F_i(\x_*)\|
        \leq L_3\|\x-\x_*\|    
        \end{align*}
        for all $\vx\in\BR^d$, where $\vx_*\in\BR^d$ is the stationary point of $\phi(\cdot)$ such that $\mJ(\vx_*)^\top\mF(\vx_*)=\vzero$.
        Then for all $\x\in\BR^d$ satisfying $\|\x-\x_*\|\leq\delta$, Assumption~\ref{ass:x_star_philip} holds with
        \begin{align}\label{eq:rho-cond-1}
        \rho
        =\frac{1}{2}\Big(2LL_2+L\sum_{i=1}^{n}\|\nabla^2 F_i(\x_*)\|
        +L_3\sqrt{n}L\delta+L_3\|\F(\x_*)\|_1\Big),
        \end{align}
        where $\|\cdot\|_1$ is the $\ell_1$-norm.
        \item  Suppose the Hessian of $\phi(\cdot)$ is $\rho_0$-Lipschitz continuous \citep[Assumption B]{zhou2010global} , i.e., there exists a constant $\rho_0>0$ such that
        \begin{align}\label{eq:phi-2-Lip-xy}
        \|\nabla^2\phi(\x)-\nabla^2\phi(\y)\|
        \leq\rho_0\|\x-\y\|
        \end{align}
        for all $\vx,\vy\in\BR^d$.
        Then Assumption~\ref{ass:x_star_philip} holds with $\rho=\rho_0/2$.
        \item Suppose Assumption~\ref{ass:Jlip} holds and the point $\x_*\in\BR^d$ is the solution of the nonlinear equations such that $\F(\x_*)= \0$ \citep{vater2024convergence, liu2023block}, then Assumption~\ref{ass:x_star_philip} holds with $\rho = {3LL_2}/{2}$.
    \end{enumerate}
\end{proposition}

We also make the following assumption on $\H(\x_*)$ and $\Q(\x_*)$, which is widely used and necessary to guarantee the local linear convergence of the Gauss--Newton method~\citep{gratton2007approximate,ortega2000iterative,wedin1974gauss}.
\begin{assumption}
\label{ass:deltacondi}
We suppose that
\begin{align}
\label{eq:condiDelta}
\theta\triangleq\Norm{\H(\x_*)^{-1/2}\Q(\x_*)\H(\x_*)^{-1/2}}\in[0,1),
\end{align}
where $\mH(\cdot)=\mJ(\cdot)^\top\mJ(\cdot)$, $\Q(\cdot)\triangleq \sum_{i=1}^n F_i(\cdot)\nabla^2 F_i(\cdot)$, and $\vx_*$ is the stationary point follows Assumption \ref{ass:nonsingular_J^*}.
\end{assumption}
Note that Assumption~\ref{ass:deltacondi} indicates the positive definiteness of $\nabla^2\phi(\x_*)$ for the stationary point, while it does not guarantee the positive definiteness of $\nabla^2\phi(\x)$ always holds even for~$\vx\in\BR^d$ in some local region.
However, the following lemma shows that we can guarantee the positive definiteness of the matrix
\begin{align*}
    \tilde{\nabla}^2\phi(\x)\triangleq \H(\x) + \Q(\x_*),
\end{align*}
which can be regarded as an approximation of the Hessian $\nabla^2\phi(\x) =\H(\x) + \Q(\x)$.
\begin{lemma}
\label{lm:pd_approximate_phi}
Under Assumptions~\ref{ass:Jlip}, \ref{ass:nonsingular_J^*}, and~\ref{ass:deltacondi}, 
for all $c\in(1,1/\theta)$ and $\x,\vy\in\BR^d$ such that~$\|\x-\x_*\|\leq(c-1)\mu_*^2/(4cLL_2)$, it holds
\begin{align}\label{eq:approx_phi_positive}
     (1-c\theta)\mu^2 \I_d \preceq (1-c\theta)\H(\x)\preceq  \tilde{\nabla}^2\phi(\x)\preceq (1+c\theta)\H(\x)\preceq (1+c\theta)L^2\I_d.
\end{align}
and
\begin{align}\label{eq:concor_phi}
    (1-\widehat M\|\y-\x\|)\tilde{\nabla}^2\phi(\x)
        \preceq\tilde{\nabla}^2\phi(\y)
        \preceq(1+\widehat M\|\y-\x\|)\tilde{\nabla}^2\phi(\x),
\end{align}
where $\widehat M\triangleq 2LL_2/((1-c\theta)\mu^2)$ and $\mu=\mu_*/\sqrt{2}$. 
\end{lemma}

Note that Gauss--Newton iteration \eqref{update:gn} for the case of $n\geq d$ can be written as
\begin{align}\label{eq:update_n>d}
\begin{split}    
    \vx_+= &\x-(\J(\x)^{\top}\J(\x))^{-1}\J(\x)^{\top}\mF(\vx)  \\
    = & \x- \mH^{-1}\nabla \phi(\x),
\end{split}
\end{align}
since we have $\mH(\vx)=\J(\x)^{\top}\J(\x)$ and $\nabla \phi(\x)=\J(\x)^{\top}\F(\x)$.
Recall that classical linear-quadratic convergence rate for iteration \eqref{eq:update_n>d} is established based on the distance $\|\x-\x_*\|$ \citep{gratton2007approximate,dennis1996numerical}.
However, this metric does not work for proposed RQGN methods because our iteration \eqref{eq:qNupdate} only accesses the Gram estimator $\mG$ rather than the exact $\mH(\vx)$.

For analyzing the convergence of our RQGN, we construct the weighted norm 
\begin{align}\label{eq:our_measure}
\begin{split}\    
    \lambda_{\mathrm{over}}(\x)
    \triangleq & \big\|(\H(\x)+\Q(\x_*))^{1/2}(\x-\x_*)\big\| \\
    = & \big\|(\tilde{\nabla}^2\phi(\x))^{1/2}(\x-\x_*)\big\|
\end{split}
\end{align}
to address the challenge of the approximation error between $\G$ and $\H(\x)$.
Note that Lemma~\ref{lm:pd_approximate_phi} shows that the matrix $\tilde{\nabla}^2\phi(\x)$ is positive-definite in the local region, which guarantees the metric $\lambda_{\mathrm{over}}(\cdot)$ is well-defined in our local convergence analysis.
Consequently, we establish the partial relation and explicit linear quadratic rate for iteration~\eqref{eq:update_n>d} as follows.
\begin{lemma}\label{lm:GNlinearquadra}
Under Assumptions~\ref{ass:Jlip}, \ref{ass:nonsingular_J^*}, \ref{ass:x_star_philip}, and \ref{ass:deltacondi}, we additionally suppose the iteration~\eqref{eq:qNupdate} in the case of $n\geq d$ holds 
\begin{align}\label{eq:condi_x_linear_quadra}
    \H(\x)\preceq\G\preceq\eta\H(\x),\quad
    \|\x-\x_*\|\leq\frac{(c-1)\mu_*^2}{4cLL_2},\quad
    \text{and}\quad\lambda_{\mathrm{over}}(\x)\leq
    \frac{(1-c\theta)^{3/2}\mu_*^5}{16D},
\end{align}
for some $\eta\geq1$ and $c\in(1,1/\theta)$, where
$D\triangleq L^3L_2+LL_2^2\|\F(\x_*)\|$.
Then it holds
\begin{align}
& \lambda_{\mathrm{over}}(\x_{+}) \leq \left(1-\frac{1-c\theta}{\eta}\right)\lambda_{\mathrm{over}}(\x) +C_1(\lambda_{\mathrm{over}}(\x))^2, \label{eq:overfit_linear_quadra}\\
&
\|\x_{+}-\x_*\| \leq \frac{1}{\sqrt{1-c\theta}\mu}
\left(\left(1-\frac{1-c\theta}{\eta}\right)
\lambda_{\mathrm{over}}(\x) 
+C_1(\lambda_{\mathrm{over}}(\x))^2\right), \label{eq:overfit_distance_lambda} \\
& \|\x_{+}-\x\| \leq\frac{L^2+L_2\|\F(\x_*)\|}{\sqrt{1-c\theta}\mu^3}
\lambda_{\mathrm{over}}(\x), \label{eq:movedistance_lambda}
\end{align}
where 
\begin{align}\label{eq:mu-C1}
\mu=\frac{\mu_*}{\sqrt{2}} 
\quad\text{and}\quad
C_1
&\triangleq\frac{2D}{(1-c\theta)^{3/2}\mu^5}
+\left(\frac{4(1+c\theta)}{(1-c\theta)^{3/2}}+\frac{2c}{c-1}\right)\frac{LL_2}{\mu^3}
+\frac{2\sqrt{1+c\theta}\,\rho}{(1-c\theta)\mu^3}.
\end{align}
\end{lemma}

We now combine Lemma~\ref{lm:GNlinearquadra}, Proposition~\ref{prop:Gram_concor}, and Lemma~\ref{lm:QN_order} to establish a linear convergence rate and additionally show that all iterates remain in the local region 
$\Omega_{\mathrm{over}}$.
\begin{theorem}
\label{thm:linear1n>d}
Under Assumptions~\ref{ass:Jlip}, \ref{ass:nonsingular_J^*}, \ref{ass:x_star_philip}, and \ref{ass:deltacondi}, 
we run Algorithm~\ref{alg:RQGN} with the initial point $\x_0$ such that
\begin{align}
\label{eq:linear_initial_condi}
\|\x_0-\x_*\|\leq\frac{(c-1)\mu_*^2}{4cLL_2}
\quad\text{and}\quad
\lambda_{\mathrm{over}}(\x_0)\leq \frac{1-c\theta}{10\eta_0 C_1},
\end{align}
for some $c\in(1,1/\theta)$, where $C_1$ follows the definition in Lemma \ref{lm:GNlinearquadra}.
Then it holds $\x_t\in\Omega_{\mathrm{over}}$ for all $t$. Moreover, we also have
\begin{align}
\label{eq:induc_lambda}
\H_t\preceq\G_t\preceq\eta_0{\rm e}^{2M\sum_{i=0}^{t-1}r_i}\H_t
\preceq\frac{3\eta_0}{2}\H_t
\quad\text{and}\quad
\lambda_{\mathrm{over}}(\x_t)\leq\left(1-\frac{1-c\theta}{2\eta_0}\right)^t\lambda_{\mathrm{over}}(\x_0).
\end{align}
\end{theorem}
Theorem~\ref{thm:linear1n>d} guarantees that $\x_t\in\Omega_{\mathrm{over}}$ for all $t$, verifying the trajectory condition required by Lemma~\ref{lm:Gram_approximate} with $\Omega=\Omega_{\mathrm{over}}$.
Hence, the sequence $\{\eta_t\}$ defined in~\eqref{eq:def_eta} is well-defined along the iterates like the underdetermined case.
Combining Lemma~\ref{lm:Gram_approximate} and Theorem~\ref{thm:linear1n>d}, we establish finite-time bounds on the Gram approximation error and a two-stage local linear convergence rate.
\begin{theorem}
\label{thm:two-stage-n>d}
Under Assumptions~\ref{ass:Jlip}, \ref{ass:nonsingular_J^*}, \ref{ass:x_star_philip}, and \ref{ass:deltacondi}, run Algorithm~\ref{alg:RQGN} with the initial point~$\x_0$ satisfying
\begin{align}
\label{eq:inicondi_final_n>d}
\|\x_0-\x_*\|\leq\frac{(c-1)\mu_*^2}{4cLL_2}
\qquad\text{and}\qquad
\lambda_{\mathrm{over}}(\x_0)\leq
\frac{(1-\gamma)(1-c\theta)}{10(2-\gamma)\eta_0C_2},
\end{align}
for some $c\in(1,1/\theta)$, where
$C_2\triangleq C_1+4\tilde{c}dD/(\sqrt{1-c\theta}\mu^5)$,
the values of $\gamma$ and $\tilde{c}$ for each of Options I--III follow Lemma~\ref{lm:Gram_approximate}, 
and the values of $D$, $\mu$, and $C_1$ follow Lemma~\ref{lm:GNlinearquadra}.
Then we have
\begin{align}
\label{eq:eta-ratio-over}
\EBP{\eta_t-1}
&\leq q_{\text{\rm over}}(t+1)(1-\nu)^t
\quad\text{and}\quad
\EBP{\frac{\lambda_{\mathrm{over}}(\x_{t+1})}{\lambda_{\mathrm{over}}(\x_t)}}
\leq c\theta+q_{\text{\rm over}}(t+1)(1-\nu)^t,
\end{align}
where $\nu\triangleq\min\{\gamma,(1-c\theta)/(2\eta_0)\}$,
$q_{\text{\rm over}}\triangleq {\rm e}^{1/5}
(\delta_0+(1-\gamma)(1-c\theta)/(10(2-\gamma)\eta_0))$, and the
value of~$\delta_0$ for each of Options I--III follows
Lemma~\ref{lm:Gram_approximate}.
Moreover, for every $p\in(0,1)$, there exists
$t_0=\OM\big(\log(1+q_{\text{\rm over}}/(p(1-c\theta)\nu^2))/\nu\big)$
such that, with probability at least $1-p$, we have
\begin{align}
\label{eq:n>dtwostage}
\lambda_{\mathrm{over}}(\x_{t_0+t})
\leq\left(\frac{1+c\theta}{2}\right)^t
\left(1-\frac{1-c\theta}{2\eta_0}\right)^{t_0}\lambda_{\mathrm{over}}(\x_0),
\end{align}
simultaneously for all $t\geq0$.
\end{theorem}

\begin{remark}
Note that even for the standard Gauss--Newton iteration \eqref{eq:update_n>d},
exact the Gram matrix~$\H(\cdot)=\J(\cdot)^{\top}\J(\cdot)$ does not include the term $\mQ(\cdot)$ in the Hessian  $\nabla^2\phi(\x)=\H(\x)+\Q(\x)$.
Naturally, our RQGN methods also does not consider the term $\mQ(\cdot)$ in the construction of the Gram estimator.
This results the expected ratio~\eqref{eq:eta-ratio-over} 
for shown in Theorem~\ref{thm:two-stage-n>d} contains an additional term $c\theta$, so that the overdetermined case cannot always enjoy the superlinear convergence rate like that of Theorem~\ref{thm:n<dsuperlinear} for the underdetermined case.
\end{remark}

If we strengthen the conditions of Theorem \ref{thm:linear1n>d} by assuming that Assumption \ref{ass:deltacondi} satisfies with $\theta=0$, we can show that our RQGN methods enjoys the superlinear local convergence.

\begin{corollary}
\label{col:theta=0n>d}
Under the same notation and setting as Theorem~\ref{thm:two-stage-n>d}, suppose Assumption~\ref{ass:deltacondi} holds with $\theta=0$ and $c>1$.
Then, for all $p\in(0,1)$, there exists $t_0=\OM(\log(1+q_{\text{\rm over}}/(p\gamma^2))/\gamma)$ such that it holds
\begin{align*}
\lambda_{\mathrm{over}}(\x_{t_0+t})\leq \left(1-\frac{\gamma}{\gamma+1}\right)^{\frac{t(t-1)}{2}}\left(\frac{1}{2}\right)^t\left(1-\frac{1}{2\eta_0}\right)^{t_0}\lambda_{\mathrm{over}}(\x_0)
\end{align*}
with probability at least $1-p$ for all $t\geq1$.
\end{corollary}

\begin{remark}
    If there exists $\vx^*\in\BR^d$ such that $\F(\x^*)=\0$ (case 3 in Proposition~\ref{prop:assumption-sati}), then Assumption~\ref{ass:deltacondi} holds with $\theta = 0$.            
    This implies our local superlinear convergence result in Corollary~\ref{col:theta=0n>d} can apply to several classical problem settings such as  square nonlinear equations~\citep{lin2021explicit,liu2023block} and strongly-convex-strongly-concave minimax optimization~\citep{liu2022quasi}.
\end{remark}

\section{Numerical Experiments}
\label{sec:exp}
In this section, we conduct experiments on scientific computing and machine learning applications to verify the advantage of proposed methods for solving both underdetermined and overdetermined nonlinear equations.
We run our RQGN (Algorithm \ref{alg:RQGN}) with SR1 and fast BFGS updates, denoted by RQGN-SR1 and RQGN-FBFGS, respectively.
We set~$\G_0=L^2\I_m$ and tune parameters $L,M$ over $\{1,10,100,1000\}$ for the algorithm.
All experiments are performed on a PC with Apple M1 and all algorithms are implemented in Python 3.8.12.

\subsection{Empirical Results for the Underdetermined Case}
We first consider the problem of finding eigenvalues of the symmetric matrix $\A\in\RB^{n\times n}$ \citep{vater2024convergence}, which can be formulated by solving the underdetermined nonlinear equation  \eqref{eq:nonlinear_obj} with $\mF:\BR^{n+1}\to\BR^n$ such that
\begin{align}
\label{eq:eigenfinding}
\F(\x)\triangleq (\A-\lambda\I_{n})\u
\qquad\text{and}\qquad
\x= \begin{bmatrix}
    \lambda \\ \u
\end{bmatrix}\in\RB^{n+1}    
\end{align}
We follow the setting of \citet{vater2024convergence} by letting $\A=[a_{i,j}]$ be a tridiagonal matrix with $a_{i,i}=10$ for $i\in\{1,\ldots,n\}$ and $a_{i,i-1}=a_{i-1,i}=-1$ for $i\in\{2,\ldots,n\}$. 

We compare our RQGN methods against baselines gradient descent (GD)~\citep{nesterov2018lectures} and generalized Broyden method~\citep{vater2024convergence}.
For GD, we perform the iteration
$\x_{t+1}=\x_t - \gamma \J(\x_t)^{\top}\F(\x_t)$ in the view of the nonlinear least-square formulation and tune the stepsize $\gamma$ from~$\{0.01,0.1,1,10,100\}$.
For generalized Broyden method, we use the scaled-identity matrix as the initialization like our RQGN.
We report empirical results on iterations and CPU time for $n\in\{300,400,500\}$ in Figure~\ref{fig:eigen_find}.
We observe that our proposed RQGN methods outperform the baseline algorithm in most cases.

\begin{figure}[ht]
\centering
\begingroup
\setlength{\tabcolsep}{1pt}
\renewcommand{\arraystretch}{0.96}
\begin{tabular}{@{}ccc@{}}
\includegraphics[width=0.327\textwidth,trim=11pt 6pt 11pt 6pt,clip]{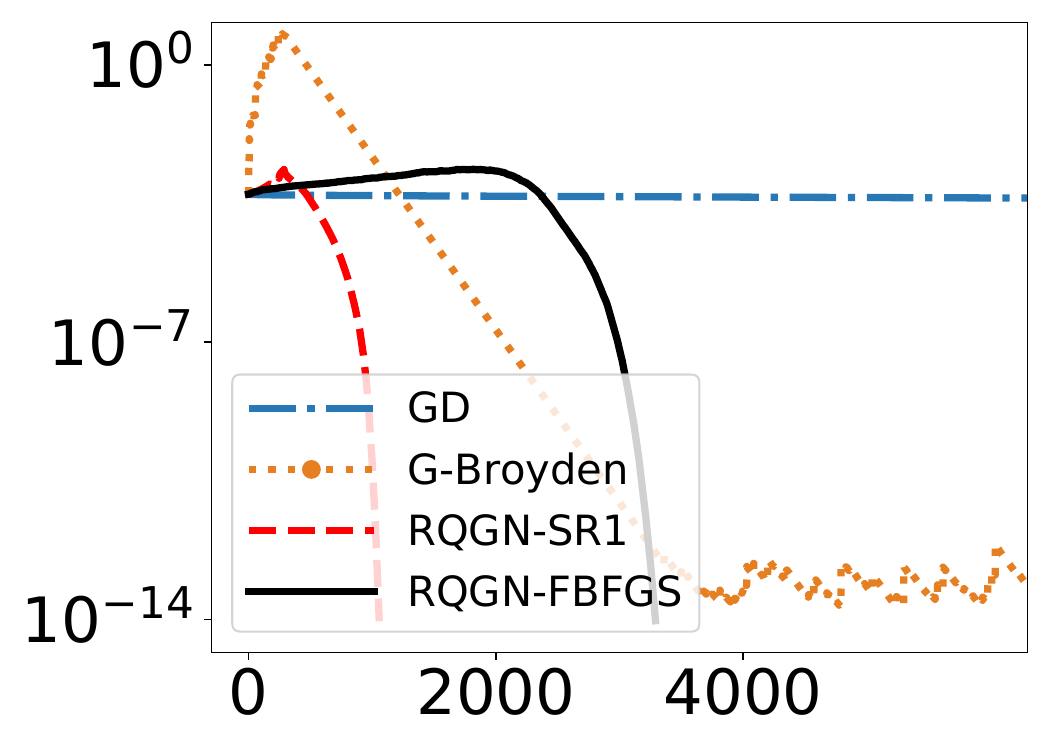} &
\includegraphics[width=0.327\textwidth,trim=11pt 6pt 11pt 6pt,clip]{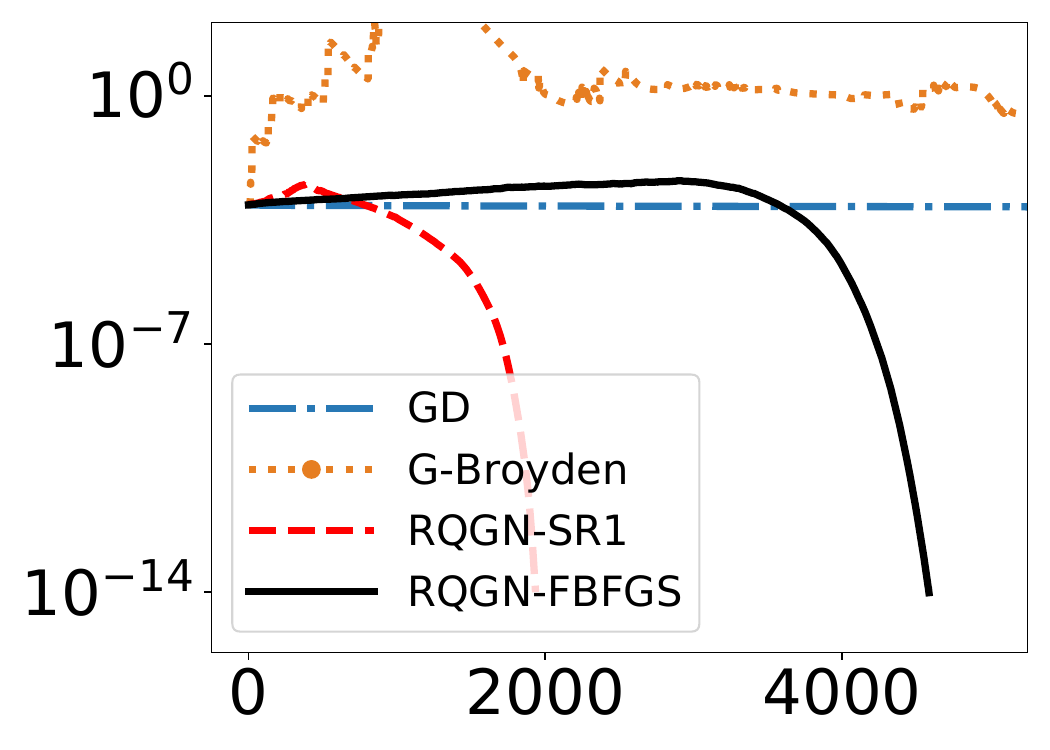} &
\includegraphics[width=0.327\textwidth,trim=11pt 6pt 11pt 6pt,clip]{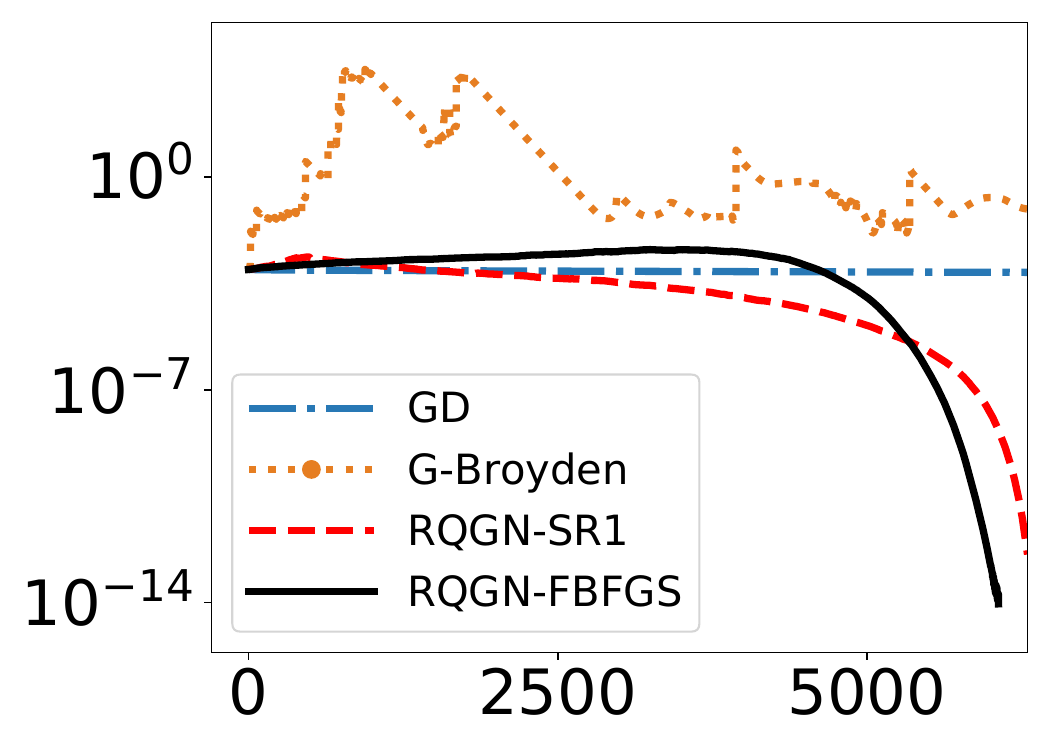} \\[-1pt]
{\footnotesize (a) $n=300$ (iteration)} & {\footnotesize (b) $n=400$ (iteration)} & {\footnotesize (c) $n=500$ (iteration)} \\[10pt]
\includegraphics[width=0.327\textwidth,trim=11pt 6pt 11pt 6pt,clip]{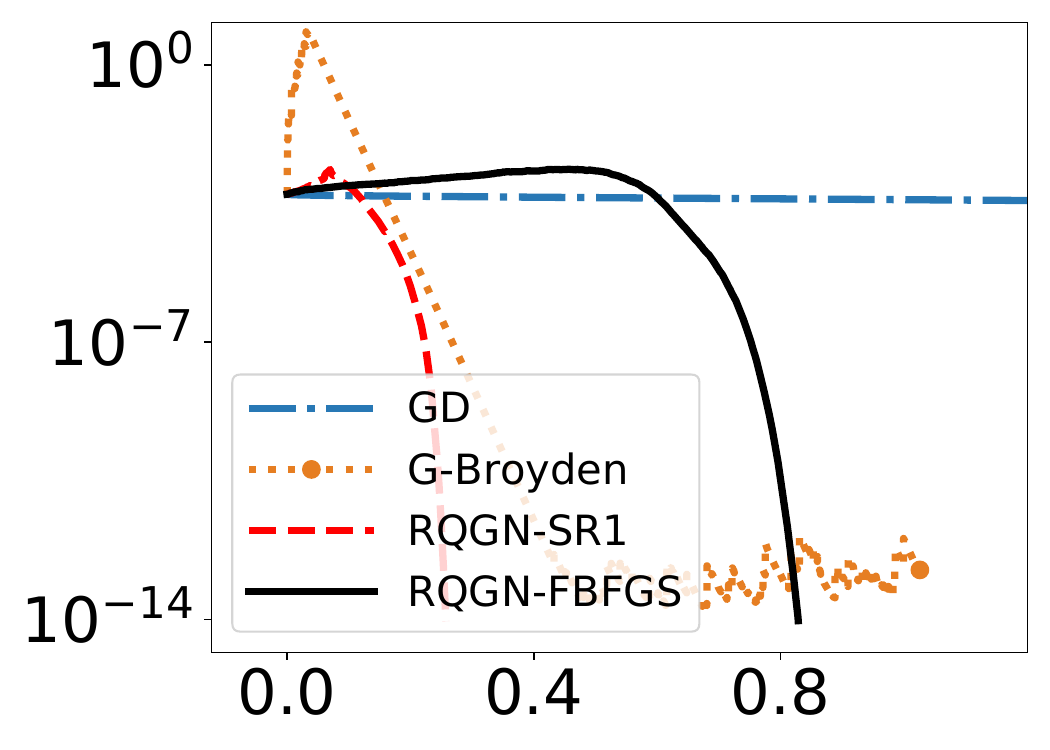} &
\includegraphics[width=0.327\textwidth,trim=11pt 6pt 11pt 6pt,clip]{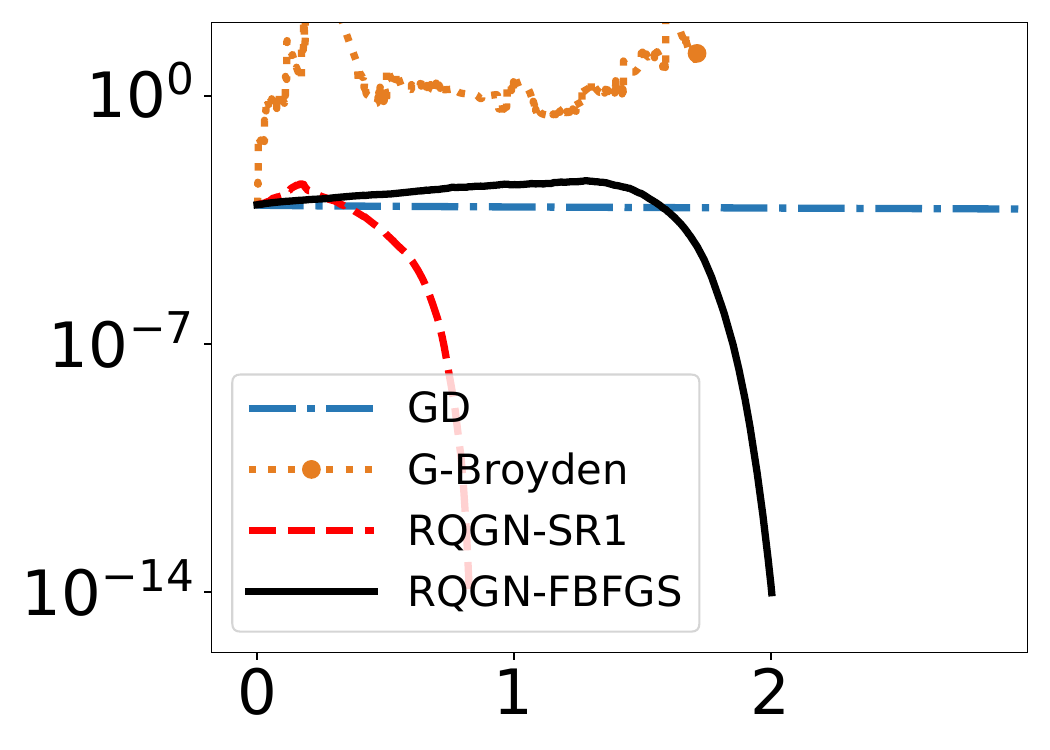} &
\includegraphics[width=0.327\textwidth,trim=11pt 6pt 11pt 6pt,clip]{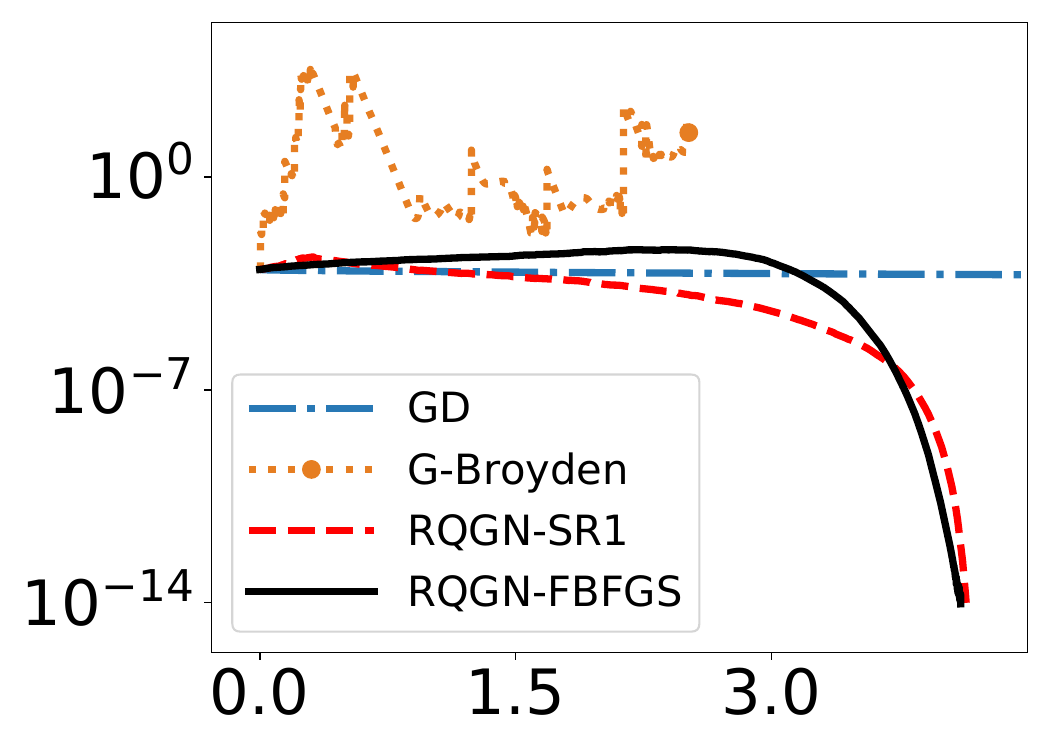} \\[-1pt]
{\footnotesize (d) $n=300$ (time)} & {\footnotesize (e) $n=400$ (time)} & {\footnotesize (f) $n=500$ (time)}
\end{tabular}
\endgroup
\caption{Empirical results of iteration counts and CPU time (seconds) against $\|\F(\x)\|$ for the underdetermined problem with nonlinear mapping \eqref{eq:eigenfinding}.
}
\label{fig:eigen_find}
\end{figure}


\subsection{Empirical Results for the Overdetermined Cases}
We then consider the generalized linear model \citep{shao2025convergence}, which can be formulated by  the nonlinear least square problem \eqref{eq:nonlinear_least_obj} with $\mF:\BR^d\to\BR^{2d}$ such that
\begin{align}
\label{eq:overfit}
    \F(\x)  \triangleq \begin{bmatrix}
        \nabla f_1(\x)\\
        \alpha\nabla f_2(\x)
    \end{bmatrix},
\end{align}
where $\alpha>0$ is the scaled parameter and $f_1:\BR^d\to\BR$ and $f_2:\BR^d\to\BR$ are objectives of log-sum-exp and logistic regression problems  \citep{rodomanov2021greedy} which are defined as 
\begin{align*}
f_1(\x)\triangleq \ln\left(\sum_{j=1}^N\exp(\va_j^{\top}\x-b_j)\right)+\frac{1}{2}\sum_{j=1}^N(\va_j^\top\x)^2
~\text{and}~f_2(\x)\triangleq \sum_{j=1}^N
 \ln\left(1+\exp\left(-b_j\va_j^{\top}\x\right)\right),
\end{align*}
where $\{(\va_i, b_i)\}_{i=1}^N$ is the training set with the feature~$\va_j\in\BR^d$ and the corresponding label~$b_j\in\BR$.
Here, we set $\alpha = 0.1$ in our experiments.

We compare our RQGN methods against the baseline GD by tuning the stepsize over $0.01$, $0.1$, $1$, $10$, and~$100$.
We perform the experiments on three real-world data sets, ``a1a'' ($N=1,605$, $d=123$), ``w1a'' ($N=2,477$, $d=300$), and ``splice'' ($N=1,000$, $d=60$), which can be downloaded from the LIBSVM repository~\citep{CC01a}.
We report empirical results on iterations and CPU time for all three data sets in Figure~\ref{fig:overfit}.
We can observe that our RQGN methods enjoy faster local convergence behavior and take less running time than the baseline.

\begin{figure}[!tbp]
\centering
\begingroup
\setlength{\tabcolsep}{1pt}
\renewcommand{\arraystretch}{0.96}
\begin{tabular}{@{}ccc@{}}
\includegraphics[width=0.327\textwidth,trim=11pt 6pt 11pt 6pt,clip]{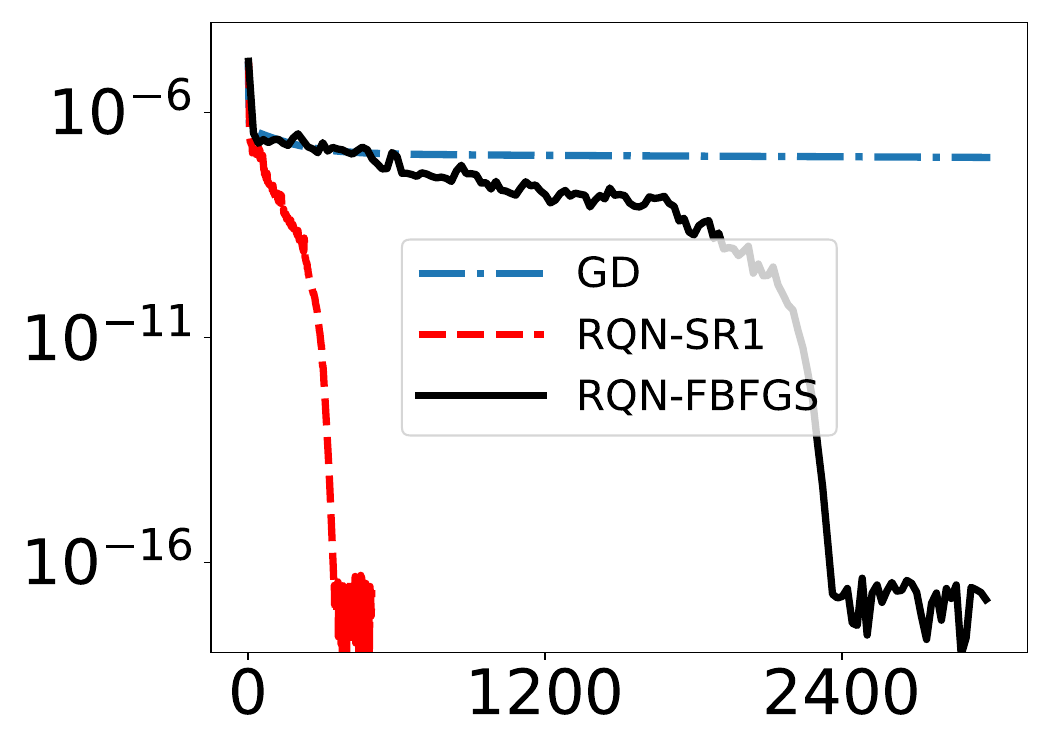} &
\includegraphics[width=0.327\textwidth,trim=11pt 6pt 11pt 6pt,clip]{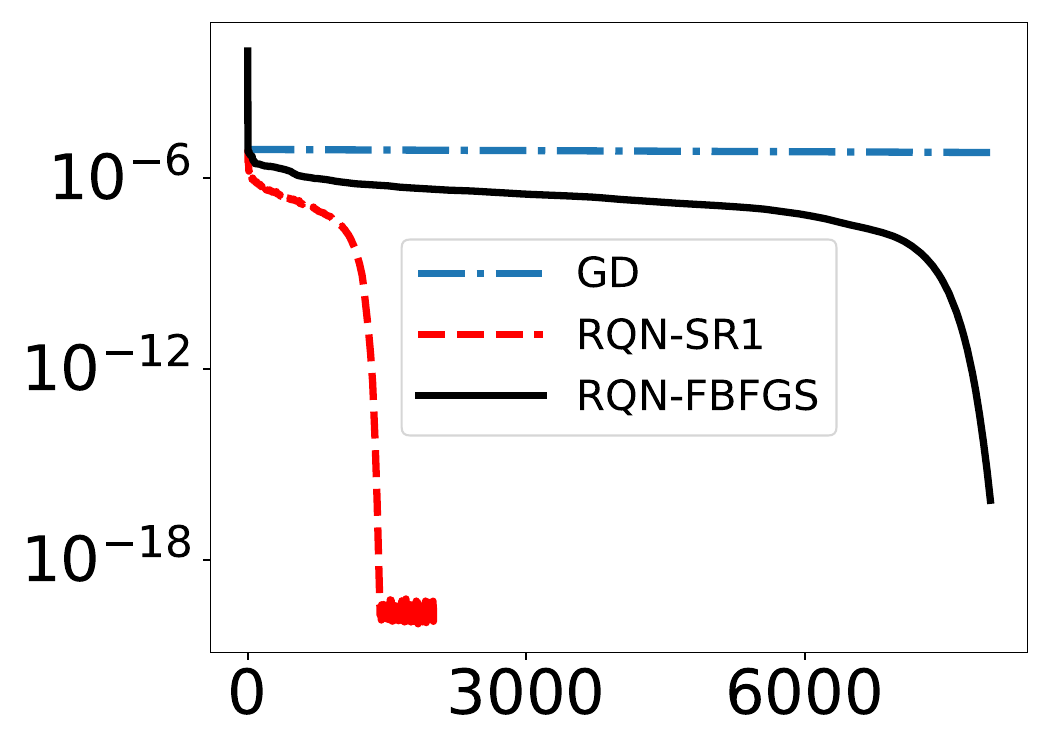} &
\includegraphics[width=0.327\textwidth,trim=11pt 6pt 11pt 6pt,clip]{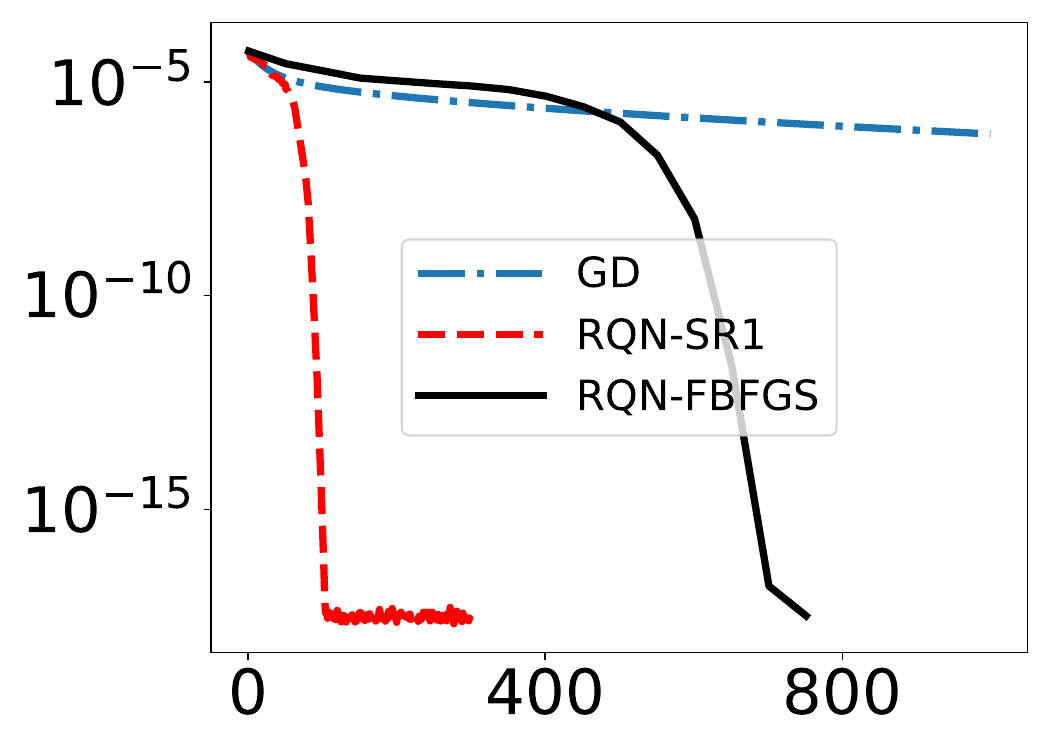} \\[-1pt]
{\footnotesize (a) ``a1a'' (iteration)} & {\footnotesize (b) ``w1a'' (iteration)} & {\footnotesize (c) ``splice'' (iteration)} \\[10pt]
\includegraphics[width=0.327\textwidth,trim=11pt 6pt 11pt 6pt,clip]{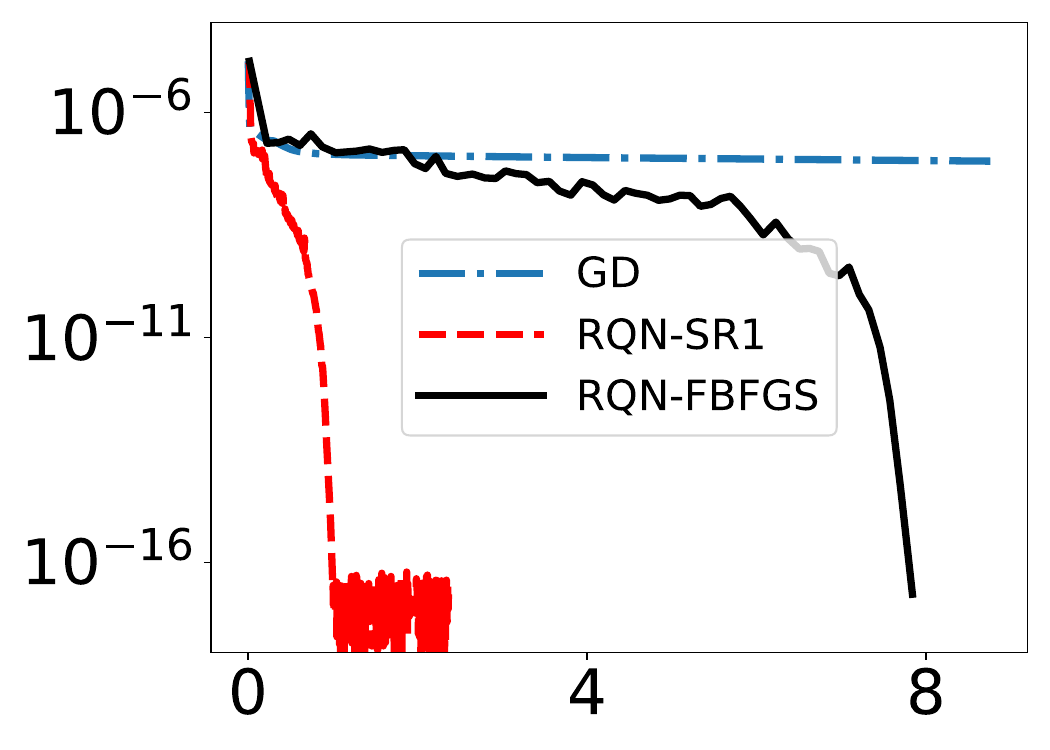} &
\includegraphics[width=0.327\textwidth,trim=11pt 6pt 11pt 6pt,clip]{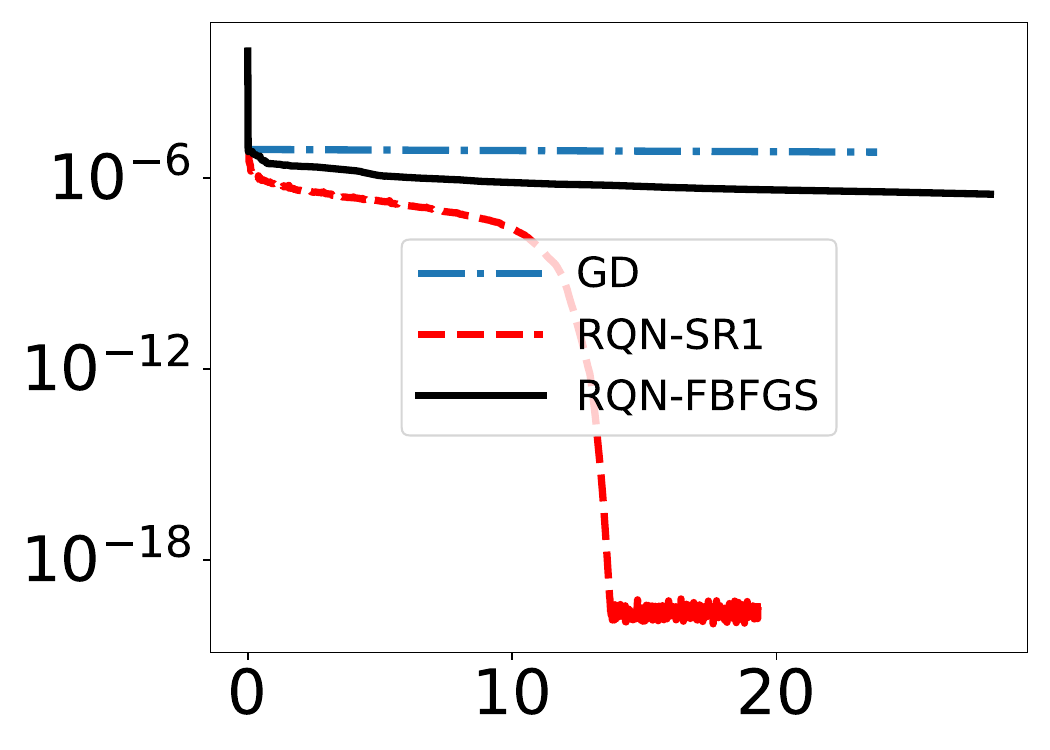} &
\includegraphics[width=0.327\textwidth,trim=11pt 6pt 11pt 6pt,clip]{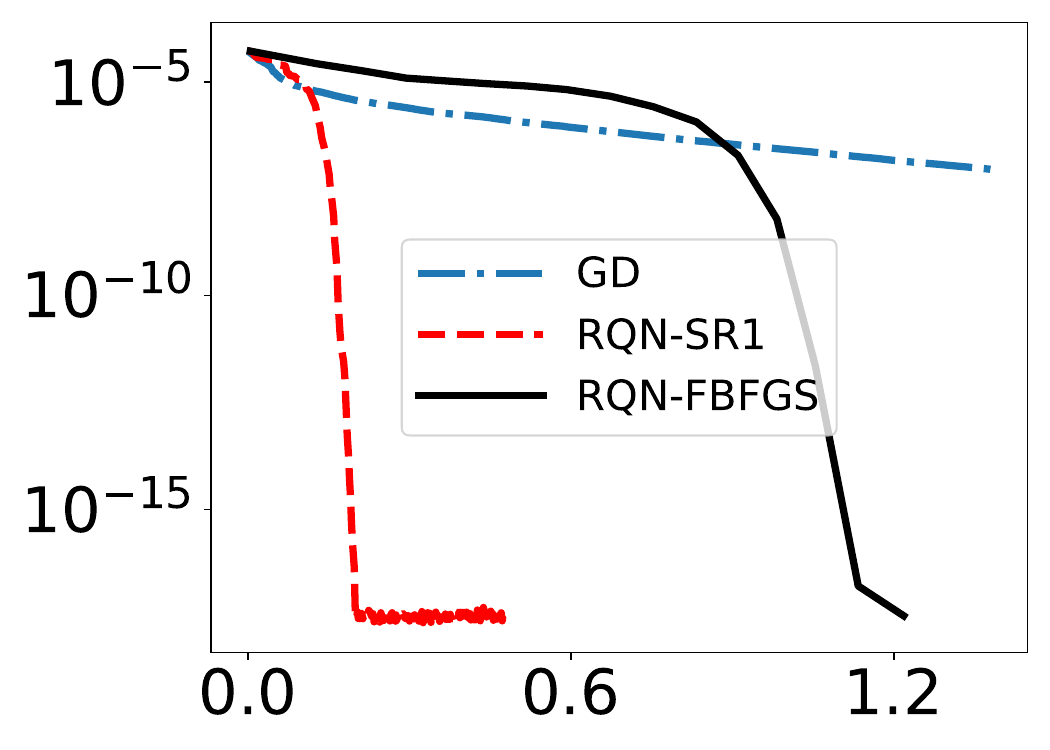} \\[-1pt]
{\footnotesize (d) ``a1a'' (time)} & {\footnotesize (e) ``w1a'' (time)} & {\footnotesize (f) ``splice'' (time)}
\end{tabular}
\endgroup
\caption{Empirical results of iteration counts and CPU time (seconds) against $\|\J(\x)^{\top}\F(\x)\|$ for the overdetermined problem with nonlinear mapping \eqref{eq:overfit}.}
\label{fig:overfit}
\end{figure}

\subsection{Numerical Results for the Square Case}
We also consider the problem of fairness-aware machine learning, which can be formulated as the minimax optimization problem
\begin{align}
\label{eq:minimax_fair}
\min_{\u\in\RB^{d-1}}\max_{v\in\RB}    f(\u,v) \triangleq \frac{1}{N}\sum_{i=1}^N(l_1(\va_i,b_i,\u)-\beta l_2(\va_i^\top\u,c_i,v))+\lambda\|\u\|^2-\gamma v^2,
\end{align}
where $(\vu,v)\in\BR^d\times\BR$ is the weights of the model, $\{(\va_i,b_i,c_i)\}_{i=1}^N$ is the training set such that $\va_i\in\RB^{d}$ contains all the input variables, $b_i\in\RB$ is the output, $c_i\in\RB$ is the input variable that we want to make it unbiased~\citep{lowd2005adversarial,zhang2018mitigating}, and~$\lambda,\beta,\gamma>0$ are the hyperparameters.
Here, the loss $l_1$ and $l_2$ are the logit functions, defined as ${\rm logit}(\va,b,\w)= \log(1+\exp(-b\va^\top\w))$.
The minimax optimization problem (\ref{eq:minimax_fair}) can be viewed as solving the square nonlinear equation $\mF(\vx)=\vzero$ such that
\begin{align}\label{eq:min-max}
\x = \begin{bmatrix}
\u \\ v    
\end{bmatrix}
\qquad\text{and}\qquad
\F(\x)\triangleq \begin{bmatrix}
~~\nabla_{\u}f(\u,v)\\
-\nabla_v f(\u,v)
\end{bmatrix}.
\end{align}

We compare our RQGN methods with the extragradient method (EG)~\citep{korpelevich1976extragradien,tseng1995linear}, which is the optimal first-order method for strongly-convex strongly-concave optimization.
For EG, we tune the step size from $\{0.01,0.05, 0.1, 0.5, 1\}$.
We set the parameters $\beta,\lambda$, and $\gamma$ as $0.5$, $0.001$, and $0.001$, respectively.
We perform the experiments on datasets ``heart'' ($N=270$, $d = 13$)~\citep{CC01a},  ``adult'' ($N= 32,561$, $d= 123$), and ``law school'' ($N=20,798$, $d= 380$)~\citep{quy2021survey}.
We report empirical results on iterations and CPU time in
Figure~\ref{fig:min-max}.
We observe that our RQGN methods converge much faster than EG.

\begin{figure}[!tbp]
\centering
\begingroup
\setlength{\tabcolsep}{1pt}
\renewcommand{\arraystretch}{0.96}
\begin{tabular}{@{}ccc@{}}
\includegraphics[width=0.327\textwidth,trim=11pt 6pt 11pt 6pt,clip]{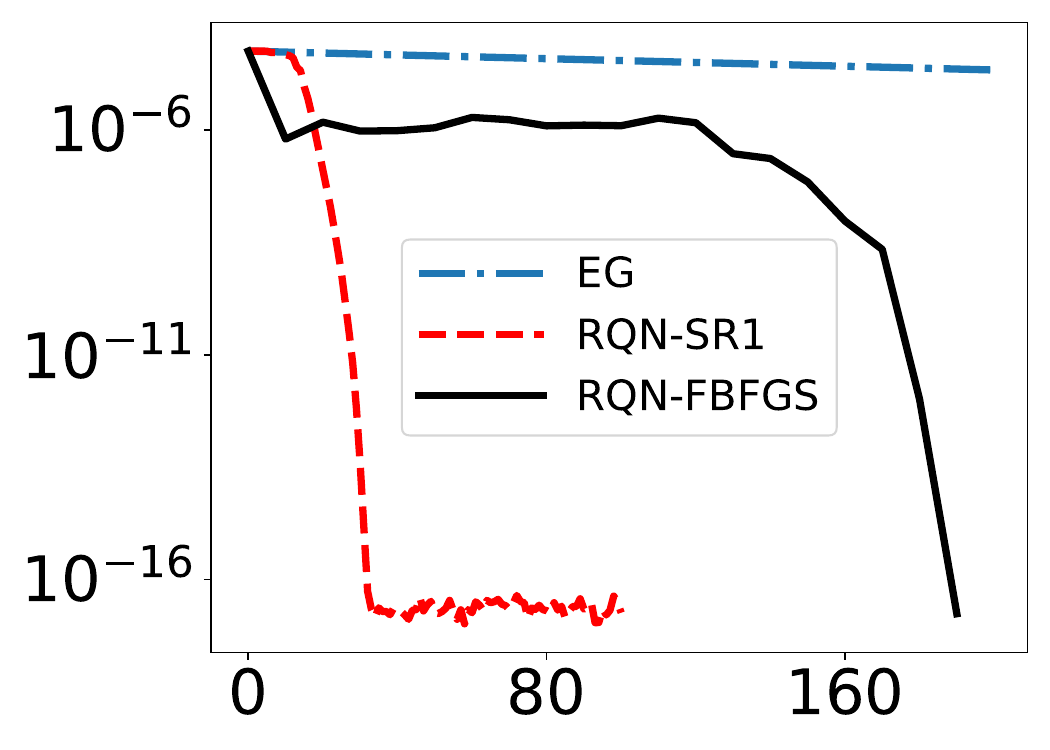} &
\includegraphics[width=0.327\textwidth,trim=11pt 6pt 11pt 6pt,clip]{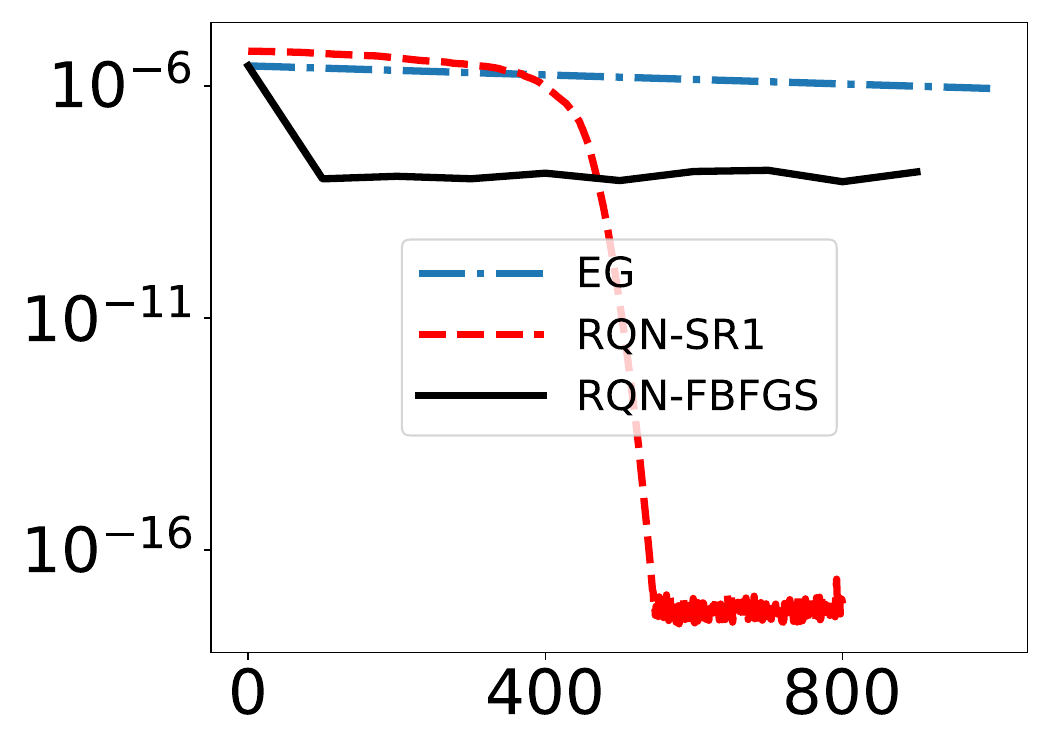} &
\includegraphics[width=0.327\textwidth,trim=11pt 6pt 11pt 6pt,clip]{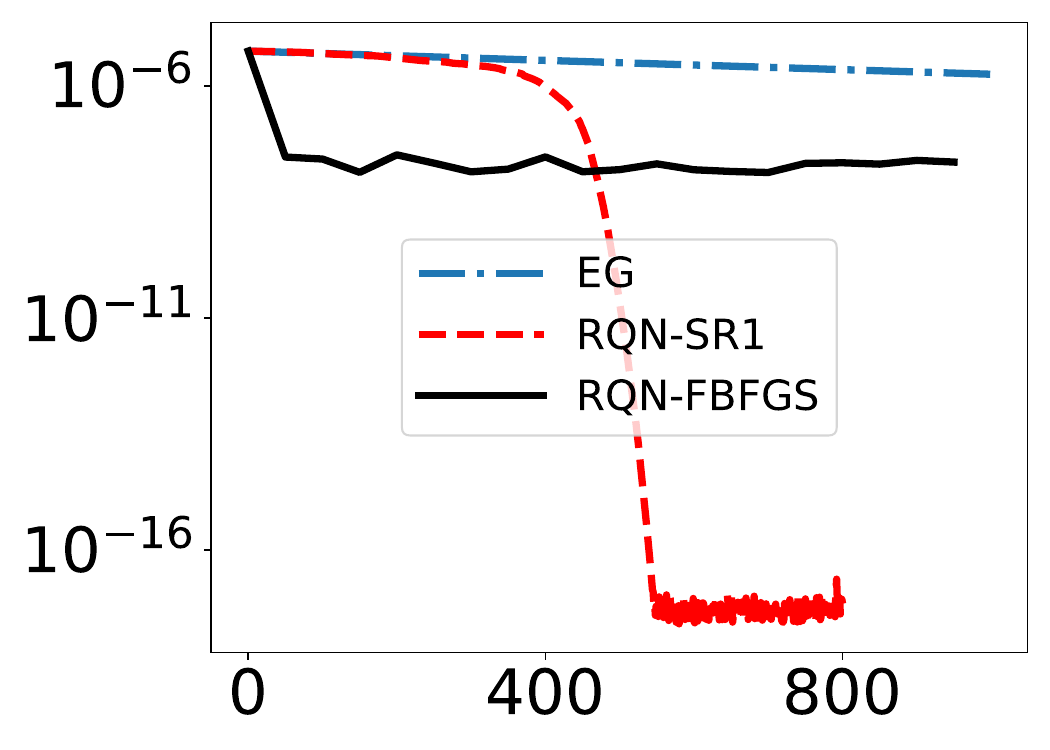} \\[-1pt]
{\footnotesize (a) ``heart'' (iteration)} & {\footnotesize (b) ``adult'' (iteration)} & {\footnotesize (c) ``law school'' (iteration)} \\[10pt]
\includegraphics[width=0.327\textwidth,trim=11pt 6pt 11pt 6pt,clip]{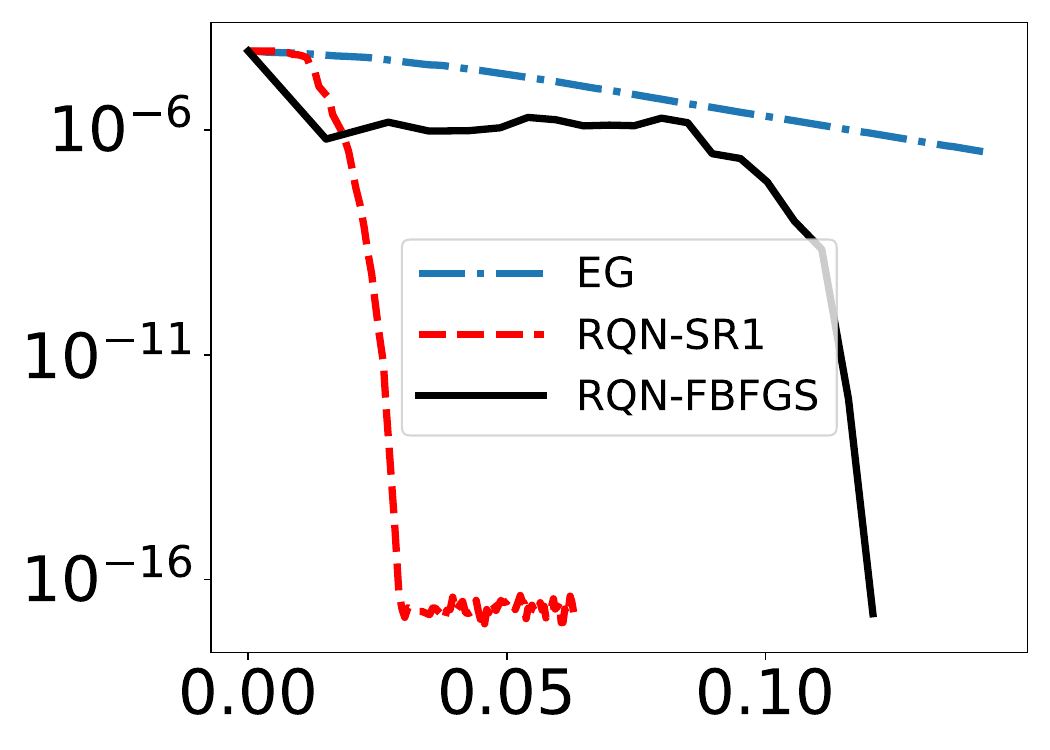} &
\includegraphics[width=0.327\textwidth,trim=11pt 6pt 11pt 6pt,clip]{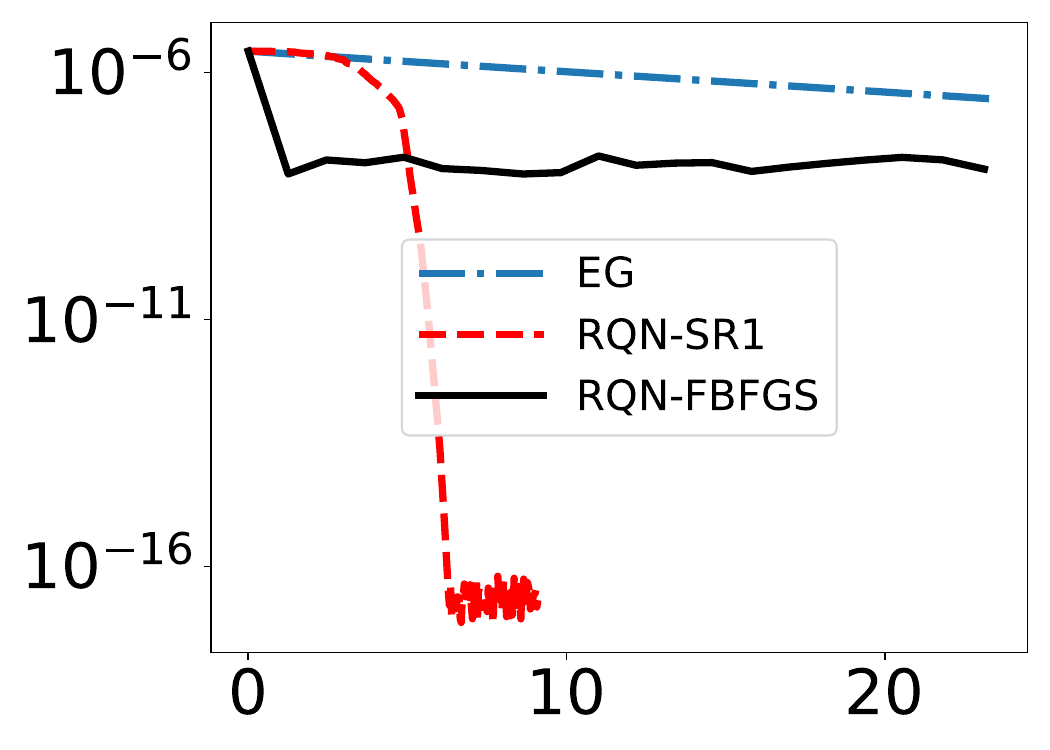} &
\includegraphics[width=0.327\textwidth,trim=11pt 6pt 11pt 6pt,clip]{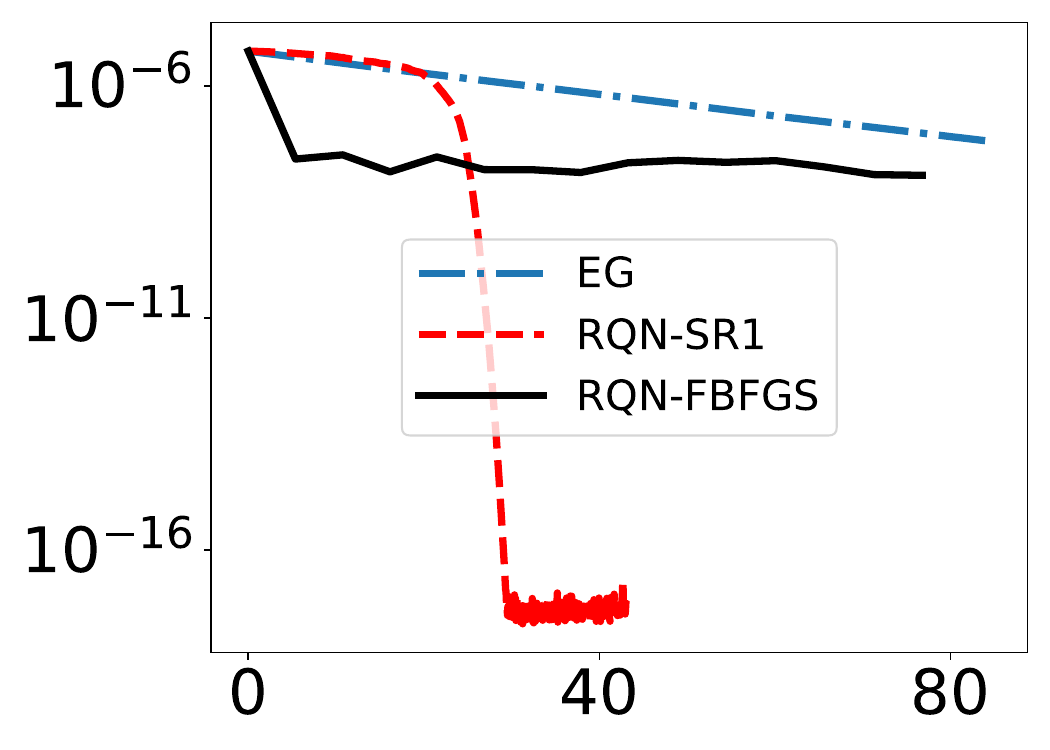} \\[-1pt]
{\footnotesize (d) ``heart'' (time)} & {\footnotesize (e) ``adult'' (time)} & {\footnotesize (f) ``law school'' (time)}
\end{tabular}
\endgroup
\caption{Empirical results of iteration counts and CPU time (seconds) against $\|\F(\x)\|$ for the square system with nonlinear mapping \eqref{eq:overfit}.}
\label{fig:min-max}
\end{figure}


\section{Conclusion and Future Work}
\label{sec:conclu}
In this paper, we propose randomized quasi-Gauss--Newton (RQGN) methods for solving both overdetermined and underdetermined nonlinear equations.
We establish local convergence guarantees to show that our RQGN enjoys condition-number-free linear rates for the overdetermined case and superlinear rates for the underdetermined equation.
Notably, our convergence results do not depend on the sufficiently accurate initial Jacobian estimator that is required by Broyden's methods.
The empirical results validate the efficiency of our proposed RQGN methods.
In future work, it is interesting to study the global convergence of quasi-Newton methods for general nonlinear equations, extending the existing results for convex optimization and monotone variational inequalities~\citep{jin2024non,jin2024nonasymptotic,rodomanov2024global,jiang2024online,jiang2024improved,kamzolov2023accelerated,agafonov2024exploring,wang2024global}. 


\appendix

\section{Proofs for Sections~\ref{sec:pre} and \ref{sec:rqn} }
\label{app:proofs_preliminaries}
We provide detailed proofs for the propositions in Sections~\ref{sec:pre} and \ref{sec:rqn}.

\subsection{Proof of Proposition~\ref{prop:gram_lip}}
\label{app:proof_prop_gram_lip}
\begin{proof}
The $L$-Lipschitz continuity of $\F(\,\cdot\,)$ follows directly from~equation \eqref{eq:lip}.
For the partial relation and the Lipschitz continuity of $\mH(\cdot)$, we focus on the case of $n\geq d$. 
Under Assumption~\ref{ass:Jlip}, we achieve
\begin{align*}
\|\H(\x)\|
&=\|\J(\x)^{\top}\J(\x)\|
\leq\|\J(\x)^{\top}\|\|\J(\x)\|
\leq L^2,
\end{align*}
which means $\H(\x)\preceq L^2\I_m$.
Besides, it also holds
\begin{align*}
    \|\H(\x)-\H(\y)\|\leq \big\|\J(\x)^{\top}\big\|\Norm{\J(\x)-\J(\y)}+\big\|\J(\x)^{\top}-\J(\y)^{\top}\big\|\Norm{\J(\y)}\leq 2LL_2\|\x-\y\|.
\end{align*}
The proof for the case $n<d$ is almost identical to that of $n\geq d$, so we omit it here.
\end{proof}

\subsection{Proof of Proposition~\ref{prop:gram_lower}}
\label{app:proof_prop_gram_lower}
\begin{proof}
Under Assumption~\ref{ass:nonsingular}, it holds
$\sigma_m(\H(\x))=\sigma_m^{2}(\J(\x))\geq\mu^2$,    
which implies $\mu^2\I_m\preceq\H(\x)$. 
\end{proof}


\subsection{Proof of Proposition~\ref{prop:Gram_concor}}
\label{app:proof_prop_gram_concordance}
\begin{proof}
From Proposition~\ref{prop:gram_lip}, we have
    \begin{align*}
    -2L_2L \|\x-\y\|\cdot\I_m \preceq   \H(\y)-\H(\x)\preceq 2L_2L \|\x-\y\|\cdot\I_m.
    \end{align*}
    By Proposition~\ref{prop:gram_lower}, we have
    \begin{align*}
     -\frac{2L_2L}{\mu^2}\|
    \x-\y\| \H(\y)\preceq   \H(\y)-\H(\x)\preceq \frac{2L_2L}{\mu^2}\|\x-\y\|\H(\x),
    \end{align*}
    rearranging above inequalities directly indicates equation \eqref{eq:Gram_bound_0}.
\end{proof}

\section{Proofs for Section~\ref{sec:local_convergence}}
\label{app:proofs}
We first provide some basic results for later proofs.

\begin{lemma}
\label{lm:matrix_order_equivalence}
Let $\A\in\BR^{m\times m}$ be positive definite and
$\B\in\BR^{m\times m}$ be symmetric.
For $a_1,a_2\in\BR$ such that $a_1\leq a_2$, the following two statements are equivalent: 
\begin{itemize}
\item $a_1\A\preceq\B\preceq a_2\A$;
\item $a_1\I_m\preceq\A^{-1/2}\B\A^{-1/2}\preceq a_2\I_m$
\end{itemize}
\end{lemma}
\begin{proof}
Suppose the first statement $a_1\A\preceq\B\preceq a_2\A$ holds.
For all $\v\in\BR^m$, applying this relation to vector $\A^{-1/2}\v$ indicates
\begin{align*}
a_1\|\v\|^2
&=a_1(\A^{-1/2}\v)^{\top}\A(\A^{-1/2}\v)\leq
(\A^{-1/2}\v)^{\top}\B(\A^{-1/2}\v)\\
&=\v^{\top}\A^{-1/2}\B\A^{-1/2}\v
\leq a_2(\A^{-1/2}\v)^{\top}\A(\A^{-1/2}\v)
=a_2\|\v\|^2.
\end{align*}
Hence, we achieve $a_1\I_m\preceq\A^{-1/2}\B\A^{-1/2}\preceq a_2\I_m$.

Conversely, suppose the second statement
$a_1\I_m\preceq\A^{-1/2}\B\A^{-1/2}\preceq a_2\I_m$ holds.
For all~$\v\in\BR^m$, applying this relation to vector $\A^{1/2}\v$ indicates
\begin{align*}
a_1\v^{\top}\A\v
&=a_1\|\A^{1/2}\v\|^2
\leq(\A^{1/2}\v)^{\top}\A^{-1/2}\B\A^{-1/2}(\A^{1/2}\v)\\
&=\v^{\top}\B\v
\leq a_2\|\A^{1/2}\v\|^2
=a_2\v^{\top}\A\v.
\end{align*}
Hence, we achieve
$a_1\A\preceq\B\preceq a_2\A$.
\end{proof}

We next state the following lemma.
\begin{lemma}[{\citet[Lemma~26]{lin2022explicit}}]
\label{lm:lin_lemma26}
Let $\{X_t\}_{t\geq0}$ be a sequence of nonnegative random variables.
Suppose that there exist constants $a\geq0$ and $\tau>1$ such that
\begin{align*}
\mathbb E[X_t]
&\leq a\biggl(1-\frac{1}{\tau}\biggr)^t
\end{align*}
for every $t\geq0$. Then, for all $p\in(0,1)$, with probability at
least $1-p$, we have
\begin{align*}
X_t
&\leq\frac{a\tau^2}{p}
\biggl(1-\frac{1}{1+\tau}\biggr)^t
\end{align*}
simultaneously for every $t\geq0$.
\end{lemma}

\subsection{Proof of Lemma~\ref{lm:Gram_approximate}}
\label{app:proof_lemma_gram_approximate}
\begin{proof}
We first use induction to prove 
\begin{align}\label{eq:induction-G-H}
    \H_t\preceq\G_t
\end{align}
holds for all $t\geq 0$.
The induction base in the case of $t=0$ directly holds by our assumption.
For the induction step, we suppose the partial relation \eqref{eq:induction-G-H} holds for all $t=\hat t$. 
Note that Proposition~\ref{prop:Gram_concor} implies
\begin{align}
\label{eq:gram_transport_lm}
\frac{1}{1+Mr_{\hat t}}\H_{\hat t}
\preceq \H_{{\hat t}+1}
\preceq(1+Mr_{\hat t})\H_{\hat t}.
\end{align}
Hence, the induction hypothesis $\H_{\hat t}\preceq\G_{\hat t}$ implies
\begin{align*}
\H_{{\hat t}+1}\overset{\eqref{eq:gram_transport_lm}}\preceq(1+Mr_{\hat t})\H_{\hat t}
\preceq(1+Mr_{\hat t})\G_{\hat t}=\widetilde{\G}_{\hat t}.
\end{align*}
Applying Lemma~\ref{lm:QN_order} to the update \eqref{eq:update_G}  yields $\H_{\hat t+1}\preceq\G_{\hat t+1}$, which finished the induction.

We then show the relation $\mG_t\preceq(1+\delta_t)\H_t$.
The relation \eqref{eq:induction-G-H} implies
\begin{align*}
\H_t^{-1/2}(\G_t-\H_t)\H_t^{-1/2}\succeq\0.
\end{align*}
Recall that all positive semi-definite matrix $\A\in\BR^{m\times m}$ satisfies
$\A\preceq\tr{\A}\I_m$. 
Applying this fact with $\mA=\H_t^{-1/2}(\G_t-\H_t)\H_t^{-1/2}$, we achieve
\begin{align}\label{eq:delta-GH-1}
\begin{split}    
\G_t
&=\H_t^{1/2}\Bigl(\I_m+\H_t^{-1/2}(\G_t-\H_t)\H_t^{-1/2}\Bigr)\H_t^{1/2}\\
&\preceq\H_t^{1/2}\big(\I_m+{\rm tr}\big(\H_t^{-1/2}(\G_t-\H_t)\H_t^{-1/2}\big)\I_m\big)\H_t^{1/2}\\
&=\big(1+{\rm tr}\big(\H_t^{-1/2}(\G_t-\H_t)\H_t^{-1/2}\big)\big)\H_t\\
&=\big(1+\tr{\H_t^{-1}(\G_t-\H_t)}\big)\H_t \\
&=\big(1+\sigma_{\H_t}(\G_t)\big)\H_t.
\end{split}
\end{align}
Moreover, the partial relation \eqref{eq:induction-G-H} indicates
\begin{align}\label{eq:G-H-I}
\G_t-\H_t\preceq\|\G_t-\H_t\|\I_m     
\end{align}
and
\begin{align}\label{eq:G-H-tr}
\|\G_t-\H_t\|\leq\tr{\G_t-\H_t}=\tau_{\H_t}(\G_t).     
\end{align}
Additionally, Propositions~\ref{prop:gram_lip} and~\ref{prop:gram_lower} imply
\begin{align}\label{eq:H-mu-L}
\mu^2\I_m\preceq\H_t\preceq L^2\I_m.    
\end{align}
Hence, it holds
\begin{align}\label{eq:delta-GH-2}
\begin{split}
\G_t
&\,=\,\H_t+(\G_t-\H_t)\\
&\overset{\eqref{eq:G-H-I}}\preceq\H_t+\|\G_t-\H_t\|\I_m\\
&\overset{\eqref{eq:H-mu-L}}\preceq\left(1+\frac{\|\G_t-\H_t\|}{\mu^2}\right)\H_t \\
&\overset{\eqref{eq:G-H-tr}}\preceq\Bigl(1+\frac{\tau_{\H_t}(\G_t)}{\mu^2}\Bigr)\H_t\\
&\,\preceq\left(1+
\frac{mL^2\tau_{\H_t}(\G_t)}
{\mu^2\tr{\H_t}}\right)\H_t,
\end{split}
\end{align}
where the last step is based on the fact
\begin{align*}
    \tr{\H_t} \overset{\eqref{eq:H-mu-L}}\leq \tr{L^2\mI_m} = mL^2.
\end{align*}
Combining the choices of $\delta_t$ in the statement and results
\eqref{eq:delta-GH-1} and \eqref{eq:delta-GH-2},
we achieve 
\begin{align*}
\H_t\preceq\G_t\preceq(1+\delta_t)\H_t     
\end{align*}
for each of Options I--III.

We now consider the upper bound for $\mathbb E_t[\delta_{t+1}]$.
For Options I and II, the choice of~$\delta_t=\sigma_{\H_t}(\G_t)$ gives
\begin{align}
\label{eq:defdelta_tequal}
    \tr{\H_t^{-1}\G_t}
&=\tr{\H_t^{-1}(\G_t-\H_t)}+\tr{\I_m}
=\delta_t+m.
\end{align}
Consequently, it holds
\begin{align}
    \label{eq:sigma_HttildeGt}
    \begin{split}
\sigma_{\H_{t+1}}(\widetilde{\G}_t)
&={\rm tr}\big(\H_{t+1}^{-1}\widetilde{\G}_t\big)-m\\
&=\tr{\H_{t+1}^{-1}\cdot(1+Mr_t)\G_t}-m\\
&=(1+Mr_t)\tr{\H_{t+1}^{-1}\G_t}-m\\
&\!\overset{\eqref{eq:Gram_bound_0}}{\leq}
(1+Mr_t)^2\tr{\H_t^{-1}\G_t}-m\\
&\!\!\overset{\eqref{eq:defdelta_tequal}}{=}(1+Mr_t)^2(\delta_t+m)-m\\
&=(1+Mr_t)^2\delta_t
  +m\bigl((1+Mr_t)^2-1\bigr)\\
&\leq(1+Mr_t)^2\delta_t
  +2mMr_t(1+Mr_t)^2\\
&=(1+Mr_t)^2(\delta_t+2mMr_t),
    \end{split}
\end{align}
where the last inequality is due to the fact that $(1+a)^2-1\leq2a(1+a)^2$ holds for all $a\geq 0$.
Finally, Lemma~\ref{lm:QN_one_iter} gives
\begin{align*}
\mathbb E_t[\delta_{t+1}]
&\leq(1-\gamma)\sigma_{\H_{t+1}}(\widetilde{\G}_t)\overset{\eqref{eq:sigma_HttildeGt}}{\leq} (1-\gamma)(1+Mr_t)^2(\delta_t+2mMr_t),
\end{align*}
where $\gamma=\mu^2/(mL^2)$ for Option I and 
$\gamma=1/m$ for Option II.
This proves~equation \eqref{eq:delta-recur} with $\tilde c=1$ for Options I and II.

For Option III, Lemma~\ref{lm:QN_one_iter} indicates
\begin{align}
    \label{eq:tau_Ht+1_G}
    \begin{split}
\tau_{\H_{t+1}}(\widetilde{\G}_t)
&={\rm tr}\big(\widetilde{\G}_t-\H_{t+1}\big)\\
&=\tr{(1+Mr_t)\G_t-\H_{t+1}}\\
&\!\overset{\eqref{eq:Gram_bound_0}}{\leq}
\tr{(1+Mr_t)\G_t-\frac{1}{1+Mr_t}\H_t}\\
&=(1+Mr_t)\tau_{\H_t}(\G_t)
 +\left(1+Mr_t-\frac{1}{1+Mr_t}\right)\tr{\H_t}\\
&\overset{\text{($\dag$)}}{\leq}(1+Mr_t)\tau_{\H_t}(\G_t)+2Mr_t\tr{\H_t}\\
&\leq(1+Mr_t)
 \left(\frac{\tau_{\H_t}(\G_t)}{\tr{\H_t}}+2Mr_t\right)
 \tr{\H_t}\\
&\overset{\eqref{eq:Gram_bound_0}}{\leq}(1+Mr_t)^2
 \left(\frac{\tau_{\H_t}(\G_t)}{\tr{\H_t}}+2Mr_t\right)
 \tr{\H_{t+1}},
\end{split}
\end{align}
where the step ($\dag$) is due to the fact that $1+a-1/(1+a)
\leq2a$ holds for all $a\geq 0$.
Consequently, Lemma~\ref{lm:QN_one_iter} gives
\begin{align}\label{eq:tau_t+1bound}
\begin{split}    
\mathbb E_t[\tau_{\H_{t+1}}(\G_{t+1})] 
&\leq \left(1-\frac1m\right) 
\tau_{\H_{t+1}}(\widetilde{\G}_t) \\
&\!\overset{\eqref{eq:tau_Ht+1_G}}{\leq} 
\left(1-\frac1m\right)(1+Mr_t)^2
\left(\frac{\tau_{\H_t}(\G_t)}{\tr{\H_t}}+2Mr_t\right)
\tr{\H_{t+1}}.
\end{split}
\end{align}
According the choice of $\delta_t={mL^2\tau_{\H_t}(\G_t)}/(\mu^2\tr{\H_t})$ for Option III, we achieve
\begin{align*}
\mathbb E_t[\delta_{t+1}]
&=\frac{mL^2}{\mu^2\tr{\H_{t+1}}}
\mathbb E_t[\tau_{\H_{t+1}}(\G_{t+1})]\\
&\!\!\overset{\eqref{eq:tau_t+1bound}}{\leq}\frac{mL^2}{\mu^2\tr{\H_{t+1}}}
\left(1-\frac1m\right)(1+Mr_t)^2
\left(\frac{\tau_{\H_t}(\G_t)}{\tr{\H_t}}+2Mr_t\right)
\tr{\H_{t+1}}\\
&=\left(1-\frac1m\right)(1+Mr_t)^2
\left(\frac{mL^2\tau_{\H_t}(\G_t)}
{\mu^2\tr{\H_t}}+\frac{2mL^2Mr_t}{\mu^2}\right)\\
&=\left(1-\frac1m\right)(1+Mr_t)^2
\left(\delta_t+\frac{2mL^2Mr_t}{\mu^2}\right),
\end{align*}
which proves equation~\eqref{eq:delta-recur} with $\gamma=1/m$ and $\tilde c=L^2/\mu^2$ for Option III.

Finally, applying Lemma~\ref{lm:matrix_order_equivalence} in the view of
$\H_0\preceq\G_0\preceq\eta_0\H_0$ gives
\begin{align*}  
\mI_m\preceq\H_0^{-1/2}\G_0\H_0^{-1/2}\preceq\eta_0\mI_m,
\end{align*}
that is
\begin{align*}
\0
&\preceq\H_0^{-1/2}(\G_0-\H_0)\H_0^{-1/2}
\preceq(\eta_0-1)\I_m.
\end{align*}
Therefore, we achieve
\begin{align*}
\sigma_{\H_0}(\G_0)
={\rm tr}\big(\H_0^{-1/2}(\G_0-\H_0)\H_0^{-1/2}\big)
\leq m(\eta_0-1),
\end{align*}
and
\begin{align*}
\frac{mL^2\tau_{\H_0}(\G_0)}
{\mu^2\tr{\H_0}}
=\frac{mL^2\tr{\G_0-\H_0}}{\mu^2\tr{\H_0}}
\leq\frac{mL^2(\eta_0-1)\tr{\H_0}}{\mu^2\tr{\H_0}}
=\frac{mL^2(\eta_0-1)}{\mu^2}.
\end{align*}
\end{proof}

\subsection{Proof of Lemma~\ref{lm:linear-quadra_n<d}}
\label{app:proof_lemma_linear_quadratic_under}
\begin{proof}
Recall that for the case of $n<d$,
we have defined   
\begin{align*}
\H(\x)=\J(\x)\J(\x)^{\top}
\qquad\text{and}\qquad
\lambda_{\mathrm{under}}(\x) =\sqrt{\F(\x)^{\top}(\H(\x))^{-1}\F(\x)}.
\end{align*}
For all $\vx$ satisfies condition \eqref{eq:G_condi_2},
it holds
\begin{align*}
\|\F(\x)\|
&=\big\|\H(\x)^{1/2}\H(\x)^{-1/2}\F(\x)\big\|
\leq\big\|\H(\x)^{1/2}\big\|\lambda_{\mathrm{under}}(\x)
\\&=\Norm{\J(\x)}\lambda_{\mathrm{under}}(\x)\overset{\eqref{eq:lip}}{\leq}L\lambda_{\mathrm{under}}(\x)
\overset{\eqref{eq:G_condi_2}}{\leq}\frac{\mu^2}{2L_2},
\end{align*}
which implies 
\begin{align}\label{eq:x-in-under}
 \x\in\Omega_{\mathrm{under}}   
\end{align}
Moreover, the condition $\H(\x)\preceq\G$  gives
\begin{align}
\label{eq:boundr_lambda}
\|\x_{+}-\x\|^2
&\overset{\eqref{eq:qNupdate}}{=}
\F(\x)^{\top}\G^{-1}\H(\x)\G^{-1}\F(\x)\notag\leq\F(\x)^{\top}\G^{-1}\F(\x)\\
&\leq\F(\x)^{\top}\H(\x)^{-1}\F(\x)
 =(\lambda_{\mathrm{under}}(\x))^2.
\end{align}
On the other hand, it holds
\begin{align}
\label{eq:Fx_+Fx}
\begin{split}
&\F(\x_{+})\\
&=\F(\x)+\int_0^1
\J\bigl(\x+z(\x_{+}-\x)\bigr)(\x_{+}-\x)\,{\rm d}z\\
&=\F(\x)+\J(\x)(\x_{+}-\x)+\int_0^1\Bigl(\J\bigl(\x+z(\x_{+}-\x)\bigr)-\J(\x)\Bigr)
(\x_{+}-\x)\,{\rm d}z\\
&\overset{\eqref{eq:qNupdate}}{=}
\F(\x)-\J(\x)\J(\x)^{\top}\G^{-1}\F(\x)+\int_0^1\Bigl(\J\bigl(\x+z(\x_{+}-\x)\bigr)-\J(\x)\Bigr)
(\x_{+}-\x)\,{\rm d}z\\
&\overset{\eqref{eq:H-Gram}}{=}
(\I_n-\H(\x)\G^{-1})\F(\x)+\R_F,
\end{split}
\end{align}
where we define
\begin{align*}
\R_F
\triangleq\int_0^1\Bigl(\J\bigl(\x+z(\x_{+}-\x)\bigr)-\J(\x)\Bigr)
(\x_{+}-\x)\,{\rm d}z.    
\end{align*}
We also have
\begin{align}
\label{eq:R_Fnormbound}
\begin{split}
\|\R_F\|
&\leq\int_0^1
\Norm{\J\bigl(\x+z(\x_{+}-\x)\bigr)-\J(\x)}
\|\x_{+}-\x\|\,{\rm d}z\\
&\overset{\eqref{eq:lip}}{\leq}
L_2\|\x_{+}-\x\|^2\int_0^1z\,{\rm d}z
=\frac{L_2}{2}\|\x_{+}-\x\|^2
\overset{\eqref{eq:boundr_lambda}}{\leq}
\frac{L_2}{2}(\lambda_{\mathrm{under}}(\x))^2.
\end{split}
\end{align}
The expansion of $\F(\x_{+})$ gives
\begin{align}
\label{eq:HFx+bound}
\H(\x)^{-1/2}\F(\x_{+})
&\overset{\eqref{eq:Fx_+Fx}}{=}\bigl(\I_n-\H(\x)^{1/2}\G^{-1}\H(\x)^{1/2}\bigr)
\H(\x)^{-1/2}\F(\x)+\H(\x)^{-1/2}\R_F.
\end{align}
The result \eqref{eq:x-in-under}  and Proposition~\ref{prop:gram_lower} indicate
\begin{align}\label{eq:H12mu}
\big\|\H(\x)^{-1/2}\big\|\leq\frac{1}{\mu}.    
\end{align}
Taking inverses in~\eqref{eq:G_condi_2} gives
\begin{align*}
\frac{1}{\eta}\H(\x)^{-1}
&\preceq\G^{-1}\preceq\H(\x)^{-1}.
\end{align*}
Applying Lemma~\ref{lm:matrix_order_equivalence} to above inequality with
$\A=\H(\x)^{-1}$ and $\B=\G^{-1}$ gives
\begin{align*}
\frac{1}{\eta}\I_n
&\preceq\H(\x)^{1/2}\G^{-1}\H(\x)^{1/2}\preceq\I_n.
\end{align*}
Therefore, we have
\begin{align*}
\0
&\preceq\I_n-\H(\x)^{1/2}\G^{-1}\H(\x)^{1/2}
\preceq\left(1-\frac{1}{\eta}\right)\I_n,
\end{align*}
which implies
\begin{align}
\label{eq:normI-HGH}
\Norm{\I_n-\H(\x)^{1/2}\G^{-1}\H(\x)^{1/2}}
&\leq1-\frac{1}{\eta}.
\end{align}
Combining these bounds with the expansion of $\F(\x_{+})$ gives
\begin{align}
\label{eq:HFx+}
\Norm{\H(\x)^{-1/2}\F(\x_{+})}
&\!\!\overset{\eqref{eq:HFx+bound}}{\leq}\Norm{\I_n-\H(\x)^{1/2}\G^{-1}\H(\x)^{1/2}}
\Norm{\H(\x)^{-1/2}\F(\x)}
+\Norm{\H(\x)^{-1/2}}\|\R_F\|\notag\\
&\!\!\overset{\eqref{eq:H12mu}}\leq\Norm{\I_n-\H(\x)^{1/2}\G^{-1}\H(\x)^{1/2}}
\lambda_{\mathrm{under}}(\x)+\frac{\|\R_F\|}{\mu}\notag\\
&\!\!\!\!\!\!\overset{\eqref{eq:normI-HGH},\,\eqref{eq:R_Fnormbound}
}{\leq}\left(1-\frac{1}{\eta}\right)\lambda_{\mathrm{under}}(\x)
  +\frac{L_2}{2\mu}(\lambda_{\mathrm{under}}(\x))^2
\end{align}
and
\begin{align*}
\|\F(\x_{+})\|
&=\Norm{\H(\x)^{1/2}\H(\x)^{-1/2}\F(\x_{+})}\\
&\leq L\Norm{\H(\x)^{-1/2}\F(\x_{+})}\\
&\overset{\eqref{eq:HFx+}}{\leq}
L\lambda_{\mathrm{under}}(\x)
+\frac{LL_2}{2\mu}(\lambda_{\mathrm{under}}(\x))^2
\overset{\eqref{eq:G_condi_2}}{\leq}
\frac{\mu^2}{4L_2}+\frac{\mu^3}{32LL_2}
\leq\frac{\mu^2}{2L_2},
\end{align*}
where the first inequality is based on Proposition \ref{prop:gram_lip}.
Thus, both $\x$ and $\x_{+}$ belong to $\Omega_{\mathrm{under}}$.

Furthermore, the definition of $\lambda_{\mathrm{under}}(\cdot)$ gives
\begin{align*}
\lambda_{\mathrm{under}}(\x_{+})
&=\sqrt{\F(\x_{+})^{\top}\H(\x_{+})^{-1}\F(\x_{+})}\\
&\overset{\eqref{eq:Gram_bound_0}}{\leq}\sqrt{1+M\|\x_{+}-\x\|}
\sqrt{\F(\x_{+})^{\top}\H(\x)^{-1}\F(\x_{+})} \\
&=
\sqrt{1+M\|\x_{+}-\x\|} \,\big\|\H(\x)^{-1/2}\F(\x_{+})\big\|\\
&\!\!\!\!\!\!\!\overset{\eqref{eq:boundr_lambda},\,\eqref{eq:HFx+}}{\leq}
\sqrt{1+M\lambda_{\mathrm{under}}(\x)}\,
 \left(\left(1-\frac{1}{\eta}\right)\lambda_{\mathrm{under}}(\x)
 +\frac{L_2}{2\mu}(\lambda_{\mathrm{under}}(\x))^2\right)\\
&\leq(1+M\lambda_{\mathrm{under}}(\x))
 \left(\left(1-\frac{1}{\eta}\right)\lambda_{\mathrm{under}}(\x)
 +\frac{L_2}{2\mu}(\lambda_{\mathrm{under}}(\x))^2\right)\\
&=\left(1-\frac{1}{\eta}\right)\lambda_{\mathrm{under}}(\x)
+\left(\left(1-\frac{1}{\eta}\right)M+\frac{L_2}{2\mu}
+\frac{ML_2}{2\mu}\lambda_{\mathrm{under}}(\x)\right)
(\lambda_{\mathrm{under}}(\x))^2\\
&\!\!\overset{\eqref{eq:G_condi_2}}{\leq}\!\!
\left(1-\frac{1}{\eta}\right)\lambda_{\mathrm{under}}(\x)
+\left(\frac{2LL_2}{\mu^2}+\frac{L_2}{2\mu}
+\frac{L_2}{4\mu}\right)(\lambda_{\mathrm{under}}(\x))^2\\
&\leq\left(1-\frac{1}{\eta}\right)\lambda_{\mathrm{under}}(\x)
 +\frac{3LL_2}{\mu^2}(\lambda_{\mathrm{under}}(\x))^2,
\end{align*}
where $M=2LL_2/\mu^2$ follows Proposition \ref{prop:Gram_concor}.
This finishes the proof.
\end{proof}

\subsection{Proof of Theorem~\ref{thm:linear_under}}
\label{app:proof_theorem_linear_under}
\begin{proof}
For the ease of presentation, we denote
\begin{align*}
\lambda_{\mathrm{under},t}\triangleq\lambda_{\mathrm{under}}(\x_t).    
\end{align*}
We use induction to prove  $\x_t\in\Omega_{\mathrm{under}}$ and equation \eqref{eq:induc_lambda_n<d} holds for all $t\geq 0$.
For the induction base, the initialization in Line~\ref{line:input} of
Algorithm~\ref{alg:RQGN} and the definition of $\eta_0$ in equation \eqref{eq:dfn-eta0} give~$\H_0\preceq\G_0\preceq\eta_0\H_0$, which proves \eqref{eq:induc_lambda_n<d} in the case of $t=0$.
Additionally, condition~\eqref{eq:linear_initial_condi_2} gives
\begin{align*}
\lambda_{\mathrm{under},0}
&\leq\frac{\mu^2}{20\eta_0LL_2}
\leq\frac{\mu^2}{4LL_2}.
\end{align*}
Thus, applying Lemma~\ref{lm:linear-quadra_n<d} with $\eta=\eta_0$ gives $\x_0\in\Omega_{\mathrm{under}}$. 

For the induction step, we suppose  $\x_t\in\Omega_{\mathrm{under}}$ and equation~\eqref{eq:induc_lambda_n<d} hold for all $t\leq \hat t$.
Since the induction hypothesis implies $\lambda_{\mathrm{under},\hat t}\leq\lambda_{\mathrm{under},0}$, Lemma~\ref{lm:linear-quadra_n<d} implies 
\begin{align}\label{eq:r_hat_t_bound}
\x_{\hat t+1}\in\Omega_{\mathrm{under}}
\qquad\text{and}\qquad
r_{\hat t}\leq \lambda_{\mathrm{under},\hat t}
\end{align}
Furthermore, applying Proposition~\ref{prop:Gram_concor} to $\x_{\hat t}$ and
$\x_{\hat t+1}$ gives
\begin{align*}
\H_{\hat t+1}
&\overset{\eqref{eq:Gram_bound_0}}{\preceq}
(1+Mr_{\hat t})\H_{\hat t}
\overset{\eqref{eq:induc_lambda_n<d}}{\preceq}
(1+Mr_{\hat t})\G_{\hat t}
=\widetilde{\G}_{\hat t}
\end{align*}
and
\begin{align*}
\widetilde{\G}_{\hat t}
&=(1+Mr_{\hat t})\G_{\hat t}\\
&\overset{\eqref{eq:induc_lambda_n<d}}{\preceq}
\eta_0(1+Mr_{\hat t})
\exp\biggl(2\sum_{i=0}^{\hat t-1}Mr_i\biggr)\H_{\hat t}\\
&\overset{\eqref{eq:Gram_bound_0}}{\preceq}
\eta_0(1+Mr_{\hat t})^2
\exp\biggl(2\sum_{i=0}^{\hat t-1}Mr_i\biggr)\H_{\hat t+1}\\
&\preceq\eta_0\exp\biggl(2\sum_{i=0}^{\hat t}Mr_i\biggr)
\H_{\hat t+1},
\end{align*}
where the last step is based on the fact $1+x\leq {\rm e}^{x}$ with $x=Mr_{\hat{t}}$.
The above two results imply we can apply Lemma~\ref{lm:QN_order} with $\mG=\widetilde{\G}_{\hat t}$, $\mH=\H_{\hat t+1}$, and $\eta=\eta_0\exp\big(2\sum_{i=0}^{\hat t}Mr_i\big)$ to~achieve
\begin{align}
    \label{eq:G_tH_tleq2}
\H_{{\hat t}+1} &\preceq\G_{{\hat t}+1}
 \preceq\eta_0\exp\biggl(2\sum_{i=0}^{\hat t}Mr_i\biggr)\H_{{\hat t}+1}.
\end{align}
We also have
\begin{align}
    \label{eq:sum_r_t_bound}
\sum_{i=0}^{t-1}r_i
&\overset{\eqref{eq:r_hat_t_bound}}{\leq}
\sum_{i=0}^{t-1}\lambda_{\mathrm{under},i}
\overset{\eqref{eq:induc_lambda_n<d}}{\leq}
 \sum_{i=0}^{t-1}\biggl(1-\frac{1}{2\eta_0}\biggr)^i
 \lambda_{\mathrm{under},0}
 \leq2\eta_0\lambda_{\mathrm{under},0}
 \overset{\eqref{eq:linear_initial_condi_2}}{\leq}\frac{1}{5M},
\end{align}
where $M=2LL_2/\mu^2$.
Therefore, it holds
    \begin{align*}
\H_{{\hat t}+1}
&\preceq\G_{{\hat t}+1}
\overset{\eqref{eq:G_tH_tleq2}}{\preceq}
\eta_0\exp\biggl(2M\sum_{i=0}^{\hat t}r_i\biggr)\H_{\hat t+1}
\overset{\eqref{eq:sum_r_t_bound}}{\preceq}
\eta_0{\rm e}^{2/5}\H_{\hat t+1}
\preceq\frac{3\eta_0}{2}\H_{\hat t+1}.
\end{align*}
Additionally, Lemma~\ref{lm:linear-quadra_n<d} implies
\begin{align*}
\lambda_{\mathrm{under},\hat t+1}
&\overset{\eqref{eq:linear-quadra-n<d}}{\leq}
\biggl(1-\frac{2}{3\eta_0}
+\frac{3LL_2}{\mu^2}\lambda_{\mathrm{under},\hat t}\biggr)
\lambda_{\mathrm{under},\hat t}\overset{\eqref{eq:induc_lambda_n<d},\,\eqref{eq:linear_initial_condi_2}}{\leq}
\biggl(1-\frac{2}{3\eta_0}+\frac{3}{20\eta_0}\biggr)
\lambda_{\mathrm{under},\hat t}\\
&~\leq\biggl(1-\frac{1}{2\eta_0}\biggr)
\lambda_{\mathrm{under},\hat t}
\overset{\eqref{eq:induc_lambda_n<d}}{\leq}
\biggl(1-\frac{1}{2\eta_0}\biggr)^{\hat t+1}
\lambda_{\mathrm{under},0}.
\end{align*}
This completes the induction.
\end{proof}

\subsection{Proof of Theorem~\ref{thm:n<dsuperlinear}}
\label{app:proof_theorem_superlinaer_under}

\begin{proof} 
We follow the notation $\lambda_{\mathrm{under},t}=\lambda_{\mathrm{under}}(\x_t)$ defined in Appendix
\ref{app:proof_theorem_linear_under}.
Recall that the definitions of $\tilde c$ and $\gamma$ in Lemma \ref{lm:Gram_approximate} satisfy
 $\tilde c n\geq1$ and $1-\gamma<1$, then
 condition~\eqref{eq:initial_condi_n<d} implies equation \eqref{eq:linear_initial_condi_2} holds.
Thus, Theorem~\ref{thm:linear_under} holds naturally that we have $\x_t\in\Omega_{\mathrm{under}}$ and
\begin{align}
\label{eq:thm_linear_condi_imply}
    \lambda_{\mathrm{under},t}\leq \left(1-\frac{1}{2\eta_0}\right)^{t}\lambda_{\mathrm{under},0}\leq \lambda_{\mathrm{under},0}\leq \frac{\mu^2}{4LL_2}
\end{align}
Therefore, Lemma~\ref{lm:Gram_approximate} gives
\begin{align}\label{eq:HGdelta}
    \H_t\preceq \G_t\preceq (1+\delta_t)\H_t.
\end{align}
Hence, we apply Lemma~\ref{lm:linear-quadra_n<d} with $\vx = \vx_t$, $\vx_+=\vx_{t+1}$, and $\eta=1+\delta_t$ to achive
\begin{align}
    \label{eq:r_tboundedbylambda_t}
r_t=\norm{\vx_{t+1}-\vx_t}\leq\lambda_{\mathrm{under},t}
\end{align}
and
\begin{align}
\label{eq:exp_lambda_n<d}
\lambda_{\mathrm{under},t+1}
&\!\!\overset{\eqref{eq:linear-quadra-n<d}}{\leq}
\biggl(1-\frac{1}{1+\delta_t}\biggr)\lambda_{\mathrm{under},t}
+\frac{3LL_2}{\mu^2}(\lambda_{\mathrm{under},t})^2\notag\\
& \leq\delta_t\lambda_{\mathrm{under},t}
+\frac{3LL_2}{\mu^2}(\lambda_{\mathrm{under},t})^2.
\end{align}
According to Lemma~\ref{lm:Gram_approximate} with $m=n$, it holds
\begin{align}
\label{eq:exp_delta_n<d}
\begin{split}    
\mathbb E_t[\delta_{t+1}]
&\!\!\overset{\eqref{eq:delta-recur}}{\leq}
(1-\gamma)(1+Mr_t)^2(\delta_t+2\tilde c nMr_t) \\
&\!\!\overset{\eqref{eq:r_tboundedbylambda_t}}{\leq}
(1-\gamma)(1+M\lambda_{\mathrm{under},t})^2
 (\delta_t+2\tilde c nM\lambda_{\mathrm{under},t}).
 \end{split}
\end{align}
Let
\begin{align}\label{eq:dfn-xi-t}
\xi_t=\delta_t + \frac{7\tilde c nLL_2\lambda_{\mathrm{under},t}}{\mu^2}.    
\end{align}
Recall that we have $M=2LL_2/\mu^2$ and $\tilde c n\geq1$, which implies
$2\tilde c nM\leq7\tilde c nLL_2/\mu^2$ and~$3LL_2/\mu^2\leq7\tilde c nLL_2/\mu^2$.
This leads to
\begin{align}
    \label{eq:xi_t_lower_bound}
\delta_t+2\tilde{c}nM\lambda_{\mathrm{under},t}\leq\xi_t\qquad\text{and}\qquad\delta_t+ \frac{3LL_2}{\mu^2}\lambda_{\mathrm{under},t}\leq\xi_t.
\end{align}
Combining equations~\eqref{eq:exp_lambda_n<d} and \eqref{eq:exp_delta_n<d}, we have
\begingroup
\small
\begin{align}
\label{eq:xi_t_one_iter_n<d}
\begin{split}    
&\mathbb E_t[\xi_{t+1}]
\\&\overset{\eqref{eq:dfn-xi-t}}{=}
\BE_t\Bigl[\delta_{t+1}
+\frac{7\tilde c nLL_2}{\mu^2}\lambda_{\mathrm{under},t+1}\Bigr]\\
&\!\!\!\!\!\overset{\eqref{eq:exp_delta_n<d},\,\eqref{eq:exp_lambda_n<d}}{\leq}
(1-\gamma)(1+M\lambda_{\mathrm{under},t})^2
\bigl(\delta_t+2\tilde{c}nM\lambda_{\mathrm{under},t}\bigr)+\frac{7\tilde c nLL_2}{\mu^2}
\left(\delta_t+\frac{3LL_2}{\mu^2}\lambda_{\mathrm{under},t}\right)
\lambda_{\mathrm{under},t}\\
&\overset{\eqref{eq:xi_t_lower_bound}}{\leq}
(1-\gamma)(1+M\lambda_{\mathrm{under},t})^2\xi_t
 +\frac{7\tilde c nLL_2}{\mu^2}\lambda_{\mathrm{under},t}\,\xi_t \\
&~\leq(1-\gamma)\biggl((1+M\lambda_{\mathrm{under},t})^2
+\frac{7\tilde c nLL_2}{(1-\gamma)\mu^2}
\lambda_{\mathrm{under},t}\biggr)\xi_t\\
&~\leq(1-\gamma)
\exp\Biggl(\biggl(2M+\frac{7\tilde c nLL_2}
{(1-\gamma)\mu^2}\biggr)\lambda_{\mathrm{under},t}\Biggr)\xi_t \\
&~\leq(1-\gamma)
 \exp\biggl(\frac{11\tilde c nLL_2}{(1-\gamma)\mu^2}\lambda_{\mathrm{under},t}\biggr)\xi_t,
\end{split}
\end{align}
\endgroup
where we use the fact $(1+a)^2+b\leq (1+a)^2(1+b)\leq \exp(2a+b)$ for all $a,b\geq 0$ by taking~$a = (1+M\lambda_{\mathrm{under},t})\geq 0$ and $b = 7\tilde{c}nLL_2\lambda_{\mathrm{under},t}/((1-\gamma)\mu^2)\geq 0$ in the second last inequality and the definition $M=2LL_2/\mu^2\leq 2\tilde{c}n/((1-\gamma)\mu^2)$ in the last inequality.

Taking expectation on~\eqref{eq:xi_t_one_iter_n<d} and applying equation \eqref{eq:thm_linear_condi_imply} gives
\begin{align}
\label{eq:recur_xi_t_n<d}
\EBP{\xi_{t+1}}
&\overset{\eqref{eq:xi_t_one_iter_n<d},\, \eqref{eq:thm_linear_condi_imply}}{\leq}
(1-\gamma)
 \exp\Biggl(\frac{11\tilde c nLL_2}{(1-\gamma)\mu^2}
 \biggl(1-\frac{1}{2\eta_0}\biggr)^t\lambda_{\mathrm{under},0}\Biggr)
 \EBP{\xi_t}.
\end{align}
Recall that the geometric series satisfies
\begin{align}
\label{eq:geosum_boun}
    \sum_{i=0}^{t-1}\biggl(1-\frac{1}{2\eta_0}\biggr)^i
&\leq\sum_{i=0}^{\infty}\biggl(1-\frac{1}{2\eta_0}\biggr)^i
=2\eta_0.
\end{align}
Hence, the recursion~\eqref{eq:recur_xi_t_n<d} gives
\begin{align}
\label{eq:xi_t_bound_n<d}
\EBP{\xi_t}
&\overset{\eqref{eq:recur_xi_t_n<d}}{\leq}(1-\gamma)^t
 \prod_{i=0}^{t-1}
 \exp\Biggl(\frac{11\tilde c nLL_2}{(1-\gamma)\mu^2}
 \biggl(1-\frac{1}{2\eta_0}\biggr)^i
 \lambda_{\mathrm{under},0}\Biggr)\xi_0\notag\\
&=(1-\gamma)^t
 \exp\Biggl(\frac{11\tilde c nLL_2}{(1-\gamma)\mu^2}
 \lambda_{\mathrm{under},0}
 \sum_{i=0}^{t-1}\biggl(1-\frac{1}{2\eta_0}\biggr)^i\Biggr)
 \xi_0\notag\\
&\overset{\eqref{eq:geosum_boun}}{\leq}(1-\gamma)^t
 \exp\biggl(\frac{22\tilde c nLL_2\eta_0}{(1-\gamma)\mu^2}\lambda_{\mathrm{under},0}\biggr)
 \xi_0\notag\\
&\overset{\eqref{eq:initial_condi_n<d}}{\leq}
(1-\gamma)^t {\rm e}^{11/10}
 \biggl(\delta_0+\frac{7(1-\gamma)}{20\eta_0}\biggr)
 =(1-\gamma)^tq_{\text{\rm under}}.
\end{align}
Additionally, the definition of $\xi_t$ gives
\begin{align}\label{eq:lambda-ratio}
\frac{\lambda_{\mathrm{under},t+1}}{\lambda_{\mathrm{under},t}}
\overset{\eqref{eq:exp_lambda_n<d}}{\leq}\delta_t+\frac{3LL_2\lambda_{\mathrm{under},t}}{\mu^2}\leq\delta_t+\frac{7\tilde c nLL_2\lambda_{\mathrm{under},t}}{\mu^2}
=\xi_t,
\end{align}
The definition of $\eta_t$ in equation \eqref{eq:def_eta} and equation \eqref{eq:HGdelta} indicates 
\begin{align}\label{eq:eta-xi}
\eta_t\leq\delta_t+1\leq\xi_t+1,     
\end{align}
where the last step is based on equation \eqref{eq:lambda-ratio}.
Therefore, we have
\begin{align}\label{eq:lambda_{t+1}/lambda_t}
\EBP{\eta_t-1}
\overset{\eqref{eq:eta-xi}}\leq\EBP{\xi_t}
\overset{\eqref{eq:xi_t_bound_n<d}}{\leq}
(1-\gamma)^tq_{\text{\rm under}}
~~\text{and}~~
\EBP{\frac{\lambda_{\mathrm{under},t+1}}{\lambda_{\mathrm{under},t}}}\leq\EBP{\xi_t}
\overset{\eqref{eq:xi_t_bound_n<d}}{\leq}
(1-\gamma)^tq_{\text{\rm under}}.
\end{align}
Combing equation \eqref{eq:lambda_{t+1}/lambda_t} and Lemma \ref{lm:lin_lemma26} with
$X_t=\lambda_{\mathrm{under},t+1}/\lambda_{\mathrm{under},t}$,
$a=q_{\text{\rm under}}$, and $\tau=1/\gamma>1$, with probability at least $1-p$, we have
\begin{align}
\label{eq:lambda_t+1-recurn<d}
\frac{\lambda_{\mathrm{under},t+1}}{\lambda_{\mathrm{under},t}}
&\leq \biggl(1-\frac{\gamma}{1+\gamma}\biggr)^t\frac{q_{\text{\rm under}}}{p\gamma^2}
\end{align}
simultaneously for every $t\geq0$.
Recall that $\gamma\in(0,1)$ and $\log(1+\gamma)\geq\gamma/2$, then taking
\begin{align*}
t_0
=\biggl\lceil
\frac{\log\bigl(1+2q_{\text{\rm under}}/(p\gamma^2)\bigr)}{\log(1+\gamma)}
\biggr\rceil=\OM\biggl(\frac{\log\bigl(1+q_{\text{\rm under}}/(p\gamma^2)\bigr)}{\gamma}\biggr)
\end{align*}
gives
\begin{align}\label{eq:superlinearleq1/2n<d}
\biggl(1-\frac{\gamma}{1+\gamma}\biggr)^{t_0}\frac{q_{\text{\rm under}}}{p\gamma^2}
&\leq\frac{1}{2}.
\end{align}
Therefore, for all $t\geq1$, we have
\begin{align*}
\lambda_{\mathrm{under},t_0+t}
&=\lambda_{\mathrm{under},t_0}
\prod_{i=0}^{t-1}
\frac{\lambda_{\mathrm{under},t_0+i+1}}
{\lambda_{\mathrm{under},t_0+i}}\\
&\!\!\!\!\!\!\!\overset{\eqref{eq:lambda_t+1-recurn<d},\, \eqref{eq:superlinearleq1/2n<d}}{\leq}\lambda_{\mathrm{under},t_0}
\prod_{i=0}^{t-1}
\biggl(\frac{1}{2}
\biggl(1-\frac{\gamma}{1+\gamma}\biggr)^i\biggr)\\
&=\biggl(1-\frac{\gamma}{1+\gamma}\biggr)^{\sum_{i=0}^{t-1}i}
\biggl(\frac{1}{2}\biggr)^t\lambda_{\mathrm{under},t_0}\\
&=\biggl(1-\frac{\gamma}{1+\gamma}\biggr)^{t(t-1)/2}
\biggl(\frac{1}{2}\biggr)^t\lambda_{\mathrm{under},t_0}\\
&\!\!\overset{\eqref{eq:induc_lambda_n<d}}{\leq}
\biggl(1-\frac{\gamma}{1+\gamma}\biggr)^{t(t-1)/2}
 \biggl(\frac{1}{2}\biggr)^t
 \biggl(1-\frac{1}{2\eta_0}\biggr)^{t_0}\lambda_{\mathrm{under},0}
\end{align*}
holds with probability at least $1-p$.
\end{proof}

\subsection{Proof of Proposition~\ref{prop:local_nonsingularity}}
\label{app:proof_prop_local_nonsingularity}
\begin{proof}
For all $\x\in\Omega_{\mathrm{over}}$, the definition of
$\Omega_{\mathrm{over}}$ and Assumption~\ref{ass:nonsingular_J^*} indicate
\begin{align}
    \label{eq:local_nonsingularity_condi}
\|\x-\x_*\|
&\leq\frac{\mu_*^2}{4LL_2}
\qquad\text{and}\qquad
\H(\x_*)=\J(\x_*)^{\top}\J(\x_*)\succeq\mu_*^2\I_d.
\end{align}
Since Proposition~\ref{prop:gram_lip} gives $\H(\cdot)$ is $2LL_2$-Lipschitz continuous, it holds
\begin{align*}
\H(\x)
\succeq\H(\x_*)-2LL_2\|\x-\x_*\|\I_d\overset{\eqref{eq:local_nonsingularity_condi}}{\succeq}\mu_*^2\I_d-2LL_2\cdot\frac{\mu_*^2}{4LL_2}\I_d
=\frac{\mu_*^2}{2}\I_d.
\end{align*}
Thus, Assumption~\ref{ass:nonsingular} holds with $\Omega=\Omega_{\mathrm{over}}$ and
$\mu=\mu_*/\sqrt{2}$.
\end{proof}

\subsection{Proof of Proposition~\ref{prop:assumption-sati}}
\label{app:proof_prop}
\begin{proof}
We consider three conditions in this proposition as follows.
\paragraph{Case 1:}
 For all $\vx\in\BR^d$ such that $\|\x-\x_*\|\leq\delta$,  Proposition~\ref{prop:gram_lip} gives
\begin{align}
    \label{eq:F_i_bound}
    \begin{split}
\sum_{i=1}^n|F_i(\x)|
&\leq\|\F(\x_*)\|_1+\|\F(\x)-\F(\x_*)\|_1\\
&\leq\|\F(\x_*)\|_1+\sqrt n\|\F(\x)-\F(\x_*)\|\\
&\overset{\eqref{eq:lip}}{\leq}\|\F(\x_*)\|_1+\sqrt nL\delta.
    \end{split}
\end{align}
Using the fact $\nabla^2\phi(\x)=\H(\x)+\sum_{i=1}^nF_i(\x)\nabla^2F_i(\x)$, we achieve
\begin{align*}
&\nabla^2\phi(\x)-\nabla^2\phi(\x_*)\\
&=\H(\x)-\H(\x_*)
+\sum_{i=1}^n\bigl(F_i(\x)-F_i(\x_*)\bigr)\nabla^2F_i(\x_*)+\sum_{i=1}^nF_i(\x)
\bigl(\nabla^2F_i(\x)-\nabla^2F_i(\x_*)\bigr)
\end{align*}
which implies
\begin{align}
\label{eq:nabla_2phi_bound}
\begin{split}    
&\|\nabla^2\phi(\x)-\nabla^2\phi(\x_*)\|\\
&\leq\|\H(\x)-\H(\x_*)\|
+\sum_{i=1}^n|F_i(\x)-F_i(\x_*)|\,\|\nabla^2F_i(\x_*)\| \\
&\quad+\sum_{i=1}^n|F_i(\x)|\,
\|\nabla^2F_i(\x)-\nabla^2F_i(\x_*)\|\\
&\overset{\eqref{eq:lip}}{\leq}\left(2LL_2+L\sum_{i=1}^n\|\nabla^2F_i(\x_*)\|
+L_3\sum_{i=1}^n|F_i(\x)|\right)\|\x-\x_*\|\\
&\overset{\eqref{eq:F_i_bound}}{\leq}\left(2LL_2+L\sum_{i=1}^n\|\nabla^2F_i(\x_*)\|
+L_3\sqrt nL\delta+L_3\|\F(\x_*)\|_1\right)\|\x-\x_*\|\\
&=2\rho\|\x-\x_*\|,
\end{split}
\end{align}
where $\rho$ follows equation \eqref{eq:rho-cond-1}.
Here, the second inequality also uses the facts that the Gram matrix $\H(\cdot)$ is $2LL_2$-Lipschitz continuous from Proposition \ref{prop:gram_lip} and the condition that each~$\nabla^2F_i(\cdot)$ is $L_3$-Lipschitz continuous. 
Since we have $\nabla\phi(\x_*)=\J(\x_*)^{\top}\F(\x_*)=\0$, it holds
\begin{align}
\label{eq:sati-1}
\begin{split}
&\Norm{\J(\x)^{\top}\F(\x)-(\mH(\x_*)+\mQ(\x_*))(\x-\x_*)} \\
&= \Norm{\nabla\phi(\vx)-\nabla\phi(\vx^*)-\nabla^2\phi(\x^*)(\x-\x^*)}\\
&=\Norm{\int_{0}^{1}
\bigl(\nabla^2\phi(\x_*+s(\x-\x_*))-\nabla^2\phi(\x_*)\bigr)
(\x-\x_*)\,{\rm d}s}\\
&\!\!\overset{\eqref{eq:nabla_2phi_bound}}{\leq}\int_0^1 2\rho s\|\x-\x_*\|^2\,{\rm d}s
=\rho\|\x-\x_*\|^2,
\end{split}
\end{align}
which verifies Assumption~\ref{ass:x_star_philip}.

\paragraph{Case 2:} Condition \eqref{eq:phi-2-Lip-xy} implies that 
\begin{align*}
\|\nabla^2\phi(\x)-\nabla^2\phi(\x_*)\|
\leq\rho_0\|\x-\x_*\|    
\end{align*}
holds for all $\vx\in\BR^d$, where $\vx^*\in\BR^d$ follows the definition in Assumption \ref{ass:x_star_philip}.
In the view of equation \eqref{eq:nabla_2phi_bound}, we can follow the derivation of equation \eqref{eq:sati-1} by replacing $\rho$ with $\rho_0/2$ to verify 
Assumption~\ref{ass:x_star_philip} with $\rho=\rho_0/2$.

\paragraph{Case 3:} 
Note that the condition $\F(\x_*)=\0$ implies $\vx^*$ is a stationary point of $\phi(\cdot)$. 
Consequently, the form of
$\nabla^2\phi(\cdot)$ gives
\begin{align}
    \label{eq:zerox_*_property}
\nabla^2\phi(\x_*)=\H(\x_*)+\sum_{i=1}^nF_i(\x_*)\nabla^2F_i(\x_*)=\H(\x_*)=\J(\x_*)^{\top}\J(\x_*).
\end{align}
Moreover, the fundamental theorem of calculus gives
\begin{align}
    \label{eq:F_diff_bound}
    \begin{split}
&\Norm{\F(\x)-\F(\x_*)-\J(\x_*)(\x-\x_*)}\\
&=\Norm{\int_0^1
\bigl(\J(\x_*+s(\x-\x_*))-\J(\x_*)\bigr)
(\x-\x_*)\,{\rm d}s}\\
&\overset{\eqref{eq:lip}}{\leq}
\int_0^1L_2s\|\x-\x_*\|^2\,{\rm d}s
=\frac{L_2}{2}\|\x-\x_*\|^2.
    \end{split}
\end{align}
Proposition~\ref{prop:gram_lip} also gives
\begin{align}
    \label{eq:normFxbounded}
\|\F(\x)\|
&=\|\F(\x)-\F(\x_*)\|
\leq L\|\x-\x_*\|.
\end{align}
Thus, it holds
\begin{align*}
&\big\|\J(\x)^{\top}\F(\x)-\nabla^2\phi(\x_*)((\x-\x_*)\big\|\\
&\overset{\eqref{eq:zerox_*_property}}{=}\big\|\J(\x)^{\top}\F(\x)-\J(\x_*)^{\top}\F(\x_*)
-\H(\x_*)(\x-\x_*)\big\|\\
&~=\big\|\big(\J(\x)^{\top}-\J(\x_*)^{\top}\big)\F(\x)
+\J(\x_*)^{\top}
(\F(\x)-\F(\x_*)-\J(\x_*)(\x-\x_*))\big\|\\
&~\leq\big\|\J(\x)^{\top}-\J(\x_*)^{\top}\big\|\,\|\F(\x)\|
+\|\J(\x_*)\|\,
\|\F(\x)-\F(\x_*)-\J(\x_*)(\x-\x_*)\|\\
&\!\!\!\overset{\eqref{eq:lip},\eqref{eq:normFxbounded}}{\leq} LL_2\|\x-\x_*\|^2+ \|\J(\x_*)\|\,
\|\F(\x)-\F(\x_*)-\J(\x_*)(\x-\x_*)\|\\
&\!\!\!\overset{\eqref{eq:lip},\eqref{eq:F_diff_bound}}{\leq}
\left(LL_2+\frac{LL_2}{2}\right)\|\x-\x_*\|^2
=\frac{3LL_2}{2}\|\x-\x_*\|^2,
\end{align*}
which verifies Assumption~\ref{ass:x_star_philip} with $\rho = {3LL_2}/{2}$.
\end{proof}

\subsection{Proof of Lemma~\ref{lm:pd_approximate_phi}}
\label{app:proof_lemma_pd_approximate_phi}
\begin{proof}
We first prove equation~\eqref{eq:approx_phi_positive}.
The condition $c\in(1,1/\theta)$ indicates $(c-1)/c<1$, 
which implies for all $\vx\in\BR^d$ such that $\|\x-\x_*\|\leq(c-1)\mu_*^2/(4cLL_2)$, it holds
\begin{align*}
\x\in\Omega_{\mathrm{over}}.     
\end{align*}
Since matrix $\Q(\x_*)$ is symmetric and Assumption~\ref{ass:nonsingular_J^*} gives $\H(\x_*)\succ\0$, 
it holds 
\begin{align*}
\big|\v^{\top}\H(\x_*)^{-1/2}\Q(\x_*)
\H(\x_*)^{-1/2}\v\big|
&\leq\big\|\H(\x_*)^{-1/2}\Q(\x_*)\H(\x_*)^{-1/2}\big\|
\|\v\|^2\overset{\eqref{eq:condiDelta}}=\theta\|\v\|^2.
\end{align*}
This implies
\begin{align*}
-\theta\I_d
&\preceq\H(\x_*)^{-1/2}\Q(\x_*)\H(\x_*)^{-1/2}
\preceq\theta\I_d.
\end{align*}
Combining the above result and Lemma~\ref{lm:matrix_order_equivalence} with
$\A=\H(\x_*)$, $\B=\Q(\x_*)$, $a_1=-\theta$, and~$a_2=\theta$, we achieve
\begin{align}
    \label{eq:Qx_star_bound_H_x_star}
-\theta \H(\x^*)\preceq \Q(\x^*)\preceq \theta \H(\x^*).
\end{align}
Therefore, Propositions \ref{prop:Gram_concor} and \ref{prop:local_nonsingularity} imply
\begin{align*}
-\theta\bigl(1+M\|\x-\x_*\|\bigr)\H(\x)
\overset{\eqref{eq:Gram_bound_0}}{\preceq}-\theta\H(\x_*)
\overset{\eqref{eq:Qx_star_bound_H_x_star}}{\preceq}\Q(\x_*)
\overset{\eqref{eq:Qx_star_bound_H_x_star}}{\preceq}\theta\H(\x_*)\overset{\eqref{eq:Gram_bound_0}}{\preceq}\theta\bigl(1+M\|\x-\x_*\|\bigr)\H(\x)
\end{align*}
for all $\vx\in\BR^d$, where $M=2LL_2/\mu^2$ and $\mu=\mu_*/\sqrt{2}$.
Additionally, for all $c>1$, the condition of $\|\x-\x_*\|\leq(c-1)\mu_*^2/(4cLL_2)$  in the statement implies
\begin{align*}
1+M\|\x-\x_*\|
&\leq1+\frac{2LL_2}{\mu^2}\cdot
\frac{(c-1)\mu_*^2}{4cLL_2}
=2-\frac{1}{c}\leq c,
\end{align*}
which indicates
\begin{align}
    \label{eq:Qx_star_bound}
-c\theta\H(\x)
&\preceq\Q(\x_*)\preceq c\theta\H(\x).
\end{align}
Recall that Propositions~\ref{prop:gram_lip} and~\ref{prop:gram_lower} give
\begin{align}\label{eq:mu^2-H-L}
\mu^2\I_d\preceq\H(\x)\preceq L^2\I_d.    
\end{align}
Thus, for matrix $\tilde{\nabla}^2\phi(\x) = \H(\x)+\Q(\x^*)$, it holds
\begin{align}
    \label{eq:tildenabla^2phi_bound}
(1-c\theta)\mu^2\I_d
\overset{\eqref{eq:mu^2-H-L}}\preceq(1-c\theta)\H(\x)
\overset{\eqref{eq:Qx_star_bound}}{\preceq}\tilde{\nabla}^2\phi(\x)\overset{\eqref{eq:Qx_star_bound}}{\preceq}(1+c\theta)\H(\x)
\overset{\eqref{eq:mu^2-H-L}}\preceq(1+c\theta)L^2\I_d
\end{align}
for all $c\in(1,1/\theta)$,
which proves~equation \eqref{eq:approx_phi_positive}.

We then prove equation~\eqref{eq:concor_phi}.
Recall that Proposition \ref{prop:gram_lip} indicates $\H(\cdot)$ is $2LL_2$-Lipschitz continuous, which leads to
\begin{align}
\label{eq:Hy-Hxboundebydistance}
-2LL_2\|\y-\x\|\I_d
&\preceq\H(\y)-\H(\x)
=\tilde{\nabla}^2\phi(\y)-\tilde{\nabla}^2\phi(\x)
\preceq2LL_2\|\y-\x\|\I_d.
\end{align}
We also have
\begin{align}
\label{eq:distanceboundbytildeHessian}
2LL_2\|\y-\x\|\I_d
\overset{\eqref{eq:tildenabla^2phi_bound}}{\preceq}\frac{2LL_2\|\y-\x\|}{(1-c\theta)\mu^2}
\tilde{\nabla}^2\phi(\x)=\widehat M\|\y-\x\|\tilde{\nabla}^2\phi(\x),
\end{align}
where $\widehat M= 2LL_2/((1-c\theta)\mu^2)$.
Therefore, it holds
\begin{align*}
\tilde{\nabla}^2\phi(\y)
&\overset{\eqref{eq:Hy-Hxboundebydistance}}{\preceq}\tilde{\nabla}^2\phi(\x)+2LL_2\|\y-\x\|\I_d
\overset{\eqref{eq:distanceboundbytildeHessian}}{\preceq}(1+\widehat M\|\y-\x\|)\tilde{\nabla}^2\phi(\x)
\end{align*}
and
\begin{align*}
\tilde{\nabla}^2\phi(\y)
&\overset{\eqref{eq:Hy-Hxboundebydistance}}{\succeq}\tilde{\nabla}^2\phi(\x)-2LL_2\|\y-\x\|\I_d
\overset{\eqref{eq:distanceboundbytildeHessian}}{\succeq}(1-\widehat M\|\y-\x\|)\tilde{\nabla}^2\phi(\x),
\end{align*}
which proves~\eqref{eq:concor_phi}.
\end{proof}

\subsection{Proof of Lemma~\ref{lm:GNlinearquadra}}
\label{app:proof_lemma_linear_quadratic_over}
\begin{proof}
For the ease of presentation, we denote
\begin{align}\label{eq:dfn-R}
\R\triangleq\nabla\phi(\x)-\nabla^2\phi(\x_*)(\x-\x_*).    
\end{align}
The expressions of $\tilde{\nabla}^2\phi(\cdot)$ and ${\nabla}^2\phi(\cdot)$ implies
\begin{align*}
\tilde{\nabla}^2\phi(\x)=\nabla^2\phi(\x_*)+\H(\x)-\H(\x_*).  
\end{align*}
Combing with the update rule~\eqref{eq:qNupdate} and the definition of $\mR$, we achieve
\begin{align*}
\x_{+}-\x_*
&=\x-\x_*-\G^{-1}\nabla\phi(\x)\\
&=(\I_d-\G^{-1}\tilde{\nabla}^2\phi(\x))(\x-\x_*)
  +\G^{-1}(\H(\x)-\H(\x_*))(\x-\x_*)-\G^{-1}\R.
\end{align*}
Left multiplying matrix $(\tilde{\nabla}^2\phi(\x))^{1/2}$ on both sides of above equation gives
\begin{align}
\label{eq:splithatlambda}
(\tilde{\nabla}^2\phi(\x))^{1/2}(\x_{+}-\x_*)
&=\u+\v-\w,
\end{align}
where
\begin{align}
\u
&\triangleq\bigl(\I_d-\bigl(\tilde{\nabla}^2\phi(\x)\bigr)^{1/2}
\G^{-1}\bigl(\tilde{\nabla}^2\phi(\x)\bigr)^{1/2}\bigr)
\bigl(\tilde{\nabla}^2\phi(\x)\bigr)^{1/2}(\x-\x_*), \label{dfn:proof-u}\\
\v
&\triangleq\bigl(\tilde{\nabla}^2\phi(\x)\bigr)^{1/2}\G^{-1}
(\H(\x)-\H(\x_*))(\x-\x_*), \label{dfn:proof-v}\\
\w
&\triangleq\bigl(\tilde{\nabla}^2\phi(\x)\bigr)^{1/2}\G^{-1}\R. \label{dfn:proof-w}
\end{align}
We provide upper bounds for $\u$, $\vv$ and $\vw$ as follows.
\begin{itemize}
\item \textbf{The upper bound for $\|\u\|$.}
The condition~\eqref{eq:condi_x_linear_quadra} for matrix $\G$ gives
\begin{align*}
\frac{1}{\eta}(\H(\x))^{-1}
&\preceq\G^{-1}\preceq(\H(\x))^{-1}.
\end{align*}
According to Lemma \ref{lm:pd_approximate_phi},
taking inverse on equation~\eqref{eq:approx_phi_positive} gives
\begin{align*}
(1-c\theta)\bigl(\tilde{\nabla}^2\phi(\x)\bigr)^{-1}
\preceq (\H(\x))^{-1}
\preceq (1+c\theta)\bigl(\tilde{\nabla}^2\phi(\x)\bigr)^{-1},
\end{align*}
where $c\in(1,1/\theta)$ and $\theta$ follows definition \eqref{eq:condiDelta}.
The above results yields
\begin{align}
\label{eq:boundGinverse}
\frac{1-c\theta}{\eta}\bigl(\tilde{\nabla}^2\phi(\x)\bigr)^{-1}
\preceq \frac{1}{\eta}(\H(\x))^{-1}\preceq\G^{-1}\preceq (\H(\x))^{-1}
\preceq(1+c\theta)\bigl(\tilde{\nabla}^2\phi(\x)\bigr)^{-1}.
\end{align}
Applying Lemma~\ref{lm:matrix_order_equivalence}
to equation \eqref{eq:boundGinverse} 
with $\A=(\tilde{\nabla}^2\phi(\x))^{-1}$, 
$\B=\G^{-1}$, $a_1=(1-c\theta)/\eta$, and $a_2=1+c\theta$
gives
\begin{align}
\label{eq:upperphiGphi}
\frac{1-c\theta}{\eta}\I_d
\preceq\bigl(\tilde{\nabla}^2\phi(\x)\bigr)^{1/2}\G^{-1}
\bigl(\tilde{\nabla}^2\phi(\x)\bigr)^{1/2}
\preceq(1+c\theta)\I_d.
\end{align}
Recall that the setting $\eta\geq 1$ implies $c\theta\leq1-(1-c\theta)/\eta$, which leads to
\begin{align}\label{eq:bound_u}
\begin{split}
\|\u\|
& \overset{\eqref{dfn:proof-u},\,\eqref{eq:our_measure}}\leq\Bigl\|\I_d-\big(\tilde{\nabla}^2\phi(\x)\big)^{1/2}\G^{-1}
\big(\tilde{\nabla}^2\phi(\x)\big)^{1/2}\Bigr\|
\lambda_{\mathrm{over}}(\x) \\
&~~\,\overset{\eqref{eq:upperphiGphi}}{\leq}~~\,
\left(1-\frac{1-c\theta}{\eta}\right)\lambda_{\mathrm{over}}(\x).
\end{split}
\end{align}
\item \textbf{The upper bound for $\|\v\|$.}
The definition of $\lambda_{\mathrm{over}}(\cdot)$ in equation
\eqref{eq:our_measure} and Lemma~\ref{lm:pd_approximate_phi}
imply
\begin{align}
\label{eq:distance_lambda}
\|\x-\x_*\|
\leq\Norm{\bigl(\tilde{\nabla}^2\phi(\x)\bigr)^{-1/2}}
\Norm{\bigl(\tilde{\nabla}^2\phi(\x)\bigr)^{1/2}(\x-\x_*)}
\overset{\eqref{eq:approx_phi_positive}}{\leq}
\frac{\lambda_{\mathrm{over}}(\x)}{\sqrt{1-c\theta}\mu},
\end{align}
where $c\in(1,1/\theta)$ and $\theta$ follows definition \eqref{eq:condiDelta}.
Based on the $2LL_2$-Lipschitz continuity of $\H(\cdot)$ from
Proposition~\ref{prop:gram_lip} and Lemma~\ref{lm:pd_approximate_phi}, we obtain
\begin{align}
\|\v\|
&\overset{\eqref{dfn:proof-v}}\leq\Norm{\bigl(\tilde{\nabla}^2\phi(\x)\bigr)^{1/2}\G^{-1}
\bigl(\tilde{\nabla}^2\phi(\x)\bigr)^{1/2}}
\Norm{\bigl(\tilde{\nabla}^2\phi(\x)\bigr)^{-1/2}}\cdot\Norm{\H(\x)-\H(\x_*)}\|\x-\x_*\|\nonumber\\
&\,\,\leq \Norm{\bigl(\tilde{\nabla}^2\phi(\x)\bigr)^{1/2}\G^{-1}
\bigl(\tilde{\nabla}^2\phi(\x)\bigr)^{1/2}}
\Norm{\bigl(\tilde{\nabla}^2\phi(\x)\bigr)^{-1/2}} \cdot 2LL_2\|\x-\x^*\|^2\nonumber
\\
&\!\!\!\!\!\overset{\eqref{eq:upperphiGphi}, \eqref{eq:approx_phi_positive}}{\leq}
\frac{2(1+c\theta)LL_2}{\sqrt{1-c\theta}\mu}\|\x-\x_*\|^2\nonumber\\
&\!\!\overset{\eqref{eq:distance_lambda}}{\leq}
\frac{2(1+c\theta)LL_2}{(1-c\theta)^{3/2}\mu^3}
\bigl(\lambda_{\mathrm{over}}(\x)\bigr)^2.\label{eq:bound_v}
\end{align}
\item \textbf{The upper bound for $\|\w\|$.}
Recall that equation \eqref{eq:boundGinverse} gives
\begin{align*}
\frac{1-c\theta}{\eta}\G\preceq\tilde{\nabla}^2\phi(\x)
\preceq(1+c\theta)\G,
\end{align*}
which implies
\begin{align}
\label{eq:boundGinverse2}
\frac{1-c\theta}{\eta}\G^{-1}\preceq\G^{-1}\tilde{\nabla}^2\phi(\x)\G^{-1}
&\preceq(1+c\theta)\G^{-1}.
\end{align}
On the other hand,
Proposition \ref{prop:gram_lip} and Lemma \ref{lm:GNlinearquadra} give
\begin{align}
\label{eq:Glowerbound}
    \mu^2\I_d
&\preceq\H(\x)\preceq\G.
\end{align}
Combining above two results, we obtain
\begin{align}
    \label{eq:boundGinversephiGinverse}
\Norm{\bigl(\tilde{\nabla}^2\phi(\x)\bigr)^{1/2}\G^{-1}}^2
&=\Norm{\G^{-1}\tilde{\nabla}^2\phi(\x)\G^{-1}}\overset{\eqref{eq:boundGinverse2}}{\leq}(1+c\theta)\Norm{\G^{-1}}
\overset{\eqref{eq:Glowerbound}}{\leq}\frac{1+c\theta}{\mu^2}.
\end{align}
Thus, the definition of $\mR$ in equation \eqref{eq:dfn-R} implies
\begin{align}
    \label{eq:bound_w}
\begin{split}
\|\w\|
&\overset{\eqref{dfn:proof-w}}\leq\Norm{\bigl(\tilde{\nabla}^2\phi(\x)\bigr)^{1/2}\G^{-1}}\Norm{\R}\\
&\!\!\!\!\!\!\overset{\eqref{eq:boundGinversephiGinverse},\,\eqref{eq:x_star_philip}}{\leq}
\frac{\sqrt{1+c\theta}\,\rho}{\mu}\|\x-\x_*\|^2
\overset{\eqref{eq:distance_lambda}}{\leq}
\frac{\sqrt{1+c\theta}\,\rho}{(1-c\theta)\mu^3}
\bigl(\lambda_{\mathrm{over}}(\x)\bigr)^2.
\end{split}
\end{align}
\end{itemize}
We then define
\begin{align*}
K
&\triangleq\frac{2(1+c\theta)LL_2}{(1-c\theta)^{3/2}\mu^3}
+\frac{\sqrt{1+c\theta}\,\rho}{(1-c\theta)\mu^3}.
\end{align*}
Plugging the above upper bounds for $\u$, $\vv$ and $\vw$ into~equation \eqref{eq:splithatlambda} gives
\begin{align}
\label{eq:phix_+-x}
\Norm{\big(\tilde{\nabla}^2\phi(\x)\big)^{1/2}(\x_{+}-\x_*)}
&\overset{\eqref{eq:bound_u},\,\eqref{eq:bound_v},\,\eqref{eq:bound_w}}{\leq}\left(1-\frac{1-c\theta}{\eta}\right)\lambda_{\mathrm{over}}(\x)
+K\bigl(\lambda_{\mathrm{over}}(\x)\bigr)^2.
\end{align}
On the other hand, Lemma \ref{lm:pd_approximate_phi} gives
\begin{align}
\label{eq:lower_x_+-x^*over}
\!\!\!\big\|\big(\tilde{\nabla}^2\phi(\x)\big)^{1/2}
(\x_{+}-\x_*)\big\|^2
&=(\x_{+}-\x_*)^{\top}\tilde{\nabla}^2\phi(\x)(\x_{+}-\x_*)\overset{\eqref{eq:approx_phi_positive}}{\geq}
(1-c\theta)\mu^2\|\x_{+}-\x_*\|^2.
\end{align}
Recall that definitions of $K$ and $C_1$ give $K\leq C_1$.
Hence, we achieve equation \eqref{eq:overfit_distance_lambda} as follows
\begin{align*}
\|\x_{+}-\x_*\|
&\overset{\eqref{eq:lower_x_+-x^*over}}{\leq}\frac{1}{\sqrt{1-c\theta}\mu}
\Norm{\bigl(\tilde{\nabla}^2\phi(\x)\bigr)^{1/2}
(\x_{+}-\x_*)}\\
&\overset{\eqref{eq:phix_+-x}}{\leq}
\frac{1}{\sqrt{1-c\theta}\mu}
\left(\left(1-\frac{1-c\theta}{\eta}\right)\lambda_{\mathrm{over}}(\x)
+K(\lambda_{\mathrm{over}}(\x))^2\right)\\
&\,~\leq\,~\frac{1}{\sqrt{1-c\theta}\mu}
\left(\left(1-\frac{1-c\theta}{\eta}\right)\lambda_{\mathrm{over}}(\x)
+C_1(\lambda_{\mathrm{over}}(\x))^2\right).
\end{align*}

We then prove equation \eqref{eq:movedistance_lambda} as follows
\begin{align}\label{eq:x_+-xlambda}
\begin{split}    
\|\x_{+}-\x\|
&\,\,\overset{\eqref{eq:qNupdate}}{=}~\, 
\Norm{\G^{-1}\J(\x)^{\top}\F(\x)} \\
&\,~=\,~\Norm{\G^{-1}\bigl(\J(\x)^{\top}\F(\x)
-\J(\x_*)^{\top}\F(\x_*)\bigr)} \\
& \overset{\eqref{eq:Glowerbound}}\leq\frac{\Norm{\J(\x)^{\top}(\F(\x)-\F(\x_*))}
 +\Norm{(\J(\x)^{\top}-\J(\x_*)^{\top})\F(\x_*)}}{\mu^2} \\
&\,\,\overset{\eqref{eq:lip}}{\leq}\,\,
\frac{L^2+L_2\|\F(\x_*)\|}{\mu^2}\|\x-\x_*\| \\
&\overset{\eqref{eq:distance_lambda}}{\leq}
\frac{L^2+L_2\|\F(\x_*)\|}{\sqrt{1-c\theta}\mu^3}
\lambda_{\mathrm{over}}(\x).
\end{split}
\end{align}

Recall that $\widehat M=2LL_2/((1-c\theta)\mu^2)$ and
$D=LL_2(L^2+L_2\|\F(\x_*)\|)$, which implies
\begin{align}
    \label{eq:Mx+-xbound}
    \begin{split}
\widehat M\|\x_{+}-\x\|
&\overset{\eqref{eq:x_+-xlambda}}{\leq}
\frac{2D\lambda_{\mathrm{over}}(\x)}{(1-c\theta)^{3/2}\mu^5} \\
&\,\overset{\eqref{eq:condi_x_linear_quadra}}{\leq}\,
\frac{2D}{(1-c\theta)^{3/2}\mu^5}
\cdot\frac{(1-c\theta)^{3/2}\mu_*^5}{16D}\\
&~\,=~\,\frac{\mu_*^5}{8\mu^5}
=\frac{1}{\sqrt{2}}<1,
\end{split}
\end{align}
where the last equality uses the fact $\mu=\mu_*/\sqrt{2}$.
Applying Lemma \ref{lm:pd_approximate_phi} with $\y=\x_{+}$ gives
\begin{align*}
\tilde{\nabla}^2\phi(\x_{+})
&\overset{\eqref{eq:concor_phi}}{\succeq}
\bigl(1-\widehat M\|\x_{+}-\x\|\bigr)
\tilde{\nabla}^2\phi(\x)\\
&\overset{\eqref{eq:approx_phi_positive}}{\succeq}
\bigl(1-\widehat M\|\x_{+}-\x\|\bigr)
    (1-c\theta)\mu^2\I_d\\
&\!\overset{\eqref{eq:Mx+-xbound}}{\succeq}\biggl(1-\frac{1}{\sqrt{2}}\biggr)
(1-c\theta)\mu^2\I_d\succ\0.
\end{align*}
Thus, the measure $\lambda_{\mathrm{over}}(\cdot)$ is well-defined, and we have
\begin{align*}
\bigl(\lambda_{\mathrm{over}}(\x_{+})\bigr)^2
&=(\x_{+}-\x_*)^{\top}\tilde{\nabla}^2\phi(\x_{+})(\x_{+}-\x_*)\\
&\!\!\overset{\eqref{eq:concor_phi}}{\leq}\!\!
\bigl(1+\widehat M\|\x_{+}-\x\|\bigr)
(\x_{+}-\x_*)^{\top}\tilde{\nabla}^2\phi(\x)(\x_{+}-\x_*)\\
&=\bigl(1+\widehat M\|\x_{+}-\x\|\bigr)
\Norm{\bigl(\tilde{\nabla}^2\phi(\x)\bigr)^{1/2}
(\x_{+}-\x_*)}^2.
\end{align*}
Note that $\lambda_{\mathrm{over}}(\x_{+})\geq0$ and
$1+\widehat M\|\x_{+}-\x\|\geq1$. Therefore, taking square roots on above inequality proves equation \eqref{eq:overfit_linear_quadra} as follows
\begin{align*}
\lambda_{\mathrm{over}}(\x_{+})
&\,~\leq~\,\sqrt{1+\widehat M\|\x_{+}-\x\|}
\Norm{\bigl(\tilde{\nabla}^2\phi(\x)\bigr)^{1/2}
(\x_{+}-\x_*)}\\
&\overset{\eqref{eq:Mx+-xbound}}{\leq}
\sqrt{1+\frac{2D\lambda_{\mathrm{over}}(\x)}{(1-c\theta)^{3/2}\mu^5}
}
\Norm{\bigl(\tilde{\nabla}^2\phi(\x)\bigr)^{1/2}
(\x_{+}-\x_*)}\\
&\overset{\eqref{eq:phix_+-x}}{\leq}
\sqrt{1+\frac{2D\lambda_{\mathrm{over}}(\x)}{(1-c\theta)^{3/2}\mu^5}}\,\left(\left(1-\frac{1-c\theta}{\eta}\right)\lambda_{\mathrm{over}}(\x)
+ K\bigl(\lambda_{\mathrm{over}}(\x)\bigr)^2\right)\\
&~\leq\,~\left(1+\frac{2D\lambda_{\mathrm{over}}(\x)}{(1-c\theta)^{3/2}\mu^5}\right)
\left(\left(1-\frac{1-c\theta}{\eta}\right)\lambda_{\mathrm{over}}(\x)
+K\bigl(\lambda_{\mathrm{over}}(\x)\bigr)^2\right)\\
&~=\left(1-\frac{1-c\theta}{\eta}\right)\lambda_{\mathrm{over}}(\x) + \left(\frac{2D}{(1-c\theta)^{3/2}\mu^5}\left(1-\frac{1-c\theta}{\eta}\right)+K\right)
\bigl(\lambda_{\mathrm{over}}(\x)\bigr)^2 \\
&~~~~~~~~~~+ \frac{2DK(\lambda_{\mathrm{over}}(\vx))^{3}}{(1-c\theta)^{3/2}\mu^5}
\\
&~\leq\,~\left(1-\frac{1-c\theta}{\eta}\right)\lambda_{\mathrm{over}}(\x)
+\left(\frac{2D}{(1-c\theta)^{3/2}\mu^5}+2K\right)
\bigl(\lambda_{\mathrm{over}}(\x)\bigr)^2\\
&~\leq\,~\left(1-\frac{1-c\theta}{\eta}\right)\lambda_{\mathrm{over}}(\x)
+C_1\bigl(\lambda_{\mathrm{over}}(\x)\bigr)^2,
\end{align*}
where the second last inequality is based on the result $2D\lambda_{\mathrm{over}}(\vx)/((1-c\theta)^{3/2}\mu^5)\leq 1$ achieved by the condition on $\lambda_{\mathrm{over}}(\x)$ shown in \eqref{eq:condi_x_linear_quadra}
and the last inequality is based on definitions of~$C_1$ and $K$ which imply $2D/((1-c\theta)^{3/2}\mu^5)+2K\leq C_1$.
\end{proof}

\subsection{Proof of Theorem~\ref{thm:linear1n>d}}
\label{app:proof_theorem_linear_over}
\begin{proof}
For the ease of presentation, we denote
\begin{align*}
\lambda_{\mathrm{over},t}\triangleq\lambda_{\mathrm{over}}(\x_t)
\end{align*}
in the following analysis.


We now use induction to prove equation \eqref{eq:induc_lambda} and the result
\begin{align}
\label{eq:inducx_tdistance}
\|\x_t-\x_*\|
\leq\frac{(c-1)\mu_*^2}{4cLL_2}
\end{align}
hold for all $t\geq0$.
For the induction base, 
the definition of $\eta_0$ gives
\begin{align*}
\H_0\preceq\G_0\preceq\eta_0\H_0,
\end{align*}
which implies equation \eqref{eq:induc_lambda} holds for $t=0$.
Additionally, the initial condition \eqref{eq:linear_initial_condi} directly gives~equation \eqref{eq:inducx_tdistance} for $t=0$.
For the induction step, we suppose equations 
\eqref{eq:inducx_tdistance} and \eqref{eq:induc_lambda} hold for all $i=0,1,\dots,t$.
Then equation \eqref{eq:inducx_tdistance} implies that $\x_i\in\Omega_{\mathrm{over}}$ for all $i=0,1,\dots,t$, which gives
\begin{align}\label{eq:lambda_over_no_grow}
\lambda_{\mathrm{over},t}
\overset{\eqref{eq:induc_lambda}}{\leq}\lambda_{\mathrm{over},0}
\overset{\eqref{eq:linear_initial_condi}}{\leq}
\frac{1-c\theta}{10\eta_0 C_1}
\overset{\eqref{eq:mu-C1}}{\leq} 
\frac{(1-c\theta)^{3/2}\mu_*^5}{16D}. 
\end{align}
Applying Lemma~\ref{lm:GNlinearquadra} with $\vx=\x_t$  and $\vx_+=\vx_{t+1}$ gives
\begin{align}\label{eq:lambda_over_t+1_bound}
\begin{split}
\lambda_{\mathrm{over},t+1}
&~~\overset{\eqref{eq:overfit_linear_quadra}}{\leq}~~
\biggl(1-\frac{1-c\theta}{\eta_t}
+C_1\lambda_{\mathrm{over},t}\biggr)\lambda_{\mathrm{over},t}\\
&\!\!\overset{\eqref{eq:induc_lambda},\,\eqref{eq:lambda_over_no_grow}}{\leq}\!\!\biggl(1-\frac{2(1-c\theta)}{3\eta_0}
+\frac{1-c\theta}{10\eta_0}\biggr)\lambda_{\mathrm{over},t}\\
&~~~\leq~~~\biggl(1-\frac{1-c\theta}{2\eta_0}\biggr)
\lambda_{\mathrm{over},t}\\
&~~\,\overset{\eqref{eq:induc_lambda}}{\leq}~~\biggl(1-\frac{1-c\theta}{2\eta_0}\biggr)^{t+1}
\lambda_{\mathrm{over},0}.
\end{split}
\end{align}
It also implies
\begin{align*}
\|\x_{t+1}-\x_*\|
&\,\,\overset{\eqref{eq:overfit_distance_lambda}}{\leq}\,
\frac{1}{\sqrt{1-c\theta}\mu}
\left(\left(1-\frac{1-c\theta}{\eta_t}\right)
\lambda_{\mathrm{over},t}+C_1\lambda_{\mathrm{over},t}^2\right)\\
&\overset{\eqref{eq:lambda_over_t+1_bound}}{\leq}
\frac{1}{\sqrt{1-c\theta}\mu}
\left(1-\frac{1-c\theta}{2\eta_0}\right)^{t+1}
\lambda_{\mathrm{over},0}\\
&\,~\leq~\frac{\lambda_{\mathrm{over},0}}
{\sqrt{1-c\theta}\mu}\overset{\eqref{eq:linear_initial_condi}}{\leq}\frac{\sqrt{1-c\theta}}{10\eta_0 C_1 \mu}\\
&\,\,\overset{\eqref{eq:mu-C1}}{\leq}
\frac{\sqrt{1-c\theta}\,(c-1)\mu^2}{20c\eta_0LL_2}
\leq\frac{(c-1)\mu_*^2}{4cLL_2},
\end{align*}
which means
\begin{align*}
\x_{t+1}\in\Omega_{\mathrm{over}}.    
\end{align*}
Furthermore, 
the induction hypothesis on equation \eqref{eq:induc_lambda} and
Proposition~\ref{prop:Gram_concor} with $\vx=\vx_t$ and~$\vy=\vx_{t+1}$ imply
\begin{align*}
\widetilde{\G}_t
&=(1+Mr_t)\G_t\overset{\eqref{eq:induc_lambda}}{\preceq}
\eta_0(1+Mr_t)
\exp\biggl(2\sum_{i=0}^{t-1}Mr_i\biggr)\H_t\\
&\overset{\eqref{eq:Gram_bound_0}}{\preceq}
\eta_0(1+Mr_t)^2
\exp\biggl(2\sum_{i=0}^{t-1}Mr_i\biggr)\H_{t+1}\preceq\eta_0\exp\biggl(2\sum_{i=0}^{t}Mr_i\biggr)\H_{t+1},
\end{align*}
where the last step is based on the fact $1+x\leq {\rm e}^{x}$ with $x=Mr_{\hat{t}}$.
It also implies
\begin{align*}
\H_{t+1}
&\overset{\eqref{eq:Gram_bound_0}}{\preceq}
(1+Mr_t)\H_t
\overset{\eqref{eq:induc_lambda}}{\preceq}(1+Mr_t)\G_t
=\widetilde{\G}_t.
\end{align*}
Therefore, we can Lemma~\ref{lm:QN_order} with $\mG=\widetilde{\G}_t$ and $\mH=\mH_{t+1}$ to achive
\begin{align}
\label{eq:over_G_t+1_H_t+1}
\H_{t+1}
&\preceq\G_{t+1}
\preceq\eta_0\exp\biggl(2\sum_{i=0}^{t}Mr_i\biggr)\H_{t+1}.
\end{align}
Additionally, Lemma \ref{lm:GNlinearquadra} with $\vx=\vx_t$ and $\vx_+=\vx_{t+1}$ implies
\begin{align}
M\sum_{i=0}^{t}r_i
&\overset{\eqref{eq:movedistance_lambda}}{\leq}
\frac{2D}{\sqrt{1-c\theta}\mu^5}
\sum_{i=0}^{t}\lambda_{\mathrm{over},i}\notag\\
&\overset{\eqref{eq:induc_lambda}}{\leq}
\frac{2D\lambda_{\mathrm{over},0}}{\sqrt{1-c\theta}\mu^5}
\sum_{i=0}^{t}\biggl(1-\frac{1-c\theta}{2\eta_0}\biggr)^i\notag\\
&~\leq\frac{4D\lambda_{\mathrm{over},0}\eta_0}{(1-c\theta)^{3/2}\mu^5}
\overset{\eqref{eq:linear_initial_condi}}{\leq}
\frac{2D}{5\sqrt{1-c\theta}\mu^5C_1}
\leq\frac{1-c\theta}{5}
\leq\frac{1}{5}.
\label{eq:sumr_ibound}
\end{align}
Consequently, we have
\begin{align*}
\G_{t+1}
\overset{\eqref{eq:over_G_t+1_H_t+1}}{\preceq}\eta_0\exp\biggl(2M\sum_{i=0}^{t}r_i\biggr)\H_{t+1}
\overset{\eqref{eq:sumr_ibound}}{\preceq}\eta_0{\rm e}^{2/5}\H_{t+1}
\preceq\frac{3\eta_0}{2}\H_{t+1}.
\end{align*}
This completes the induction. 
\end{proof}

\subsection{Proof of Theorem~\ref{thm:two-stage-n>d}}
\label{app:proof_theorem_two_stage_over}
\begin{proof}
We follow the notation  $\lambda_{\mathrm{over},t}=\lambda_{\mathrm{over}}(\x_t)$ defined in Appendix \ref{app:proof_theorem_linear_over}.
Condition~\eqref{eq:inicondi_final_n>d} means
the initial condition of Theorem~\ref{thm:linear1n>d} holds, then we achieve
\begin{align}
\label{eq:linear_path_over}
\lambda_{\mathrm{over},t}\overset{\eqref{eq:induc_lambda}}\leq \biggl(1-\frac{1-c\theta}{2\eta_0}\biggr)^t
\lambda_{\mathrm{over},0}
\qquad \text{and} \qquad
\|\x_t-\x_*\|\overset{\eqref{eq:inicondi_final_n>d}}{\leq}\frac{(c-1)\mu_*^2}{4cLL_2}.
\end{align}
Additionally, Lemma~\ref{lm:Gram_approximate} with $\Omega=\Omega_{\mathrm{over}}$ and $m=d$ gives
\begin{align}\label{eq:thm_delta_t_recur}
\H_t\preceq\G_t\preceq(1+\delta_t)\H_t
\qquad\text{and}\qquad
\mathbb E_t[\delta_{t+1}]
\leq(1-\gamma)(1+Mr_t)^2(\delta_t+2\tilde c dMr_t).
\end{align}
Note that equations \eqref{eq:linear_path_over} and \eqref{eq:thm_delta_t_recur} indicate conditions of
Lemma~\ref{lm:GNlinearquadra} holds by taking $\vx=\vx_t$, $\vx_+=\vx_{t+1}$, $\mG=\mG_t$, and $\eta=1+\delta_t$, then we achieve
\begin{align}
\label{eq:Mr_tboundedbylambda_t}
Mr_t
&\overset{\eqref{eq:movedistance_lambda}}{\leq}
\frac{2D\lambda_{\mathrm{over},t}}{\sqrt{1-c\theta}\mu^5}.
\end{align}
and
\begin{align}\label{eq:lambda_over_t+1boundbylambda_t}
\begin{split}    
\lambda_{\mathrm{over},t+1}
&\overset{\eqref{eq:overfit_linear_quadra}}{\leq}
\left(1-\frac{1-c\theta}{1+\delta_t}
+C_1\lambda_{\mathrm{over},t}\right)\lambda_{\mathrm{over},t} \\
&~=~ \left(\frac{\delta_t+c\theta}{1+\delta_t}
+C_1\lambda_{\mathrm{over},t}\right)\lambda_{\mathrm{over},t} \\
&~\leq\,\,\bigl(\delta_t
+C_1\lambda_{\mathrm{over},t}+c\theta\bigr)\lambda_{\mathrm{over},t}.
\end{split}
\end{align}
Therefore, the definition $C_2\triangleq C_1+4\tilde{c}dD/\big(\sqrt{1-c\theta}\mu^5\big)$ leads to
\begin{align}\label{eq:E_tdeltat+1}
\begin{split}    
\mathbb E_t[\delta_{t+1}]
&\overset{\eqref{eq:thm_delta_t_recur}}{\leq}
(1-\gamma)(1+Mr_t)^2(\delta_t+2\tilde c dMr_t)\\
&\overset{\eqref{eq:Mr_tboundedbylambda_t}}{\leq}
(1-\gamma)\biggl(1+\frac{2D\lambda_{\mathrm{over},t}}{\sqrt{1-c\theta}\mu^5}
\biggr)^2
\biggl(\delta_t+\frac{4\tilde c dD\lambda_{\mathrm{over},t}}{\sqrt{1-c\theta}\mu^5}
\biggr)\\
&~\,=~\,(1-\gamma)\biggl(1+\frac{2D\lambda_{\mathrm{over},t}}{\sqrt{1-c\theta}\mu^5}
\biggr)^2
\bigl(\delta_t+(C_2-C_1)\lambda_{\mathrm{over},t}\bigr).
\end{split}
\end{align}
We denote
\begin{align}\label{eq:xi-t}
    \xi_t\triangleq\delta_t+C_2\lambda_{\mathrm{over},t} \geq \delta_t+C_1\lambda_{\mathrm{over},t},
\end{align}
then it holds
\begin{align}\label{eq:overfit_delta_recur}
\begin{split}    
\mathbb E_t[\delta_{t+1}]
&\overset{\eqref{eq:E_tdeltat+1}}{\leq}(1-\gamma)\left(1+\frac{2D\lambda_{\mathrm{over},t}}{\sqrt{1-c\theta}\mu^5}
\right)^2\bigl(\delta_t+(C_2-C_1)\lambda_{\mathrm{over},t}\bigr)\\
&~\,\leq\,~(1-\gamma)
\left(1+\frac{C_2\lambda_{\mathrm{over},t}}{2}\right)^2
\left(\delta_t+C_2\lambda_{\mathrm{over},t}\right)\\
&\overset{\eqref{eq:xi-t}}{=}(1-\gamma)
\left(1+\frac{C_2\lambda_{\mathrm{over},t}}{2}\right)^2\xi_t,  
\end{split}
\end{align}
where the second inequality is based on the definition of $C_2$ and the setting $\tilde c \geq 1$ which result~$2D/(\sqrt{1-c\theta}\mu^5)\leq C_2/2$. 
Additionally, equation~\eqref{eq:lambda_over_t+1boundbylambda_t} gives
\begin{align}
\label{eq:overfit_lambda_recur}
\begin{split}
C_2\lambda_{\mathrm{over},t+1}
 \overset{\eqref{eq:lambda_over_t+1boundbylambda_t}}{\leq} C_2\lambda_{\mathrm{over},t}
\bigl(\delta_t+C_1\lambda_{\mathrm{over},t}\bigr)
+c\theta C_2\lambda_{\mathrm{over},t}  
 \overset{\eqref{eq:xi-t}}\leq C_2\lambda_{\mathrm{over},t}\xi_t
+c\theta C_2\lambda_{\mathrm{over},t}.
\end{split}
\end{align}
Recall that we have defined $\nu\triangleq\min\{\gamma,(1-c\theta)/(2\eta_0)\}<\gamma<1$, then combining results of~\eqref{eq:overfit_delta_recur} and \eqref{eq:overfit_lambda_recur} gives
\begin{subequations}
\begin{align}
\label{eq:xi_one_step_over_raw}
\mathbb E_t[\xi_{t+1}]
&\!\overset{\eqref{eq:overfit_delta_recur},\,\eqref{eq:overfit_lambda_recur}}{\leq}\Bigl((1-\gamma)
\Bigl(1+\frac{C_2\lambda_{\mathrm{over},t}}{2}\Bigr)^2
+C_2\lambda_{\mathrm{over},t}\Bigr)\xi_t
+c\theta C_2\lambda_{\mathrm{over},t}\\
\label{eq:xi_one_step_over}
&~~~~\leq (1-\nu)
\Bigl(\Bigl(1+\frac{C_2\lambda_{\mathrm{over},t}}{2}\Bigr)^2
+\frac{C_2\lambda_{\mathrm{over},t}}{1-\nu}\Bigr)\xi_t
+c\theta C_2\lambda_{\mathrm{over},t}\notag\\
&~~~~\leq (1-\nu)
\Bigl(1+\frac{C_2\lambda_{\mathrm{over},t}}{2}\Bigr)^2\Bigl(1
 + \frac{C_2\lambda_{\mathrm{over},t}}{1-\nu}\Bigr)\xi_t
 + c\theta C_2\lambda_{\mathrm{over},t}\notag\\
&~~~~\leq(1-\nu)\exp\Bigl(C_2\lambda_{\mathrm{over},t}\Bigr)
\exp\Bigl(\frac{C_2\lambda_{\mathrm{over},t}}{1-\nu}\Bigr)\xi_t
+c\theta C_2\lambda_{\mathrm{over},t}\notag\\
&~~~~=(1-\nu)
\exp\Bigl(\frac{(2-\nu)C_2\lambda_{\mathrm{over},t}}{1-\nu}
\Bigr)\xi_t
+c\theta C_2\lambda_{\mathrm{over},t},
\end{align}
\end{subequations}
where the last second inequality is based on the result $1+s\leq {\rm e}^s$ with~$s=C_2\lambda_{{\rm over},t}$
and~$s={C_2\lambda_{\mathrm{over},t}}/(1-\nu)$.

Taking expectation on~result \eqref{eq:xi_one_step_over} and applying~equation \eqref{eq:linear_path_over} give
\begin{align}
    \label{eq:xi_t+1_recur_xi_t}
    \begin{split}
\mathbb E[\xi_{t+1}]
&\overset{\eqref{eq:xi_one_step_over},\,\eqref{eq:linear_path_over}}{\leq}
(1-\nu)
\exp\Biggl(\frac{(2-\nu)C_2}{1-\nu}
\biggl(1-\frac{1-c\theta}{2\eta_0}\biggr)^t
\lambda_{\mathrm{over},0}\Biggr)\mathbb E[\xi_t]
\\
&~~~~~~~~~~~~~~~~~+c\theta C_2\biggl(1-\frac{1-c\theta}{2\eta_0}\biggr)^t
\lambda_{\mathrm{over},0}.
    \end{split}
\end{align}
Expanding the term $\mathbb E[\xi_t]$ according to above recursion gives
\begin{align*}
\mathbb E[\xi_{t+1}]
&\leq(1-\nu)^2
\exp\Biggl(\frac{(2-\nu)C_2}{1-\nu}\lambda_{\mathrm{over},0}
\Bigl(\bigl(1-\frac{1-c\theta}{2\eta_0}\bigr)^t
+\bigl(1-\frac{1-c\theta}{2\eta_0}\bigr)^{t-1}\Bigr)\Biggr)
\mathbb E[\xi_{t-1}]\\
&\quad+c\theta C_2(1-\nu)
\biggl(1-\frac{1-c\theta}{2\eta_0}\biggr)^{t-1}
\lambda_{\mathrm{over},0}
\exp\Biggl(\frac{(2-\nu)C_2}{1-\nu}
\biggl(1-\frac{1-c\theta}{2\eta_0}\biggr)^t
\lambda_{\mathrm{over},0}\Biggr)\\
&\quad+c\theta C_2\biggl(1-\frac{1-c\theta}{2\eta_0}\biggr)^t
\lambda_{\mathrm{over},0}.
\end{align*}
Repeating this substitution gives
\begingroup
\small
\begin{align}\label{eq:finialxi_t_recur}
\begin{split}
&\!\!\!\mathbb E[\xi_t]
\leq(1-\nu)^t\xi_0
\exp\Biggl(\frac{(2-\nu)C_2\lambda_{\mathrm{over},0}}{1-\nu}
\sum_{j=0}^{t-1}\biggl(1-\frac{1-c\theta}{2\eta_0}\biggr)^j\Biggr)\\
&~~+c\theta C_2\lambda_{\mathrm{over},0}
\sum_{i=0}^{t-1}(1-\nu)^{t-1-i}
\biggl(1-\frac{1-c\theta}{2\eta_0}\biggr)^i
\exp\Biggl(\frac{(2-\nu)C_2\lambda_{\mathrm{over},0}}{1-\nu}
\sum_{j=i+1}^{t-1}\biggl(1-\frac{1-c\theta}{2\eta_0}\biggr)^j\Biggr)
    \end{split}
\end{align}
\endgroup
for all $t\geq0$.
In addition, we have
\begin{align*}
\frac{(2-\nu)C_2}{1-\nu}
\lambda_{\mathrm{over},0}
\sum_{j=i}^{t-1}\biggl(1-\frac{1-c\theta}{2\eta_0}\biggr)^j
\leq\frac{2(2-\nu)C_2\eta_0}{(1-\nu)(1-c\theta)}
\lambda_{\mathrm{over},0}\overset{\eqref{eq:inicondi_final_n>d}}{\leq}
\frac{(2-\nu)(1-\gamma)}{5(1-\nu)(2-\gamma)}
\leq\frac{1}{5}
\end{align*}
for all  $i=0,\dots,t$, 
where the last inequality uses the setting $\nu\leq\gamma$.
Consequently, we have
\begin{align}
\label{eq:exponential_bound}
\exp\Biggl(\frac{(2-\nu)C_2}{1-\nu}
\lambda_{\mathrm{over},0}
\sum_{j=i}^{t-1}\biggl(1-\frac{1-c\theta}{2\eta_0}\biggr)^j\Biggr)
&\leq {\rm e}^{1/5}.
\end{align}
Recall that $\eta_0\geq1$ and $\nu\leq(1-c\theta)/(2\eta_0)$, then we have
\begin{align*}
c\theta=1-(1-c\theta)\leq1 - \frac{1-c\theta}{2\eta_0} \leq 1-\nu,
\end{align*}
which implies
\begin{align}
\label{eq:sumweightedbound}
\sum_{i=0}^{t-1}c\theta(1-\nu)^{t-1-i}
\biggl(1-\frac{1-c\theta}{2\eta_0}\biggr)^i
&\leq\sum_{i=0}^{t-1}(1-\nu)^{t-i}(1-\nu)^i
=t(1-\nu)^t.
\end{align}
Combing above reulst, we achive
\begin{align}
\mathbb E[\xi_t]
&\overset{\eqref{eq:finialxi_t_recur},\, \eqref{eq:exponential_bound}}{\leq} {\rm e}^{1/5}(1-\nu)^t\xi_0
 +{\rm e}^{1/5}c\theta C_2\lambda_{\mathrm{over},0}
 \sum_{i=0}^{t-1}(1-\nu)^{t-1-i}
 \biggl(1-\frac{1-c\theta}{2\eta_0}\biggr)^i\notag\\
&~~~~\,=\,\,{\rm e}^{1/5}\Biggl((\delta_0+C_2\lambda_{\mathrm{over},0})(1-\nu)^t
+C_2\lambda_{\mathrm{over},0}\sum_{i=0}^{t-1}c\theta(1-\nu)^{t-1-i}
\biggl(1-\frac{1-c\theta}{2\eta_0}\biggr)^i\Biggr)\notag\\
&\,\,\,~~\overset{\eqref{eq:sumweightedbound}}{\leq} {\rm e}^{1/5}
\bigl(\delta_0+(t+1)C_2\lambda_{\mathrm{over},0}\bigr)(1-\nu)^t\notag\\
&~~~~\overset{\eqref{eq:inicondi_final_n>d}}{\leq}
{\rm e}^{1/5}\left(\delta_0+
\frac{(1-\gamma)(1-c\theta)(t+1)}{10(2-\gamma)\eta_0}\right)
(1-\nu)^t\notag\\
\label{eq:xi_decay_over}
&~~\,\,\,\,\,\leq\,\,q_{\text{\rm over}}(t+1)(1-\nu)^t.
\end{align}
The relation of $\lambda_{\mathrm{over},t+1}$ and $\lambda_{\mathrm{over},t}$ in~\eqref{eq:overfit_lambda_recur} gives
\begin{align}
\label{eq:ratio_one_step_over}
\frac{\lambda_{\mathrm{over},t+1}}
{\lambda_{\mathrm{over},t}}
\leq c\theta+\xi_t.
\end{align}
Note that $0\leq\eta_t-1\leq\delta_t\leq\xi_t$, then taking expectation and using~\eqref{eq:xi_decay_over} give
\begin{align*}
\EBP{\eta_t-1}
\leq\EBP{\xi_t}
\overset{\eqref{eq:xi_decay_over}}{\leq}
q_{\text{\rm over}}(t+1)(1-\nu)^t\\
\end{align*}
and
\begin{align*}
\EBP{\frac{\lambda_{\mathrm{over},t+1}}{\lambda_{\mathrm{over},t}}}
\overset{\eqref{eq:ratio_one_step_over}}{\leq}
c\theta+\EBP{\xi_t}
\overset{\eqref{eq:xi_decay_over}}{\leq}
c\theta+q_{\text{\rm over}}(t+1)(1-\nu)^t,
\end{align*}
which proves~equation \eqref{eq:eta-ratio-over}.

It remains to prove the high-probability result.
We define the event $\fE_j$ as the inequality
\begin{align*}
    \xi_j>\frac{4q_{\text{\rm over}}}{p\nu^2}
\left(1-\frac{\nu}{2}\right)^j
\end{align*}
holds. Then we have
\begin{align*}
\mathbb P\Bigg(\bigcup_{j=0}^{+\infty}\fE_j\Bigg)
& \leq \sum_{j=0}^{\infty}\mathbb P\left(\fE_j\right) = \sum_{j=0}^{+\infty}\mathbb P\left(\xi_j>\frac{4q_{\text{\rm over}}}{p\nu^2}
\left(1-\frac{\nu}{2}\right)^j\right) \\
& \leq
\frac{p\nu^2}{4q_{\text{\rm over}}}
\sum_{j=0}^{+\infty}\frac{\mathbb E[\xi_j]}{(1-\nu/2)^j}\overset{\eqref{eq:xi_decay_over}}{\leq}
\frac{p\nu^2}{4}\sum_{j=0}^{\infty}(j+1)
\left(\frac{1-\nu}{1-\nu/2}\right)^j\\
& \leq\frac{p\nu^2}{4}\sum_{j=0}^{+\infty}(j+1)
\left(1-\frac{\nu}{2}\right)^j=p,
\end{align*}
where the first inequality is based on union bound;
the second inequality is based on Markov's inequality; 
the last inequality follows $1-\nu\leq(1-\nu/2)^2$;
and the equality uses the result~$\sum_{j=0}^{\infty}(j+1)z^j=1/(1-z)^2$ with $z=1-\nu/2$.
This implies
\begin{align*}
\mathbb P\left(~\bigcap_{j=0}^{+\infty}\overline{\fE_j}~\right) 
=\mathbb P\left(~\overline{\bigcup_{j=0}^{+\infty}\fE_j}~\right) \geq 1-p.
\end{align*}
Therefore, with probability at least $1-p$, the inequality
\begin{align}
\label{eq:xi_j_union_bound}
\xi_j
&\leq\frac{4q_{\text{\rm over}}}{p\nu^2}
\left(1-\frac{\nu}{2}\right)^j
\end{align}
holds for all $j\geq0$.
We take
\begin{align*}
t_0
=\Biggl\lceil
\frac{2}{\nu}
\log\Bigl(1+\frac{8q_{\text{\rm over}}}{p(1-c\theta)\nu^2}\Bigr)
\Biggr\rceil
=\OM\Biggl(\frac{\log\bigl(1+q_{\text{\rm over}}/(p(1-c\theta)\nu^2)\bigr)}{\nu}\Biggr),
\end{align*}
then the fact $1-\nu/2\leq{\rm e}^{-\nu/2}$ and the definition of $q_{\rm over}$ leads to
\begin{align}
\label{eq:qover_phase_bound}
\frac{4q_{\text{\rm over}}}{p\nu^2}
\left(1-\frac{\nu}{2}\right)^{t_0}
\leq\frac{4q_{\text{\rm over}}}{p\nu^2}
\exp\left(-\frac{\nu t_0}{2}\right)
\leq\frac{4q_{\text{\rm over}}/(p\nu^2)}
{1+8q_{\text{\rm over}}/(p(1-c\theta)\nu^2)}
\leq\frac{1-c\theta}{2}.
\end{align}
Therefore, with probability at least $1-p$, the inequality
\begin{align}\label{eq:ratio_one_step_over_bound}
\begin{split}    
\frac{\lambda_{\mathrm{over},t_0+i+1}}
{\lambda_{\mathrm{over},t_0+i}}
& \overset{\eqref{eq:ratio_one_step_over}}{\leq}
c\theta + \xi_{t_0+i}
\overset{\eqref{eq:xi_j_union_bound}}{\leq}
c\theta + \frac{4q_{\text{\rm over}}}{p\nu^2}
\left(1-\frac{\nu}{2}\right)^{t_0+i} \\
& \overset{\eqref{eq:qover_phase_bound}}{\leq}
c\theta +  \frac{1-c\theta}{2}\left(1-\frac{\nu}{2}\right)^i
\leq c\theta +  \frac{1-c\theta}{2} = \frac{1+c\theta}{2}
\end{split}
\end{align}
holds for all $i\geq 0$.
This implies for all $t\geq1$, the upper bound
\begin{align*}
\lambda_{\mathrm{over},t_0+t}
&=\lambda_{\mathrm{over},t_0}
\prod_{i=0}^{t-1}
\frac{\lambda_{\mathrm{over},t_0+i+1}}
{\lambda_{\mathrm{over},t_0+i}}\\
&\!\!\overset{\eqref{eq:ratio_one_step_over_bound}}{\leq}\lambda_{\mathrm{over},t_0}
\prod_{i=0}^{t-1}\frac{1+c\theta}{2}\\
&=\biggl(\frac{1+c\theta}{2}\biggr)^t
\lambda_{\mathrm{over},t_0}\\
&\!\!\overset{\eqref{eq:linear_path_over}}{\leq}
\biggl(\frac{1+c\theta}{2}\biggr)^t
\biggl(1-\frac{1-c\theta}{2\eta_0}\biggr)^{t_0}
\lambda_{\mathrm{over},0}.
\end{align*}
holds with probability at least $1-p$, which proves equation \eqref{eq:n>dtwostage}.
\end{proof}

\subsection{Proof of Corollary~\ref{col:theta=0n>d}}
\label{app:proof_coro_n>d}
\begin{proof}
We follow notations 
$\lambda_{\mathrm{over},t}=\lambda_{\mathrm{over}}(\x_t)$ and
$\xi_t=\delta_t+C_2\lambda_{\mathrm{over},t}$ defined in Appendix~\ref{app:proof_theorem_two_stage_over}.
Note that equation~\eqref{eq:linear_path_over} with $\theta=0$ gives
\begin{align}
\label{eq:linear_path_theta_zero}
\lambda_{\mathrm{over},t}
\leq\biggl(1-\frac{1}{2\eta_0}\biggr)^t
\lambda_{\mathrm{over},0}
\qquad\text{and}\qquad
\|\x_t-\x_*\|\leq\frac{(c-1)\mu_*^2}{4cLL_2}.
\end{align}
Additionally, equation \eqref{eq:xi_one_step_over_raw} with $\theta=0$  gives
\begin{align}
\label{eq:xi_one_step_theta_zero}
\begin{split}    
\mathbb E_t[\xi_{t+1}]
&\leq\left((1-\gamma)
\left(1+\frac{C_2\lambda_{\mathrm{over},t}}{2}\right)^2
+C_2\lambda_{\mathrm{over},t}\right)\xi_t \\
&\leq (1-\gamma)\left(1+\frac{C_2\lambda_{\mathrm{over},t}}{2}\right)^2\left(1+\frac{C_2\lambda_{\mathrm{over},t}}{1-\gamma}\right)\xi_t\\
&\leq(1-\gamma)
\exp\left(\frac{(2-\gamma)C_2\lambda_{\mathrm{over},t}}{1-\gamma}
\right)\xi_t,
\end{split}
\end{align}
where we use inequality $1+x\leq {\rm e}^x$ with $x=C_2\lambda_{\mathrm{over},t}/2$ and $x=C_2\lambda_{\mathrm{over},t}/(1-\gamma)$ in the last step.
Setting $\theta=0$ in~\eqref{eq:overfit_lambda_recur} gives
\begin{align}
\label{eq:lambda_one_step_theta_zero}
\lambda_{\mathrm{over},t+1}
&\leq\xi_t\lambda_{\mathrm{over},t}.
\end{align}
Taking expectation on~\eqref{eq:xi_one_step_theta_zero} and using~\eqref{eq:linear_path_theta_zero}, we have
\begin{align}\label{eq:xi_decay_theta_zero}
\begin{split}    
\EBP{\xi_t}
&\!\!\overset{\eqref{eq:xi_one_step_theta_zero}}{\leq} (1-\gamma)\exp\Biggl(\frac{(2-\gamma)C_2}{1-\gamma}\lambda_{\mathrm{over},t-1}\Biggr)\EBP{\xi_{t-1}}\\
&\!\!\overset{\eqref{eq:xi_one_step_theta_zero}}{\leq} (1-\gamma)^t \exp\Biggl(\frac{(2-\gamma)C_2}{1-\gamma}\sum_{i=0}^{t-1}\lambda_{\mathrm{over},i}\Biggr)\xi_0\\
&\!\!\overset{ \eqref{eq:linear_path_theta_zero}}{\leq}
(1-\gamma)^t\exp\Biggl(\frac{(2-\gamma)C_2}{1-\gamma}
\sum_{i=0}^{t-1}\biggl(1-\frac{1}{2\eta_0}\biggr)^i
\lambda_{\mathrm{over},0}\Biggr)
(\delta_0+C_2\lambda_{\mathrm{over},0})\\
&\,\leq\,\,(1-\gamma)^t
\exp\Biggl(\frac{2(2-\gamma)C_2\eta_0\lambda_{\mathrm{over},0}}{1-\gamma}\Biggl)
(\delta_0+C_2\lambda_{\mathrm{over},0})\\
&\overset{\eqref{eq:inicondi_final_n>d}}{\leq}
(1-\gamma)^tq_{\text{\rm over}}.
\end{split}
\end{align}
Hence, we apply Lemma~\ref{lm:lin_lemma26} with
$X_t=\xi_t$, $a=q_{\text{\rm over}}$, and $\tau=1/\gamma>1$, which implies with probability at least $1-p$, the inequality
\begin{align}
\label{eq:xi_high_probability_theta_zero}
\xi_t
&\leq\biggl(1-\frac{\gamma}{1+\gamma}\biggr)^t
\frac{q_{\text{\rm over}}}{p\gamma^2}
\end{align}
holds for all $t\geq0$.
This implies
\begin{align}
\label{eq:lambda_t+1-recur_over}
\lambda_{\mathrm{over},t+1}\overset{\eqref{eq:lambda_one_step_theta_zero}}{\leq} \xi_t\lambda_{\mathrm{over},t}
\overset{
\eqref{eq:xi_high_probability_theta_zero}}{\leq}
\biggl(1-\frac{\gamma}{1+\gamma}\biggr)^t
\frac{q_{\text{\rm over}}\lambda_{\mathrm{over},t}}{p\gamma^2}.
\end{align}
Taking
\begin{align*}
t_0
=\biggl\lceil
\frac{\log\bigl(1+2q_{\text{\rm over}}/(p\gamma^2)\bigr)}{\log(1+\gamma)}
\biggr\rceil=\OM\biggl(\frac{\log\bigl(1+q_{\text{\rm over}}/(p\gamma^2)\bigr)}{\gamma}\biggr)
\end{align*}
gives 
\begin{align}
\label{eq:superlinearleq1/2_over}
\biggl(1-\frac{\gamma}{1+\gamma}\biggr)^{t_0}
\frac{q_{\text{\rm over}}}{p\gamma^2}
&\leq\frac{1}{2}.
\end{align}
Then for all $i\geq0$, it holds
\begin{align*}
\lambda_{\mathrm{over},t_0+i+1}
&\overset{\eqref{eq:lambda_t+1-recur_over}}{\leq}
\biggl(1-\frac{\gamma}{1+\gamma}\biggr)^{t_0+i}
\frac{q_{\text{\rm over}}\lambda_{\mathrm{over},t_0+i}}{p\gamma^2}
\\
&~\,=\,\,\,\biggl(1-\frac{\gamma}{1+\gamma}\biggr)^i
\biggl(\biggl(1-\frac{\gamma}{1+\gamma}\biggr)^{t_0}
\frac{q_{\text{\rm over}}}{p\gamma^2}\biggr)
\lambda_{\mathrm{over},t_0+i}\\
&\overset{\eqref{eq:superlinearleq1/2_over}}{\leq}
\frac{1}{2}\biggl(1-\frac{\gamma}{1+\gamma}\biggr)^i
\lambda_{\mathrm{over},t_0+i}
\end{align*}
with probability at least $1-p$. 
Consequently, for all $t\geq1$, the upper bound
\begin{align*}
\lambda_{\mathrm{over},t_0+t}
&=\lambda_{\mathrm{over},t_0}
\prod_{i=0}^{t-1}
\frac{\lambda_{\mathrm{over},t_0+i+1}}
{\lambda_{\mathrm{over},t_0+i}}\\
&\leq\lambda_{\mathrm{over},t_0}
\prod_{i=0}^{t-1}\left(\frac{1}{2}
\left(1-\frac{\gamma}{1+\gamma}\right)^i\right)\\
&=\left(1-\frac{\gamma}{1+\gamma}\right)^{t(t-1)/2}
\biggl(\frac{1}{2}\biggr)^t\lambda_{\mathrm{over},t_0}\\
&\!\!\!\overset{\eqref{eq:linear_path_theta_zero}}{\leq}\!\!
\biggl(1-\frac{\gamma}{1+\gamma}\biggr)^{t(t-1)/2}
\biggl(\frac{1}{2}\biggr)^t
\biggl(1-\frac{1}{2\eta_0}\biggr)^{t_0}
\lambda_{\mathrm{over},0}
\end{align*}
holds with probability at least $1-p$.
\end{proof}

\section{Implementation of RQGN with F-BFGS}
\label{app:fbfgs}

We give the detailed implementation of RQGN with the fast BFGS update (Option II) in Algorithm~\ref{alg:fbfgs}, 
which can be regarded as Algorithm~\ref{alg:RQGN} with $\G_t=(\L_t^{\top}\L_t)^{-1}$. In particular, the fast BFGS update satisfies
\begin{align*}
\G_{t+1}
&=\fbfgs(\widetilde{\G}_t,\H(\x_{t+1}),\u_t) \\
&=\bfgs(\widetilde{\G}_t,\H(\x_{t+1}),\widetilde{\L}_t^{\top}\u_t),
\end{align*}
where $\widetilde{\G}_t=\big(\widetilde{\L}_t^{\top}\widetilde{\L}_t\big)^{-1}$.
Following \citet[Proposition~5.7]{liu2023symmetric} or \citet[Equation~(9.11)]{gower2017randomized}, the update on $\mG_{t+1}$ can be implemented by the closed-form iteration 
\begin{align*}
\L_{t+1}={}&\widetilde{\L}_t+
\big(\u_t(\u_t^\top\u_t)^{-1/2}
-\widetilde{\L}_t\H(\x_{t+1})\widetilde{\L}_t^\top\u_t
(\u_t^\top\widetilde{\L}_t\H(\x_{t+1})\widetilde{\L}_t^\top\u_t)^{-1/2}\big)\\
&~\quad\quad\times
\big(\u_t^\top\widetilde{\L}_t\H(\x_{t+1})\widetilde{\L}_t^\top\u_t\big)^{-1/2}
\u_t^\top\widetilde{\L}_t,
\end{align*}
which requires the cost of $\OM(m^2)$ flops per iteration.

\begin{algorithm}[H]
\caption{RQGN with the fast BFGS update}
\label{alg:fbfgs}
\begin{algorithmic}[1]
\STATE \textbf{Input:} $\x_0\in\Omega$, $\L_0\in\RB^{m\times m}$ such that $(\L_0^{\top}\L_0)^{-1}\succeq\H(\x_0)$, and $M\geq0$ \\[0.15cm]
\STATE \textbf{for} $t=0,1,\ldots$ \\[0.15cm]
\STATE \quad $\displaystyle
\x_{t+1}=\begin{cases}
\x_t-\J(\x_t)^{\top}\L_t^{\top}\L_t\F(\x_t),& n<d,\\
\x_t-\L_t^{\top}\L_t\J(\x_t)^{\top}\F(\x_t),& n\geq d;
\end{cases}$ \\[0.15cm]
\STATE \quad $r_t=\|\x_{t+1}-\x_t\|$ and $\widetilde{\L}_t=\L_t/\sqrt{1+Mr_t}$ \\[0.15cm]
\STATE \quad $[\u_t]_i\overset{\mathrm{i.i.d.}}{\sim}\mathcal N(0,1)$ \\[0.15cm]
\STATE \quad $\displaystyle
\begin{aligned}
\L_{t+1}={}&\widetilde{\L}_t+
\big(\u_t(\u_t^\top\u_t)^{-1/2}
-\widetilde{\L}_t\H(\x_{t+1})\widetilde{\L}_t^\top\u_t
(\u_t^\top\widetilde{\L}_t\H(\x_{t+1})\widetilde{\L}_t^\top\u_t)^{-1/2}\big)\\
&~\quad\quad\times
\big(\u_t^\top\widetilde{\L}_t\H(\x_{t+1})\widetilde{\L}_t^\top\u_t\big)^{-1/2}
\u_t^\top\widetilde{\L}_t
\end{aligned}$
\STATE \textbf{end for}
\end{algorithmic}
\end{algorithm}

\FloatBarrier

\bibliographystyle{plainnat}
\bibliography{ref}
\end{document}